\documentclass{amsart}

\title{Norms and Transfers in Motivic Homotopy Theory}
\author{Brian Shin}
\date{2024 October 26}

\usepackage{preamble}
\usepackage{macros}

\begin{document}

\begin{abstract}
    In this article, we establish the compatibility between norms and transfers in motivic homotopy theory.
    More precisely, we construct norm functors for motivic spaces equipped with various flavours of transfer.
    This yields a norm monoidal refinement of the infinite $\mathbb{P}^1$-delooping machine of Elmanto-Hoyois-Khan-Sosnilo-Yakerson.
    We apply this refinement to construct a normed algebra structure on the very effective Hermitian $\mathrm{K}$-theory spectrum.
\end{abstract}

\maketitle

\tableofcontents

\section{Introduction}

The goal of this article is to study the interactions between multiplicative norms and additive transfers in motivic homotopy theory.
Specifically, we develop a theory of norm functors for motivic spaces equipped with various flavors of transfers.
We use this to upgrade the Motivic Recognition Principle of Elmanto-Hoyois-Khan-Sosnilo-Yakerson to a norm monoidal equivalence.
We also use this to promote the very effective Herimitian $\mathrm{K}$-theory spectrum to a normed algebra.

\subsection{Background}

\subsubsection{Motivic Homotopy Theory}

Motivic homotopy theory is the study of algberaic geometry through the lens of homotopy theory.
The basic objects of interest are \emph{motivic spaces}, \emph{i.e.} Nisnevich sheaves $F : \Sm_S^\op \to \Spc$ over some base scheme $S$ such that $F(X \times \mathbb{A}^1) \cong F(X)$ for every smooth $S$-scheme $X$.
We write $\HH(S)$ for the $\infty$-category of motivic spaces over $S$.
This $\infty$-category provides a convenient setting where we can apply homotopy-theoretic tools to algebro-geometric objects.
Abstractly, $\HH(S)$ is the presentable $\infty$-category generated by smooth $S$-schemes and subject to the relations of Nisnevich descent and $\mathbb{A}^1$-invariance.
The definition of $\HH(S)$ is justified in part by the following observation: if we replace smooth $S$-schemes by manifolds, the Nisnevich topology with the usual open cover topology, and $\mathbb{A}^1$ with $\mathbb{R}$, then we recover the $\infty$-category of spaces.

As in ordinary homotopy theory, one is often interested in studying generalized cohomology theories for motivic spaces.
These can be packaged into objects called \emph{motivic spectra}.
Explicitly, a motivic spectrum consists of a sequence of pointed motivic spaces $X_0, X_1, X_2, \ldots$ together with equivalences $X_i \cong \Omega_{\mathbb{P}^1} X_{i+1}$.
Here we write $\Omega_{\mathbb{P}^1} X = \mathrm{Hom}_\bullet(\mathbb{P}^1, X)$ for the motivic space of based $\mathbb{P}^1$-loops in $X$.
On a categorical level, motivic spectra are obtained by taking the $\infty$-category of pointed motivic spaces and formally inverting $\mathbb{P}^1$ under smash product.
Everything takes place over a base scheme $S$, and we write $\SH(S)$ for the $\infty$-category of motivic spectra over $S$.
Many of the usual spectra from ordinary homotopy theory admit motivic analogues.
For example, there are spectra $\mathrm{H}\mathbb{Z}_S$, $\mathrm{KGL}_S$, and $\mathrm{MGL}_S$ which are motivic analogues of the ordinary spectra $\mathrm{H}\mathbb{Z}$, $\mathrm{KU}$, and $\mathrm{MU}$.

\subsubsection{Framed Transfers and the Recognition Principle}

Passing from ordinary to motivic homotopy theory introduces new and interesting complexity to various classical constructions.

A fundamental result in ordinary homotopy theory is the \emph{Recognition Principle} for infinite loop spaces (\cite{BV_HomotopyEverything}, \cite{Segal_CatsCoh}, \cite{May_GeometryIteratedLoops}).
This establishes a tight connection between spectra and $\mathbb{E}_\infty$ monoids of spaces via the infinite loop space functor.
Specifically, it says that $\Omega^\infty : \Spt \to \Spc_\bullet$ refines to an equivalence between connective spectra and grouplike $\mathbb{E}_\infty$ monoids.
In other words, we may view spectra as being some sort of homotopical generalization of abelian groups.

Building on an insight of Voevodsky (\cite{Voevodsky_FramedNotes}) and work of Garkusha-Panin and their collaborators (\cite{GP_FramedMotivesAfterVoevodsky}, \cite{AGP_cancellation}, \cite{GP_HISheavesFramedTransfers}, \cite{DP_Surjectivity}, \cite{Druzhinin_Framed}, \cite{GNP_FramedMotiveRelativeSphere}), the authors of \cite{EHKSY_MotivicInfiniteLoops} prove a motivic analogue of the above recognition principle.
They construct an $\infty$-category $\HH^\fr(S)$ of motivic spaces with \emph{framed transfers} and show that it plays the role of $\mathbb{E}_\infty$ monoids of spaces.
More precisely, they show that the diagram
\begin{equation*}
    \begin{tikzcd}
        & \HH^\fr(S) \ar[d,"forget"] \\
        \SH(S) \ar[r,"\Omega^\infty_{\mathbb{P}^1}"] \ar[ur,dashed] & \HH(S)_\bullet
    \end{tikzcd}
\end{equation*}
admits a lift, this lift admits a left adjoint $\Sigma^\infty_{\mathbb{P}^1,\fr} : \HH^\fr(S) \to \SH(S)$, and (when $S$ is the spectrum of a perfect field) this adjunction restricts to an equivalence
\begin{equation*}
    \SH(S)^\veff \cong \HH^\fr(S)^\gp.
\end{equation*}
The $\infty$-category on the left side of this equivalence consists of \emph{very effective} motivic spectra, which are a motivic analogue of connective spectra in ordinary homotopy theory.

\subsubsection{Norm Functors and Normed Algebras}

In a separate vein, Voevodsky introduced norm functors in motivic homotopy theory in order to define symmetric powers of motives (\cite{Deligne_Voevodsky}).
In \cite{BachmannHoyois_Norms}, Bachmann-Hoyois push the theory further, showing that norm functors taken together can be viewed as a motivic enhancement of symmetric monoidal structure.
For a finite \'etale map $p : S \to T$, the norm along $p$ is a certain functor $p_\otimes : \SH(S) \to \SH(T)$ which behaves like a ``twisted $n$-ary smash product'' functor
When $p : T^{\sqcup n} \to T$ is a fold map, we can identify $p_\otimes$ with the $n$-aray smash product functor under he equivalence $\SH(T^{\sqcup n}) \cong \SH(T)^{\times n}$.
Taken together, these norm functors give $\SH$ the structure of a \emph{norm monoidal} $\infty$-category.

Once one has a notion of norm monoidal $\infty$-category, one can make a definition of \emph{normed algebra} within such a structure.
Bachmann-Hoyois show that normed algebras in $\SH$ come equipped with a rich theory of motivic power operations.
For $(\mathrm{H}\mathbb{F}_p)_S$ recover Voevodsky's motivic Steenrod operations.
These played a critical role in Rost and Voevodsky's proof of the Bloch-Kato conjecture.

\subsection{The Present Work}

\subsubsection{Multiplicative Motivic Infinite Loop Space Theory}

In this article, we establish a compatibility between norm functors and framed transfers.

\begin{theorem}[see Theorem~\ref{theorem:THE-CONSEQUENCES}]
    The association $S \mapsto \HH^\fr(S)$ can be promoted to a norm monoidal $\infty$-category.
    The functors $\HH(S)_\bullet \to \HH^\fr(S)$ can be assembled into a norm monoidal functor.
\end{theorem}

This lets us prove the following recognition principle for normed infinite $\mathbb{P}^1$-loop spaces.

\begin{theorem}[see Theorem~\ref{theorem:multiplicative-recognition}]
    Let $k$ be a perfect field.
    There is an equivalence
    \begin{equation*}
        \NAlg(\HH^{\fr,\gp} | k) \simeq \NAlg(\SH^\mathrm{veff} | k)
    \end{equation*}
    of $\infty$-categories of normed algebras over $\Spec k$.
\end{theorem}

In order to carry out the construction of $S \mapsto \HH^\fr(S)$ as a norm monoidal $\infty$-category, we extend the technology of labeling functors developed in \cite{EHKSY_MotivicInfiniteLoops}.
More specifically, we make suitable modifications to the technology to be able to handle norms, and then follow the basic steps in the original construction of $\HH^\fr(S)$.

\subsubsection{Normed Orientations}

As an application of the theory of norm functors for motivic spaces with framed transfers, we prove that the very effective $\mathrm{K}$-theory spectra $\mathrm{kgl}$ and $\mathrm{ko}$ admit normed orientations.

\begin{theorem}[See Theorem~\ref{theorem:normed-orientations}]
    Let $S$ be a scheme essentially smooth over a Dedekind scheme with $2 \in \mathcal{O}_S^\times$.
    The commutative diagram
    \begin{equation*}
        \begin{tikzcd}
            \mathrm{MSL}_S \ar[d] \ar[r] & \mathrm{MGL}_S \ar[d] \\
            \mathrm{ko}_S \ar[r] & \mathrm{kgl}_S
        \end{tikzcd}
    \end{equation*}
    can be promoted to one in $\NAlg(\SH|S)$.
\end{theorem}

In particular, $\mathrm{ko}$ admits a normed algebra structre, which is new.
Recall that normed algebra structures for $\mathrm{kgl}$, $\mathrm{MGL}$, and $\mathrm{MSL}$ were constructed already in \cite{BachmannHoyois_Norms}.
In this work, we give a new construction for normed algebra structures for these.
In future work, we will return to the question of comparing the norm structures.

The essence of our proof lies in the fact that each of the motivic spectra arising in Theorem~\ref{theorem:normed-orientations} can be constructed by applying the infinite $\mathbb{P}^1$-delooping functor $\Sigma^\infty_{\mathbb{P}^1,\fr} : \HH^\fr(S) \to \SH(S)$ to an extremely concrete geometric object.
For example, $\mathrm{MGL}$ comes from the moduli space of finite syntomic schemes, while $\mathrm{ko}$ comes from the moduli space of oriented finite Gorenstein schemes.
See \cite{EHKSY_ModulesOverMGL}, \cite{HJNTY_HilbertSchemes}, and \cite{HJNY_Hermitian}.
Access to these descriptions allows us to construct the relevant normed algebra structures using descent.

It's worth pointing out that none of the assumptions in Theorem~\ref{theorem:normed-orientations} are relevant for the techniques developed in the present article.
In fact, forthcoming work of Bachmann, Elmanto, and Morrow will show that $S$ can be any scheme for which $2$ is a unit.
See Remark~\ref{remark:assumptions-are-true}.

\subsubsection{Normed Algebras over Cobordism and \texorpdfstring{$\mathrm{K}$}{K}-Theory Spectra}

The authors of \cite{EHKSY_ModulesOverMGL}, \cite{HJNTY_HilbertSchemes}, and \cite{HJNY_Hermitian} study motivic spaces and spectra equipped with transfers more general than framed ones
In Theorem~\ref{theorem:THE-CONSEQUENCES}, we show that these all admit norm functors that are compatible with those of motivic spaces with framed transfers.
We use these norm monoidal structures to prove multiplicative recognition theorems for modules over various motivic spectra.

\begin{theorem}[see Theorem~\ref{theorem:normed-algebras-over-cobordism}]
    Let $k$ be a perfect field.
    There is an equivalence
    \begin{equation*}
        \NAlg_{\mathrm{MGL}}(\SH^\mathrm{veff}|k) \cong \NAlg(\HH^{\fsyn,\gp}|k)
    \end{equation*}
    of $\infty$-categories of normed algebras over $\Spec k$.
\end{theorem}

We obtain similar theorems for normed algebras over $\mathrm{MSL}$, $\mathrm{kgl}$, and $\mathrm{ko}$.
See Theorems~\ref{theorem:normed-algebras-over-cobordism} and \ref{theorem:normed-algebras-over-k-theory}.

\subsection{Outline}

We begin in Section~2 with some preliminaries.
We review the connection between flagged $\infty$-categories and Segal spaces, and use this to study descent for presheaves of flagged $\infty$-categories.
We also record some facts about the interaction between polynomial functors and group completion.

In Section~3, we extend the developing body of work on the theory of norm monoidal $\infty$-categories.
In particular, we highlight the role played by lax norm monoidal functors.
We warn the reader familiar with \cite{BachmannHoyois_Norms} that we have made several small modifications to the language.
See the beginning of Section~3 for more details.

Sections 4 and 5 form the technical heart of this article.
They both begin with detailed discussions of their contents, so we'll be brief here.
To streamline the discussion around various $\infty$-categories that are built by ``replacing $\Sm_S$ by something else''  in the construction of $\mathscr{SH}(S)$, we introduce the notion of a \emph{motivic pattern} in Section~4.
The rough idea is to capture the features necessary in order to ``transport motivic homotopy theory'' along a functor $\Sm_S \to \mathscr{C}$.
We put particular emphasis on the structure needed to construct norm functors.
Section~5 is devoted to construction of various motivic patterns equipped with norm functors.

Finally, we reap the fruit of our labor in Section~6.

\subsection{Notation and Conventions}

Throughout this article, we will work in the setting of quasi-compact quasi-separated algebraic spaces.
We'll write $\AlgSpc$ for the category of all such algebraic spaces.
The extension of motivic homotopy theory to algebraic spaces comes essentially for free since every algebraic space is Nisnevich-locally affine \cite[Chapter II Theorem 6.4]{Knutson_AlgSpc}.
For any $S \in \AlgSpc$, we'll write
\begin{equation*}
    \Sm_S \subseteq \AlgSpc_S = \AlgSpc_{/S}
\end{equation*}
for the full subcategory spanned by smooth algebraic spaces over $S$.

We'll use the language of $\infty$-categories.
We'll write $\Cat_\infty$ for the (large) $\infty$-category of $\infty$-categories, and $\Spc \subseteq \Cat_\infty$ for the full subcategory spanned by spaces.
We'll write $\mathbf{\Delta}$ to denote the ordinary category of finite nonempty linearly ordered sets, which we regard as a full subcategory of $\Cat_\infty$.
We'll write $\Delta^n$ or $[n]$ for the linearly ordered set $\{ 0 \leq \cdots \leq n \}$.

We will write $\mathscr{O}\Cat_\infty$ and $\mathscr{M}\Cat_\infty$ to denote the $\infty$-categories defined by the fiber products
\begin{equation*}
    \begin{tikzcd}
        \mathscr{M}\Cat_\infty \ar[d] \ar[r] \arrow[dr,phantom,"\lrcorner", very near start] & \mathscr{O}\Cat_\infty \ar[d] \ar[r] \arrow[dr,phantom,"\lrcorner", very near start] & \mathscr{E} \ar[d] \\
        \Cat_\infty \ar[r,"\mathrm{Fun}(\Delta^1{,}{-})"] & \Cat_\infty \ar[r] & \mathscr{P}\mathrm{os}
    \end{tikzcd}
\end{equation*}
where $\mathscr{P}\mathrm{os}$ is the ordinary category of partially ordered sets, $\mathscr{E} \to \mathscr{P}\mathrm{os}$ denotes the universal coCartesian fibration in partially ordered sets, and the functor $\Cat_\infty \to \mathscr{P}\mathrm{os}$ sends an $\infty$-category $C$ to the power set of the set of equivalences classes of objects in $C$.
For any functor $F : C \to \Cat_\infty$, we'll also write $\mathscr{O}C$ and $\mathscr{M}C$ to denote the fiber products of $C$ with $\mathscr{O}\Cat_\infty$ and $\mathscr{M}\Cat_\infty$, respectively.

For an $\infty$-category $C$ with a final object $* \in C$, we'll write $C_\bullet = C_{*/}$ for the $\infty$-category of pointed objects.
If $C$ moreover admits finite coproducts then this comes with a functor $({-})_+ : C \to C_\bullet$ given by adding a disjoint basepoint.
In this case, we'll write $C_+ \subseteq C_\bullet$ for the full subcategory spanned by objects of the form $U_+$ for $U \in C$.

For a functor $C \to D$, we'll try to adhere to the following conventions.
\begin{enumerate}
    \item Given a functor $f : C \to D$, we'll write $f^* : \PSh(D) \to \PSh(C)$ for the functor given by precomposition with $f$, and we'll write $f_!$ and $f_*$ for the left and right adjoint to $f^*$, respectively.

    \item If the name of $C \to D$ comes decorated with a symbol, then the left Kan extension $\PSh(C) \to \PSh(D)$ will be given the same name.
    For example, given a functor $f_\diamond : C \to D$, we'll write $f_\diamond : \PSh(C) \to \PSh(D)$ for the left Kan extension.
\end{enumerate}
We will be explicit about notation if there is any potential ambiguity.

\subsection{Acknowledgements}
This article is a distillation of the my doctoral thesis.
It is a pleasure to thank my advisor Jeremiah Heller for his guidance and support throughout this project.
I would like to thank Daniel Carmody and Tsutomu Okano for many helpful conversations.
I would also like to thank Elden Elmanto, Marc Hoyois, Denis Nardin, and Tom Bachmann for stimulating conversations, their patience with my many questions, and their continued interest in my work.

\section{Preliminaries}

\subsection{Flagged \texorpdfstring{$\infty$}{oo}-Categories as Segal Spaces}

\begin{definition}
    A \emph{flagged $\infty$-category} is an $\infty$-category $\mathscr{C}_1$ together with an essentially surjective functor
    \begin{equation*}
        \phi: \mathscr{C}_0 \to \mathscr{C}_1
    \end{equation*}
    where $\mathscr{C}_0$ is a space.
    We write $\fCat \subseteq \mathrm{Fun}(\Delta^1, \Cat_\infty)$ for the full subcategory spanned by flagged $\infty$-categories.
    We write $\mathrm{ev}_i : \fCat \to \Cat_\infty$ for the functor that sends $\phi : \mathscr{C}_0 \to \mathscr{C}_1$ to $\mathscr{C}_i$.
\end{definition}

\begin{proposition}
    The full subcategory $\fCat$ of $\mathrm{Fun}(\Delta^1, \Cat_\infty)$ is closed under finite products.
\end{proposition}

\begin{proof}
    This follows from the fact that spaces and essentially surjective functors are closed under finite products in $\Cat_\infty$.
\end{proof}

\begin{proposition}
    The functor $\mathrm{ev}_1 : \fCat \to \Cat_\infty$ admits a fully faithful right adjoint.
\end{proposition}

\begin{proof}
    For $\mathscr{D}$ an $\infty$-category, the functor
    \begin{equation*}
        \operatorname{Map}(\mathrm{ev}_1({-}), \mathscr{D}) : \fCat^\mathrm{op} \to \Spc
    \end{equation*}
    is represented by the inclusion $\mathscr{D}^\simeq \to \mathscr{D}$ of the maximal subspace of $\mathscr{D}$.
\end{proof}

\begin{definition}
    For $n \geq 1$, the \emph{$n$-spine} $\mathrm{Sp}^n \in \PSh(\mathbf{\Delta})$ is the object
    \begin{equation*}
        \Delta^{\{0,1\}} \sqcup_{\Delta^{\{1\}}} \cdots \sqcup_{\Delta^{\{n-1\}}} \Delta^{\{n-1,n\}}.
    \end{equation*}
    This comes with a map $\mathrm{Sp}^n \to \Delta^n$.
\end{definition}

\begin{notation}
    Let $J \in \PSh(\mathbf{\Delta})$ denote the composite
    \begin{equation*}
        \mathbf{\Delta}^\op \to \mathscr{S}\mathrm{et} \to \Spc
    \end{equation*}
    where the first functor is the nerve of the contractible $1$-groupoid with two objects.
\end{notation}

\begin{definition}
    A \emph{Segal space} is a functor $X_\bullet : \Delta^\op \to \Spc$ such that, for every $n \geq 1$, the morphism
    \begin{equation*}
        X_n = \mathrm{Map}_{\PSh(\mathbf{\Delta})}(\Delta^n, X_\bullet) \to \mathrm{Map}_{\PSh(\mathbf{\Delta})}(\mathrm{Sp}^n, X_\bullet)
    \end{equation*}
    is an equivalence.
    A \emph{complete Segal space} is a Segal space $X_\bullet : \Delta^\op \to \Spc$ such that the morphism
    \begin{equation*}
        X_0 = \mathrm{Map}_{\PSh(\mathbf{\Delta})}(\Delta^0, X_\bullet) \to \mathrm{Map}_{\PSh(\mathbf{\Delta})}(J, X_\bullet)
    \end{equation*}
    is an equivalence.

    We write $\SegalSpaces$ and $\CompleteSegalSpaces$ for the full subcategories of $\PSh(\mathbf{\Delta})$ spanned by Segal spaces and complete Segal spaces, respectively.
\end{definition}

\begin{proposition} \label{prop:segal-accessible}
    The inclusions $\SegalSpaces \subseteq \PSh(\mathbf{\Delta})$ and $\CompleteSegalSpaces \subseteq \SegalSpaces$ are accessible localizations.
\end{proposition}

\begin{proof}
    This follows from the fact that (complete) Segal spaces are presheaves that satisfy a certain small set of conditions.
\end{proof}

\begin{notation}
    We write $\mathrm{N}_\bullet : \fCat \to \PSh(\mathbf{\Delta})$ for the composite
    \begin{equation*}
        \fCat \to \PSh(\fCat) \to  \PSh(\mathbf{\Delta})
    \end{equation*}
    where the first functor is the Yoneda embedding, and the second functor is restriction along the composite of inclusions $\mathbf{\Delta} \to \Cat_\infty \to \fCat$.
\end{notation}

\begin{remark} \label{remark:flagged-mapping-space}
    Fix a flagged $\infty$-category $\mathscr{C} = (\phi : \mathscr{C}_0 \to \mathscr{C}_1)$.
    Note that we have an identification $\mathrm{N}_0 \mathscr{C} = \mathscr{C}_0$, and for every pair of objects $x, y \in \mathscr{C}_0$, we have a Cartesian square
    \begin{equation*}
        \begin{tikzcd}
            \operatorname{Map}_{\mathscr{C}_1}(\phi(x),\phi(y)) \ar[r] \ar[d] \arrow[dr,phantom,"\lrcorner", very near start] & \mathrm{N}_1 \mathscr{C} \ar[d,"(d_1{,}d_0)"] \\
            * \ar[r,"(x{,}y)"] & \mathrm{N}_0 \mathscr{C} \times \mathrm{N}_0 \mathscr{C}
        \end{tikzcd}
    \end{equation*}
    of spaces.
\end{remark}

\begin{proposition} \label{prop:flagged-segal}
    The functor $\mathrm{N}_\bullet : \fCat \to \PSh(\mathbf{\Delta})$ is fully faithful, and the essential image is $\SegalSpaces$.
    Moreover, $\mathrm{N}_\bullet$ restricts to an equivalence between the full subcategories $\Cat_\infty \subseteq \fCat$ and $\CompleteSegalSpaces \subseteq \SegalSpaces$.
\end{proposition}

\begin{proof}
    This is \cite[Theorem~0.26]{AyalaFrancis_Flagged}.
\end{proof}

\subsection{Descent for Flagged \texorpdfstring{$\infty$}{oo}-Categories}
\begin{notation}
    For this subsection, fix an $\infty$-category $\mathscr{C}$ with a Grothendieck topology $\tau$.
\end{notation}

\begin{proposition} \label{prop:descent-for-segal-spaces}
    A presheaf $X_\bullet : \mathscr{C}^\mathrm{op} \to \SegalSpaces$ of Segal spaces is a sheaf for the $\tau$-topology if and only if $X_0$ and $X_1$ are sheaves for the $\tau$-topology.
\end{proposition}

\begin{proof}
    Fix a presheaf $X_\bullet : \mathscr{C}^\mathrm{op} \to \SegalSpaces$.
    By Proposition~\ref{prop:segal-accessible}, the inclusion $\SegalSpaces \to \PSh(\mathbf{\Delta})$ preserves and reflects limits.
    In particular, the presheaf $X_\bullet$ is a sheaf if and only if $X_n : \mathscr{C}^\mathrm{op} \to \Spc$ is a sheaf for every $n \geq 0$.
    The claim now follows from the equivalence
    \begin{equation*}
        X_n = X_1 \times_{X_0} \cdots \times_{X_0} X_1
    \end{equation*}
    for $n \geq 2$.
\end{proof}

\begin{notation}
    We'll identify presheaves $\mathscr{F} : \mathscr{C}^\mathrm{op} \to \fCat$ of flagged $\infty$-categories on $\mathscr{C}$ as triples $(\mathscr{F}_0, \mathscr{F}_1, \phi)$, where $\mathscr{F}_i = \mathrm{ev}_i \circ \mathscr{F}$ and $\phi : \mathscr{F}_0 \to \mathscr{F}_1$ is the induced natural transformation.
\end{notation}

\begin{definition} \label{def:tau-local-mapping-spaces}
    We say a presheaf $\mathscr{F} = (\mathscr{F}_0,\mathscr{F}_1,\phi): \mathscr{C}^\mathrm{op} \to \fCat$ of flagged $\infty$-categories \emph{has $\tau$-local mapping spaces} if, for every object $c \in\mathscr{C}$, and every pair of objects $x, y \in \mathscr{F}_0(c)$, the presheaf
    \begin{equation*}
        \underline{\operatorname{Map}}_{\mathscr{F}_1}(\phi(x),\phi(y)) : (f : d \to c) \mapsto \mathrm{Map}_{\mathscr{F}_1(d)}(f^*\phi(x), f^*\phi(y))
    \end{equation*}
    on $\mathscr{C}_{/c}$ is a sheaf for the $\tau$-topology.
\end{definition}

\begin{remark} \label{remark:tau-local}
    Let $\mathscr{F} = (\mathscr{F}_0,\mathscr{F}_1,\phi)$ be a presheaf of flagged $\infty$-categories on $\mathscr{C}$.

    For every object $c \in \mathscr{C}$, every covering sieve $\mathscr{U} \to c$, and every pair of objects $x, y \in \mathscr{F}(c)$, we can form a commutative diagram
    \begin{equation*}
        \begin{tikzcd}
            \operatorname{Map}_{\mathscr{F}_1(c)}(\phi(x), \phi(y)) \ar[r] \ar[d] & \mathrm{N}_1 \mathscr{F} (c) \ar[r] \ar[d] & \mathrm{N}_0 \mathscr{F}(c) \times \mathrm{N}_0 \mathscr{F}(c) \ar[d] \\
            \operatorname{Map}(\mathscr{U}, \underline{\operatorname{Map}}_{\mathscr{F}_1}(\phi(x), \phi(y))) \ar[r] & \operatorname{Map}(\mathscr{U}, \mathrm{N}_1 \mathscr{F}) \ar[r] & \operatorname{Map}(\mathscr{U}, \mathrm{N}_0 \mathscr{F} \times \mathrm{N}_0 \mathscr{F})
        \end{tikzcd}
    \end{equation*}
    of spaces where each row is a fiber sequence and we have used Remark~\ref{remark:flagged-mapping-space} to identify the fibers.

    Thus, we find that $\mathscr{F}$ has $\tau$-local mapping spaces if and only if, for every object $c \in \mathscr{C}$ and every covering sieve $\mathscr{U} \to c$, the diagram
    \begin{equation*}
        \begin{tikzcd}
            \mathrm{N}_1 \mathscr{F} (c) \ar[r] \ar[d] & \mathrm{N}_0 \mathscr{F}(c) \times \mathrm{N}_0 \mathscr{F}(c) \ar[d] \\
            \operatorname{Map}(\mathscr{U}, \mathrm{N}_1 \mathscr{F}) \ar[r] & \operatorname{Map}(\mathscr{U}, \mathrm{N}_0 \mathscr{F} \times \mathrm{N}_0 \mathscr{F})
        \end{tikzcd}
    \end{equation*}
    is Cartesian.
\end{remark}

\begin{proposition} \label{prop:sheaf-criteria-flagged-cats}
    Let $\mathscr{F} = (\mathscr{F}_0, \mathscr{F}_1, \phi)$ be a presheaf of flagged $\infty$-categories on $\mathscr{C}$.
    Then $\mathscr{F}$ is a sheaf for the $\tau$-toplogy if and only if it has $\tau$-local mapping spaces and $\mathscr{F}_0$ is a sheaf for the $\tau$-topology.
\end{proposition}

\begin{proof}
    This follows immediately from Proposition~\ref{prop:flagged-segal}, Proposition~\ref{prop:descent-for-segal-spaces}, Remark~\ref{remark:flagged-mapping-space}, and Remark~\ref{remark:tau-local}.
\end{proof}

\subsection{Group Completion}

\begin{definition}
    An $\infty$-category $\mathscr{C}$ is \emph{semiadditive} if
    \begin{enumerate}
        \item it is pointed,
        \item it admits finite products and finite coproducts, and
        \item the map
        \begin{equation*}
            \begin{pmatrix}
                \mathrm{id}_X & 0 \\
                0 & \mathrm{id}_Y
            \end{pmatrix} : X \sqcup Y \to X \times Y
        \end{equation*}
        is an equivalence for every $X, Y \in \mathscr{C}$.
    \end{enumerate}
    For $X, Y \in \mathscr{C}$, we write $X \oplus Y$ to denote the product $X \times Y$.
\end{definition}

\begin{definition}
    Let $\mathscr{C}$ be a semiadditive $\infty$-category.
    An object $X \in \mathscr{C}$ is called \emph{group complete} if the shear map
    \begin{equation*}
    \begin{pmatrix}
        \mathrm{id}_X & \mathrm{id}_X \\
        0 & \mathrm{id}_X
    \end{pmatrix} : X \oplus X \to X \oplus X
    \end{equation*}
    is an equivalence.
    We write $\mathscr{C}^\gp$ for the full subcategory of $\mathscr{C}$ spanned by group complete objects.
\end{definition}

\begin{definition}
    Let $\mathscr{C}$ be a semiadditive $\infty$-category.
    We say $\mathscr{C}$ is \emph{additive} if every object $X \in \mathscr{C}$ is group complete.
\end{definition}

\begin{proposition} \label{prop:grplike-accessible}
    Let $\mathscr{C}$ be a semiadditive presentable $\infty$-category.
    Then the full subcategory $\mathscr{C}^\gp \subseteq \mathscr{C}$ is closed under small limits and small colimits.
    In particular, it is additive presentable.
\end{proposition}

\begin{proof}
    This is clear using the fact small limits and sifted colimits both commute with finite products, and the fact that finite products coincide with finite coproducts.
\end{proof}

\begin{definition}
    Let $\mathscr{C}$ be a semiadditive presentable $\infty$-category.
    We write $({-})^\gp : \mathscr{C} \to \mathscr{C}^\gp$ for the left adjoint of the inclusion $\mathscr{C}^\gp \to \mathscr{C}$.
    A morphism $f : X \to Y$ in $\mathscr{C}$ is called a $({-})^\gp$-equivalence if $f^\gp : X^\gp \to Y^\gp$ is an equivalence.
\end{definition}

\begin{definition}
    Let $\mathscr{C}$ be a pointed $\infty$-category with finite colimits.
    A morphism $f : X \to Y$ is called a \emph{$\Sigma$-equivalence} if $\Sigma f : \Sigma X \to \Sigma Y$ is an equivalence.
\end{definition}

\begin{proposition} \label{prop:gp-equiv-sigma-equiv}
    Let $\mathscr{C}$ be a semiadditive presentable $\infty$-category and let $f : X \to Y$ be a morphism in $\mathscr{C}$.
    If $f$ is a $({-})^\gp$-equivalence, then it is a $\Sigma$-equivalence.
    If moreover $\Sigma : \mathscr{C}^\gp \to \mathscr{C}^\gp$ is conservative, then the converse holds.
\end{proposition}

\begin{proof}
    Fix a morphism $f : X \to Y$ in $\mathscr{C}$.

    Suppose that $f^\gp$ is an equivalence, and let $Z$ be an arbitrary object of $\mathscr{C}$.
    Using the adjunction $\Sigma \dashv \Omega$, we get a commutative diagram
    \begin{equation*}
        \begin{tikzcd}
            \operatorname{Map}(\Sigma Y, Z) \ar[r,"(\Sigma f)^*"] \ar[d,"\simeq"] & \operatorname{Map}(\Sigma X, Z) \ar[d,"\simeq"] \\
            \operatorname{Map}(Y, \Omega Z) \ar[r,"f^*"] & \operatorname{Map}(X, \Omega Z).
        \end{tikzcd}
    \end{equation*}
    Since $\Omega Z$ is group complete, we have that bottom horizontal arrow is an equivalence, and hence so is the top horizontal arrow.
    Since $Z$ was arbitrary, we conclude that $\Sigma f$ is an equivalence.

    Now suppose $\Sigma : \mathscr{C}^\gp \to \mathscr{C}^\gp$ is conservative and $\Sigma f$ is an equivalence.
    Since $({-})^\gp : \mathscr{C} \to \mathscr{C}^\gp$ preserves colimits, the map $\Sigma (f^\gp)$ is also an equivalence.
    By the conservativity assumption, we have that $f^\gp$ is an equivalence.
\end{proof}

\begin{remark} \label{remark:add-proj-gen-is-prestable}
    In practice, we will use the above proposition when $\mathscr{C}$ is semiadditive projectively generated, \emph{i.e.} when $\mathscr{C} \simeq \PSh_\Sigma(\mathscr{C}_0)$ for some small semiadditive $\infty$-category $\mathscr{C}_0$.
    In this case, the $\infty$-category $\mathscr{C}^\gp$ is additive projectively generated, and thus Grothendieck prestable by \cite[Remark~C.1.5.10]{Lurie_SAG}.
\end{remark}

\begin{proposition} \label{prop:sigma-semiadd}
    Let $\mathscr{C}$ be a semiadditive projectively generated $\infty$-category.
    For all $X \in \mathscr{C}$, we have $\Sigma X \in \mathscr{C}^\gp$
\end{proposition}

\begin{proof}
    This follows from the argument in \cite[Lemma~2.14]{Bachmann_Cancellation}.
\end{proof}

\subsection{Polynomial Functors}

\begin{definition}
    Let $\mathscr{C}$ and $\mathscr{D}$ be pointed $\infty$-categories that admit finite colimits, and let $n \in \mathbb{Z}$.
    A functor $F : \mathscr{C} \to \mathscr{D}$ is \emph{polynomial of degree $\leq n$} if
    \begin{itemize}
        \item we have $n \leq -1$ and $F$ is the zero functor, or
        \item we have $n \geq 0$ and for every $Y \in \mathscr{C}$ the functor $\mathscr{D}_Y F : \mathscr{C} \to \mathscr{D}$ defined by
        \begin{equation*}
            X \mapsto \operatorname{cofib}(F(i_1) : F(Y) \to F(X \vee Y))
        \end{equation*}
        is polynomial of degree $\leq n-1$.
    \end{itemize}
\end{definition}

\begin{proposition} \label{prop:multivariable-deriv-is-poly}
    Let $\mathscr{C}$ and $\mathscr{D}$ be pointed $\infty$-categories that admit finite colimits, and let $F : \mathscr{C} \to \mathscr{D}$ be a functor polynomial of degree $\leq n$.
    For every $k \geq 0$, the functor
    \begin{align*}
        \Delta F : \mathscr{C}^{\times k} &\to \mathscr{D} \\
        (X_1, \ldots, X_k) &\mapsto \operatorname{cofib}\left( \bigvee_{i=1}^k F(X_i) \to F \left(\bigvee_{i=1}^k X_i \right) \right)
    \end{align*}
    is polynomial of degree $\leq n-1$ in each variable.
\end{proposition}

\begin{proof}
    Fix $ k \geq 0$.
    If $k = 0$ then there is nothing to show, so assume $k \geq 1$.
    By symmetry, it suffices to show $\Delta F : \mathscr{C}^{\times k} \to \mathscr{D}$ is polynomial of degree $\leq n-1$ its first variable.
    Fix objects $X_2, \ldots, X_k$ in $\mathscr{C}$ and let $G : \mathscr{C} \to \mathscr{D}$ be the functor
    \begin{equation*}
        X \mapsto \Delta F(X,X_2, \ldots, X_k).
    \end{equation*}
    For all $X$ and $Y$, we can form the commutative diagram
    \begin{equation*}
        \begin{tikzcd}
            F(X) \vee \left( \bigvee_{i=2}^k F(X_i) \right) \ar[r] \ar[d] & F\left(X \vee \left( \bigvee_{i=2}^k X_i \right) \right) \ar[r] \ar[d] & G(X) \ar[d] \\
            F(X \vee Y) \vee \left( \bigvee_{i=2}^k F(X_i) \right) \ar[r] \ar[d] & F\left(X \vee Y \vee \left( \bigvee_{i=2}^k X_i \right) \right) \ar[r] \ar[d] & G(X \vee Y) \ar[d] \\
            \mathscr{D}_Y F(X) \ar[r] & \mathscr{D}_Y F \left(X \vee \left( \bigvee_{i=2}^k X_i \right) \right) \ar[r] & \mathscr{D}_Y G(X)
        \end{tikzcd}
    \end{equation*}
    in $\mathscr{D}$ where each row and each column is a cofiber sequence.
    We find thus that $\mathscr{D}_Y G \simeq \mathscr{D}_Z \mathscr{D}_Y F$, where
    \begin{equation*}
        Z = \bigvee_{i=2}^k X_i.
    \end{equation*}
    Since $F$ is polynomial of degree $\leq n$, we have that $\mathscr{D}_Y G$ is polynomial of degree $\leq n - 2$, and so $G$ is polynomial of degree $\leq n -1$.
\end{proof}

We are grateful to Tom Bachmann for providing us with the following proposition.

\begin{proposition} \label{prop:polynomial-preserves-gp-equiv}
    Let $\mathscr{C}$ and $\mathscr{D}$ be semiadditive projectively generated $\infty$-categories, and let $F : \mathscr{C} \to \mathscr{D}$ be a polynomial functor that preserves geometric realizations of simplicial objects.
    Then $F$ sends $({-})^\gp$-equivalences to $({-})^\gp$-equivalences.
\end{proposition}

\begin{proof}
    For this proof, we'll write $\mathrm{L}_\gp = ({-})^\gp : \mathscr{C} \to \mathscr{C}^\gp$ for the group completion functor, and $\iota : \mathscr{C}^\gp \to \mathscr{C}$ for inclusion.

    We proceed by induction on $n$.
    If $n \leq -1$, there is nothing to show.

    By Remark~\ref{remark:add-proj-gen-is-prestable} and Proposition~\ref{prop:gp-equiv-sigma-equiv}, it suffices to show that $\Sigma F : \mathscr{C} \to \mathscr{D}$ sends the unit transformation
    \begin{equation*}
        \eta : \mathrm{id} \to \iota \mathrm{L}_\gp
    \end{equation*}
    to an equivalence.
    Using Proposition~\ref{prop:sigma-semiadd}, we have that $\Sigma F$ lands in $\mathscr{D}^\gp \subseteq \mathscr{D}$.

    Consider the cofiber sequence
    \begin{equation*}
        \Sigma F \to \mathrm{L}_\gp F\Sigma \to G
    \end{equation*}
    of functors $\mathscr{C} \to \mathscr{D}^\gp$.
    By Proposition~\ref{prop:gp-equiv-sigma-equiv}, the functor $\mathrm{L}_\gp F \Sigma$ inverts $({-})^\gp$-equivalences.
    Thus, to show that $\Sigma F$ inverts $({-})^\gp$-equivalences, it suffices to show that same for $G$.
    Using the canonical description of suspension as a geometric realization of coproducts and the assumption that $F$ preserves geometric realizations, we have that $G$ is a sifted colimit of terms of the form
    \begin{equation*}
        \operatorname{cofib}\left( \mathrm{L}_\gp F(X)^{\vee k} \to \mathrm{L}_\gp F \left(X^{\vee k} \right) \right).
    \end{equation*}
    Since $\mathrm{L}_\gp$ preserves colimits, we have that $\mathrm{L}_\gp F$ is polynomial of degree $\leq n$ \cite[Lemma~5.24]{BachmannHoyois_Norms}.
    By Proposition~\ref{prop:multivariable-deriv-is-poly} and the inductive hypothesis, we conclude that each of these terms inverts $({-})^\gp$-equivalences, and so we have the result.
\end{proof}

\section{Normed Category Theory}

In this section, we further develop the theory of norm monoidal $\infty$-categories introduced in \cite[\S\S 6 and 7]{BachmannHoyois_Norms} and \cite[\S 3]{Bachmann_MotivicSpectralMackey}.
We introduce notions of lax norm monoidal functors (Definition~\ref{def:lax-norm-monoidal-functor}) and norm monoidal adjunctions (Definition~\ref{def:norm-monoidal-adjunction}), which play critical roles in our proofs later.

\begin{warning}
    We deviate slightly from the language set out in \cite{BachmannHoyois_Norms}.
    \begin{itemize}
        \item Our notion of normed algebra is more general than the one used in \cite[\S 7]{BachmannHoyois_Norms}. (See Remark~\ref{remark:deviation-nalg})
        \item Our notion of presentably norm monoidal $\infty$-category is more restrictive than the one used in \cite[\S 6]{BachmannHoyois_Norms}. (See Remark~\ref{remark:deviation-presentably-norm-monoidal})
    \end{itemize}
    These modifications make for a more robust theory, and do not change any of the major results of \cite{BachmannHoyois_Norms}.
    In any case, the language we use is consistent with that used in \cite{Bachmann_MotivicSpectralMackey}.
\end{warning}

\begin{remark} \label{remark:fold-transfers-cmons}
    Let $\mathscr{B}$ be an extensive $\infty$-category. (See \cite[Defintion 2.3]{BachmannHoyois_Norms}.)
    We write $\Corr^\fold(\mathscr{B}) = \mathscr{S}\mathrm{pan}(\mathscr{B}, \fold, \all)$ for the $\infty$-category of spans in $\mathscr{B}$ where the backwards morphism is a finite coproduct of fold maps $\nabla : X^{\sqcup n} \to X$.
    Recall \cite[Appendix C]{BachmannHoyois_Norms} that there is an equivalence between
    \begin{itemize}
        \item functors $\Corr^\fold(\mathscr{B})^\op \to \Cat_\infty$ that preserve finite products, and
        \item functors $\mathscr{B}^\op \to \CAlg(\Cat_\infty)$ that preserve finite products.
    \end{itemize}
    If $\mathscr{D} : \mathscr{B}^\op \to \CAlg(\Cat_\infty)$ preserves finite products, and $\hat{\mathscr{D}} : \Corr^\fold(\mathscr{B})^\op \to \Cat_\infty$ is the corresponding functor, then there is also an equivalence between
    \begin{itemize}
        \item sections of the coCartesian fibration corresponding to $\hat{\mathscr{D}}$, and
        \item sections of the coCartesian fibration corresponding to $\CAlg(\mathscr{D}({-}))$.
    \end{itemize}
\end{remark}

\begin{notation}
    Let $S$ be an algebraic space.
    We write $\mathscr{B} \subseteq_\fet \AlgSpc_S$ to indicate that $\mathscr{B}$ is a full subcategory of $\AlgSpc_S$ such that
    \begin{itemize}
        \item $S \in \mathscr{B}$, and
        \item $\mathscr{B}$ is closed under finite coproducts and finite \'etale extensions.
    \end{itemize}
    In this case, we write $\Corr^\fet(\mathscr{B})$ for the $\infty$-category of spans 
    \begin{equation*}
        \begin{tikzcd}
            & \widetilde{X} \ar[dl,"p"'] \ar[dr,"f"] \\
            X & & Y
        \end{tikzcd}
    \end{equation*}
    in $\mathscr{B}$ where the backwards morphism $p$ is finite \'etale.

    We also write $\mathscr{B} \subseteq_\fet \Sm_S$ to mean $\mathscr{B} \subseteq_\fet \AlgSpc_S$ and $\mathscr{B} \subseteq \Sm_S$.
\end{notation}

\subsection{Norm Monoidal \texorpdfstring{$\infty$}{oo}-Categories}

\begin{definition}
    Let $B$ be an algebraic space and let $\mathscr{B} \subseteq_\fet \AlgSpc_B$.

    A \emph{norm monoidal $\infty$-category} over $\mathscr{B}$ is a coCartesian fibration
    \begin{equation*}
        \mathscr{D}^\otimes \to \Corr^\fet(\mathscr{B})^\op
    \end{equation*}
    such that the corresponding functor $\mathscr{D} : \Corr^\fet(\mathscr{B})^\op \to \Cat_\infty$
    preserves finite products.
    For an arbitrary morphism $f : S \to T$ in $\mathscr{B}$ and a finite \'etale morphism $p : \widetilde{S} \to S$ in $\mathscr{B}$, we write $f^* : \mathscr{D}(T) \to \mathscr{D}(S)$ and $p_\otimes : \mathscr{D}\big( \widetilde{S} \big) \to \mathscr{D}(S)$ for the images of 
    \begin{equation*}
        \begin{tikzcd}
            & S \ar[dl,"\mathrm{id}"'] \ar[dr,"f"] \\
            S & & T
        \end{tikzcd} \qquad\text{and}\qquad \begin{tikzcd}
            & \widetilde{S} \ar[dl,"p"'] \ar[dr,"\mathrm{id}"] \\
            S & & \widetilde{S}
        \end{tikzcd}
    \end{equation*}
    under $\mathscr{D}$, respectively.

    We write $\NMon(\Cat_\infty | \mathscr{B}) \subseteq \Fun(\Corr^\fet(\mathscr{B})^\op, \Cat_\infty)$ for the full subcategory spanned by norm monoidal $\infty$-categories over $\mathscr{B}$.
\end{definition}

\begin{remark}
    Let $B$ be an algebraic space, and let $\mathscr{B} \subseteq_\fet \AlgSpc_B$.
    Recall from \cite[Appendix C]{BachmannHoyois_Norms} that there is an equivalence between $\NMon(\Cat_\infty|\mathscr{B})$ and functors
    \begin{equation*}
        \Corr^\fet(\mathscr{B})^\op \to \CAlg(\Cat_\infty)
    \end{equation*}
    that preserve finite products.
    In particular, given a norm monoidal $\infty$-category $\mathscr{D}$ over $\mathscr{B}$ and $S \in \mathscr{B}$, we can view $\mathscr{D}(S)$ as a symmetric monoidal $\infty$-category.
\end{remark}

\begin{example}
    Let $B$ be an algebraic space and let $\mathscr{B} \subseteq_\fet \AlgSpc_B$.
    The terminal norm monoidal $\infty$-category over $\mathscr{B}$ is the constant functor
    \begin{equation*}
        \Corr^\fet(\mathscr{B})^\op \to \Cat_\infty, S \mapsto *.
    \end{equation*}
    We write $*_\mathscr{B}$ for this norm monoidal $\infty$-category.
    Note that the corresponding coCartesian fibration is the identity functor on $\Corr^\fet(\mathscr{B})^\op$.
\end{example}

\begin{example}
    The assignment
    \begin{equation*}
        S \mapsto \AlgSpc_S
    \end{equation*}
    can be promoted to a norm monoidal $\infty$-category over $\AlgSpc$.
    Indeed, for a finite \'etale map $p : \widetilde{S} \to S$ of algebraic spaces, the functor $p^* : \AlgSpc_S \to \AlgSpc_{\widetilde{S}}$ admits a right adjoint called Weil restriction.
    See \cite[\href{https://stacks.math.columbia.edu/tag/05YF}{Tag 05YF}]{stacks-project}
    We can thus apply the unfurling construction of Barwick \cite[\S11]{Barwick_SpectralMackey} to extend the functor
    \begin{equation*}
        \AlgSpc_\star : \AlgSpc^\op \to \Cat_\infty
    \end{equation*}
    along $\AlgSpc^\op \to \Corr^{\fet}(\AlgSpc)^\op$.
    The induced symmetric monoidal structures on $\AlgSpc_S$ is the Cartesian monoidal structure.
    Since Weil restriction preserves smoothness, we can also promote $S \mapsto \Sm_S$ to a norm monoidal $\infty$-category $\Sm_\star$ over $\AlgSpc$.
\end{example}

\begin{example} \label{example:sm-s-plus-normed}
    We can promote the assignment
    \begin{equation*}
        S \mapsto \Sm_{S+}
    \end{equation*}
    to a norm monoidal $\infty$-category over $\AlgSpc$ where the norm functors are again (induced by) Weil restriction.
    See the beginning of \cite[\S6.1]{BachmannHoyois_Norms}.
    The induced symmetric monoidal structure on $\Sm_{S+}$ is the smash product $X_+ \otimes Y_+ = (X \times Y)_+$.
    This example will be particularly important to us in the sequel.
\end{example}

\begin{example}
    One of the main constructions in \cite{BachmannHoyois_Norms} is a norm monoidal refinement of the assignment $S \mapsto \SH(S)$.
    The induced symmetric monoidal structure on $\SH(S)$ is the usual tensor product of motivic spectra.
\end{example}

\begin{definition} \label{def:lax-norm-monoidal-functor}
    Let $B$ be an algebraic space, let $\mathscr{B} \subseteq_\fet \AlgSpc_B$, and let $\mathscr{D}_1, \mathscr{D}_2$ be norm monoidal $\infty$-categories over $\mathscr{B}$.
    A \emph{lax norm monoidal functor} $F : \mathscr{D}_1 \to \mathscr{D}_2$ over $\mathscr{B}$ is a commutative diagram
    \begin{equation*}
        \begin{tikzcd}
            \mathscr{D}_1^\otimes \ar[rr,"F"] \ar[dr] & & \mathscr{D}_2^\otimes \ar[dl] \\
            & \Corr^\fet(\mathscr{B})^\op
        \end{tikzcd}
    \end{equation*}
    such that the restriction along $\mathscr{B}^\op \to \Corr^\fet(\mathscr{B})^\op$ preserves coCartesian edges over smooth morphisms.
    We say that $F$ is \emph{strong norm monoidal} (or simply \emph{norm monoidal}) if it preserves all coCartesian edges.
\end{definition}

\begin{remark}
    Let $B$ be an algebraic space and let $\mathscr{B} \subseteq_\fet \AlgSpc_B$.
    Note that strong norm monoidal functors between norm monoidal $\infty$-categories over $\mathscr{B}$ are exactly the morphisms in $\NMon(\Cat_\infty|\mathscr{B})$.
\end{remark}

\begin{remark}
    Let $B$ be an algebraic space and let $\mathscr{B} \subseteq_\fet \AlgSpc_B$.
    It is possible to construct an $\infty$-bicategory $\mathbf{NMon}(\Cat_\infty|\mathscr{B})$ where the objects are norm monoidal $\infty$-categories over $\mathscr{B}$ and the one-cells are lax norm monoidal functors.
    It is a sub-bicategory of the $\infty$-bicategory of coCartesian fibrations over $\Corr^\fet(\mathscr{B})^\op$.
    This is useful to keep in mind, but we do not make any serious use of the framework of $\infty$-bicategories here.
\end{remark}

\begin{construction}
    Let $f : B_1 \to B_2$ be a morphism of algebraic spaces, let $\mathscr{B}_1 \subseteq_\fet \AlgSpc_{B_1}$ and $\mathscr{B}_2 \subseteq_\fet \AlgSpc_{B_2}$, and suppose $f_\sharp (\mathscr{B}_1) \subseteq \mathscr{B}_2$.
    Then there is a functor $\Corr^\fet(\mathscr{B}_1) \to \Corr^\fet(\mathscr{B}_2)$.
    Restriction along this functor sends norm monoidal $\infty$-categories over $\mathscr{B}_2$ to norm monoidal $\infty$-categories over $\mathscr{B}_1$, and similarly for (lax) norm monoidal functors.
    We will often perform these restriction implicitly and without comment.
\end{construction}

\begin{definition}
    Let $B$ be an algebraic space, let $\mathscr{B} \subseteq_\fet \AlgSpc_B$, and let $\mathscr{D}$ be a norm monoidal $\infty$-category over $\mathscr{B}$.
    A \emph{normed algebra} $\mathscr{A}$ in $\mathscr{D}$ over $\mathscr{B}$ is a section
    \begin{equation*}
        \begin{tikzcd}
            \mathscr{D}^\otimes \ar[d] \\
            \Corr^\fet(\mathscr{B})^\op\ar[u,bend left=45, dashed,"\mathscr{A}"]
        \end{tikzcd}
    \end{equation*}
    whose restriction along $\mathscr{B}^\op \to \Corr^\fet(\mathscr{B})^\op$ sends all smooth morphisms to coCartesian edges.
    We say $\mathscr{A}$ is \emph{steady} if the restriction sends all morphisms in $\mathscr{B}^\op$ to coCartesian edges.
    We write $\NAlg(\mathscr{D} | \mathscr{B})$ for the $\infty$-category of normed algebras in $\mathscr{D}$ over $\mathscr{B}$, and $\NAlg^\mathrm{stdy}(\mathscr{D} | \mathscr{B})$ for the full subcategory spanned by steady normed algebras.

    When $\mathscr{B} = \Sm_B$, we refer to these simply as normed algebras in $\mathscr{D}$ over $B$, and we write $\NAlg(\mathscr{D} | B) = \NAlg(\mathscr{D} | \Sm_B)$.
\end{definition}

\begin{remark}
    Let $B$ be an algebraic space and let $\mathscr{B} \subseteq_\fet \AlgSpc_B$.
    Note that a normed algebra in $\mathscr{D}$ over $\mathscr{B}$ is exactly a lax norm monoidal functor $*_\mathscr{B} \to \mathscr{D}$.
\end{remark}

\begin{remark} \label{remark:sm-nalg-is-nalg}
    Let $B$ be an algebraic space, let $\mathscr{B} \subseteq_\fet \AlgSpc_B$, and let $\mathscr{D}$ be a norm monoidal $\infty$-category over $\mathscr{B}$.
    If $\mathscr{B} \subseteq \Sm_B$, then every normed algebra in $\mathscr{D}$ over $\mathscr{B}$ is steady.
    Indeed, let $\mathscr{A} \in \NAlg(\mathscr{D}|\mathscr{B})$ and let
    \begin{equation*}
        \begin{tikzcd}
            S \ar[rr,"f"] \ar[dr,"p"'] & & T \ar[dl,"q"] \\
            & B
        \end{tikzcd}
    \end{equation*}
    be a morphism in $\mathscr{B}$.
    By assumption, the arrows $p$ and $q$ are smooth morphisms, so $\mathscr{A}(p)$ and $\mathscr{A}(q)$ are coCartesian edges in $\mathscr{D}^\otimes$.
    By \cite[\href{https://kerodon.net/tag/01TS}{Tag 01TS}]{kerodon}, we conclude that $\mathscr{A}(f)$ is also a coCartesian edge.
\end{remark}

\begin{remark} \label{remark:deviation-nalg}
    The above definition of normed algebra is more general than the one in \cite[\S 7]{BachmannHoyois_Norms}.
    Normed algebras over $\mathscr{B}$ there are what we call steady normed algebras over $\mathscr{B}$.
    By Remark~\ref{remark:sm-nalg-is-nalg}, the two definitions agree when $\mathscr{B} \subseteq \Sm_B$.
    In \cite[\S 3]{Bachmann_MotivicSpectralMackey}, normed algebras are only defined over $\Sm_B$.
\end{remark}

\begin{example}
    Let $B$ be an algebraic space, let $\mathscr{B} \subseteq_\fet \AlgSpc_B$, and let $\mathscr{D}$ be a norm monoidal $\infty$-category over $\mathscr{B}$.
    The assignment $S \mapsto \mathbf{1}_{\mathscr{D}(S)}$ can be promoted to a steady normed algebra $\mathbf{1}_\mathscr{D} \in \NAlg^\mathrm{stdy}(\mathscr{D} | \mathscr{B})$.
    We call it the \emph{norm monoidal unit} of $\mathscr{D}$ over $\mathscr{B}$.
\end{example}

\begin{construction}
    Let $B$ be an algebraic space, let $\mathscr{B} \subseteq_\fet \AlgSpc_B$, and let $\mathscr{D}$ be a norm monoidal $\infty$-category over $\mathscr{B}$.
    Note that there is a functor $\Phi : \Corr^\fold(\mathscr{B})^\op \to \Corr^\fet(\mathscr{B})^\op$.
    Using the equivalences in Remark~\ref{remark:fold-transfers-cmons}, precomposition with $\Phi$ induces a functor
    \begin{equation*}
        \NAlg(\mathscr{D} | \mathscr{B}) \to \CAlg(\mathscr{D} | \mathscr{B}),
    \end{equation*}
    where $\CAlg(\mathscr{D} | \mathscr{B}) \subseteq \Sect(\CAlg(\mathscr{D}({-})) | \mathscr{B}^\op)$ is the full subcategory spanned by sections that send smooth morphisms to coCartesian ones.
\end{construction}

\begin{remark} \label{remark:parametrized-calg}
    Let $B$ be an algebraic space, let $\mathscr{B} \subseteq_\fet \AlgSpc_B$, and let $\mathscr{D}$ be a norm monoidal $\infty$-category over $\mathscr{B}$.
    If $\mathscr{B} \subseteq \Sm_B$, then $\CAlg(\mathscr{D} | \mathscr{B}) \cong \CAlg(\mathscr{D}(B))$.
\end{remark}

\begin{proposition} \label{prop:nalg-forget-conservative}
    Let $B$ be an algebraic space, let $\mathscr{B} \subseteq_\fet \AlgSpc_B$, and let $\mathscr{D}$ be a norm monoidal $\infty$-category over $\mathscr{B}$.
    The forgetful functor
    \begin{equation*}
        \NAlg(\mathscr{D} | \mathscr{B}) \to \CAlg(\mathscr{D} | \mathscr{B})
    \end{equation*}
    is conservative.
\end{proposition}

\begin{proof} 
    To check that a natural transformation is an equivalence, it suffices to check objectwise.
    Thus, the claim follows from the fact that $\Phi : \Corr^\fold(\mathscr{B})^\op \to \Corr^\fet(\mathscr{B})^\op$ is essentially surjective.
\end{proof}

\begin{construction}
    Let $B$ be an algebraic space, let $\mathscr{B} \subseteq_\fet \AlgSpc_B$, and let $F : \mathscr{D}_1 \to \mathscr{D}_2$ be a lax norm monoidal functor over $\mathscr{B}$.
    Postcomposition with $F : \mathscr{D}_1^\otimes \to \mathscr{D}_2^\otimes$ induces a functor $F : \NAlg(\mathscr{D}_1|\mathscr{B}) \to \NAlg(\mathscr{D}_2|\mathscr{B})$.
    If $F$ is strong norm monoidal, then it restricts to a functor $F : \NAlg^\mathrm{stdy}(\mathscr{D}_1| \mathscr{B}) \to \NAlg^\mathrm{stdy}(\mathscr{D}_2|\mathscr{B})$.
\end{construction}

\begin{definition} \label{def:norm-monoidal-adjunction}
    Let $B$ be an algebraic space, let $\mathscr{B} \subseteq_\fet \AlgSpc_B$, and let
    \begin{equation*}
        \begin{tikzcd}
            \mathscr{D}_1^\otimes \ar[dr,"p"'] & & \mathscr{D}_2^\otimes \ar[dl] \\
            & \Corr^\fet(\mathscr{B})^\op
        \end{tikzcd}
    \end{equation*}
    be norm monoidal $\infty$-categories.
    A \emph{norm monoidal adjunction} between $\mathscr{D}_1$ and $\mathscr{D}_2$ consists of lax norm monoidal functors
    \begin{equation*}
        F : \mathscr{D}_1 \to \mathscr{D}_2 \qquad\text{and}\qquad  G : \mathscr{D}_2 \to \mathscr{D}_1
    \end{equation*}
    and a natural transformation $\eta : \mathrm{id} \to GF$ such that
    \begin{enumerate}
        \item $\eta$ is the unit of an adjunction $F : \mathscr{D}_1^\otimes \rightleftarrows \mathscr{D}_2^\otimes : G$, and
        \item $p(\eta)$ is the identity transformation.
    \end{enumerate}
\end{definition}

\begin{remark}
    In other words, a norm monoidal adjunction is a relative adjunction consisting of lax norm monoidal functors.
    See \cite[\S 7.3.2]{Lurie_HA} or \cite[Appendix D]{BachmannHoyois_Norms} for more about relative adjunctions.
\end{remark}

\begin{proposition} \label{prop:norm-adjunction-criterion}
    Let $B$ be an algebraic space, let $\mathscr{B} \subseteq_\fet \AlgSpc_B$, let $\mathscr{D}_1$ and $\mathscr{D}_2$ be norm monoidal $\infty$-categories over $\mathscr{B}$, and let $F : \mathscr{D}_1 \to \mathscr{D}_2$ be a lax norm monoidal functor.
    \begin{enumerate}
        \item The norm monoidal functor $F$ is the left adjoint in a norm monoidal adjunction if and only if the following three conditions hold:
            \begin{enumerate}
                \item $F$ is strong norm monoidal,
                \item for each $S \in \mathscr{B}$, the functor $F_S : \mathscr{D}_1(S) \to \mathscr{D}_2(S)$ admits a right adjoint $G_S$, and
                \item for every smooth morphism $p : X \to S$ in $\mathscr{B}$, the exchange transformation
                    \begin{equation*}
                        p^* G_S \to G_X p^* : \mathscr{D}_2(S) \to \mathscr{D}_1(X)
                    \end{equation*}
                    is an equivalence.
            \end{enumerate}
        \item The norm monoidal functor $F$ is the right adjoint in a norm monoidal adjunction if and only if the following two conditions hold:
        \begin{enumerate}
            \item for each $S \in \mathscr{B}$, the functor $F_S : \mathscr{D}_1(S) \to \mathscr{D}_2(S)$ admits a left adjoint $G_S$, and
            \item for every morphism $f : S \to T$ in $\mathscr{B}$, the exchange transformation
                \begin{equation*}
                    f^* G_T \to G_S f^* : \mathscr{D}_2(T) \to \mathscr{D}_1(S)
                \end{equation*}
                is an equivalence.
        \end{enumerate}
    \end{enumerate}
\end{proposition}

\begin{proof}
    Recall from (the dual of) \cite[\href{https://kerodon.net/tag/01U6}{Tag 01U6}]{kerodon} that a map in the total space of a coCartesian fibration is locally coCartesian if and only if it is coCartesian.
    Using \cite[Lemma D.3]{BachmannHoyois_Norms}, we find that in either case the functor $F : \mathscr{D}_1^\otimes \to \mathscr{D}_2^\otimes$ admits the approriate adjoint $G : \mathscr{D}_2^\otimes \to \mathscr{D}_1^\otimes$ relative to $\Corr^\fet(\mathscr{B})^\op$ if and only if conditions (a) and (b) hold.
    Finally, the third condition in (1) is equivalent to $G$ being a lax norm monoidal functor.
\end{proof}

\begin{construction} \label{construct:nalg-adj}
    Let $B$ be an algebraic space, let $\mathscr{B} \subseteq_\fet \AlgSpc_B$, let $\mathscr{D}_1, \mathscr{D}_2$ be norm monoidal $\infty$-categories over $\mathscr{B}$, and let $F : \mathscr{D}_1 \rightleftarrows \mathscr{D}_2 : G$ be a norm monoidal adjunction.
    Then we get an induced adjunction
    \begin{equation*}\label{eqn:nalg-adjunction}
        F : \NAlg(\mathscr{D}_1|\mathscr{B}) \rightleftarrows \NAlg(\mathscr{D}_2|\mathscr{B}) : G.
    \end{equation*}
    where the functors, the unit, and the counit are all computed objectwise.
    In particular, adjunction is compatible with forgetful functors.
    For example, the square
    \begin{equation*}
        \begin{tikzcd}
            \NAlg(\mathscr{D}_1|\mathscr{B}) \ar[r,"F"] \ar[d,"U"] & \NAlg(\mathscr{D}_2|\mathscr{B}) \ar[d,"U"] \\
            \mathscr{D}_1(B) \ar[r,"F_B"] & \mathscr{D}_2(B)
        \end{tikzcd}
    \end{equation*}
    commutes, and the exchange transformation $U \circ G_B \to G \circ U$ is an equivalence.
\end{construction}

\begin{definition}
    Let $B$ be an algebraic space, let $\mathscr{B} \subseteq_\fet \AlgSpc_B$, let $\mathscr{D}$ be a norm monoidal $\infty$-category over $\mathscr{B}$, and let $\mathscr{R} \in \NAlg(\mathscr{D}|\mathscr{B})$.
    A \emph{normed $\mathscr{R}$-algebra} is a morphism $\mathscr{R} \to \mathscr{A}$ in $\NAlg(\mathscr{D}|\mathscr{B})$.
    We write $\NAlg_\mathscr{R}(\mathscr{D}|\mathscr{B}) = \NAlg(\mathscr{D}|\mathscr{B})_{\mathscr{R}/}$ for the $\infty$-category of normed $\mathscr{R}$-algebras.
\end{definition}

\begin{proposition}
    Let $B$ be an algebraic space, let $\mathscr{B} \subseteq_\fet \AlgSpc_B$, let $\mathscr{D}$ be a norm monoidal $\infty$-category over $\mathscr{B}$, and let $\mathscr{R} \in \NAlg(\mathscr{D}|\mathscr{B})$. 
    Then the assignment
    \begin{equation*}
        S \mapsto \Mod_{\mathscr{R}(S)}(\mathscr{D}(S))
    \end{equation*}
    can be promoted to a norm monoidal $\infty$-category $\Mod_\mathscr{R}(\mathscr{D})$ over $\mathscr{B}$.
    Moreover, we have an equivalence
    \begin{equation*}
        \NAlg_{\mathscr{R}}(\mathscr{D}|\mathscr{B}) \cong \NAlg(\Mod_{\mathscr{R}}(\mathscr{D})|\mathscr{B}).
    \end{equation*}
\end{proposition}

\begin{proof}
    The same proof as for \cite[Proposition 7.6(4)]{BachmannHoyois_Norms} goes through here.
\end{proof}

\begin{proposition} \label{prop:norm-adj-factorization}
    Let $B$ be an algebraic space, let $\mathscr{B} \subseteq_\fet \AlgSpc_B$, let $\mathscr{D}_1, \mathscr{D}_2$ be norm monoidal $\infty$-categories over $\mathscr{B}$, let $F : \mathscr{D}_1 \rightleftarrows \mathscr{D}_2 : G$ be a norm monoidal adjunction, and let $\mathscr{R} = G(\mathbf{1}_{\mathscr{D}_2})$.
    Then we get a pair of norm monoidal adjunctions
    \begin{equation*}
        \mathscr{D}_1 \rightleftarrows \Mod_{\mathscr{R}}(\mathscr{D}_1) \qquad\text{and}\qquad \Mod_{\mathscr{R}}(\mathscr{D}_1) \rightleftarrows \mathscr{D}_2
    \end{equation*}
    that factor the given adjunction $F \dashv G$.
\end{proposition}

\begin{proof}
    This comes from the functoriality of the construction $(\mathscr{D}, \mathscr{A}) \mapsto \Mod_\mathscr{A}(\mathscr{D})$ from \cite[\S 4.8.3]{Lurie_HA}.
\end{proof}

\subsection{Norms and Presentability}

\begin{notation}
    For this subsection, fix an algebraic space $B$ and a full subcategory $\mathscr{B} \subseteq_{\fet} \AlgSpc_B$.
\end{notation}

\begin{definition}
    A \emph{distribution diagram} in $\mathscr{B}$ is a commutative diagram
    \begin{equation} \label{eqn:distr-diag}
        \begin{tikzcd}
            X \ar[dr,"h"'] & (\mathrm{R}_p X)_S \ar[r,"q"] \ar[d,"g"] \ar[l,"e"'] \arrow[dr,phantom,"\lrcorner", very near start] & \mathrm{R}_p X \ar[d,"f"] \\
            & S \ar[r,"p"] & T
        \end{tikzcd}
    \end{equation}
    in $\mathscr{B}$ where
    \begin{itemize}
        \item $p$ is finite \'etale,
        \item $h$ is smooth,
        \item $f$ is the Weil restriction of $h$ along $p$,
        \item $g$ is the pullback of $f$ along $p$, and
        \item $e$ is the counit map.
    \end{itemize}
    Note that $e$, $f$, and $g$ are smooth, and $q$ is finite \'etale.
\end{definition}

\begin{definition}
    Let $\mathscr{D} : \Corr^{\fet}(\mathscr{B})^\mathrm{op} \to \Cat_\infty$ be a norm monoidal $\infty$-category over $\mathscr{B}$ such that, for every smooth morphism $f : X \to S$ in $\mathscr{B}$, the functor $f^* : \mathscr{D}(S) \to \mathscr{D}(X)$ admits a left adjoint $f_\sharp$.

    Consider a distribution diagram in $\mathscr{B}$ as in (\ref{eqn:distr-diag}).
    The associated \emph{distribution transformation} in $\mathscr{D}$ is a natural transformation
    \begin{equation*}
        \mathrm{Dis} : f_\sharp q_\otimes e^* \to p_\otimes h_\sharp : \mathscr{D}(X) \to \mathscr{D}(T)
    \end{equation*}
    obtained as the composite of the exchange transformation
    \begin{equation*}
        f_\sharp q_\otimes e^* \to p_\otimes g_\sharp e^*
    \end{equation*}
    and the counit transformation
    \begin{equation*}
        p_\otimes g_\sharp e^* \simeq p_\otimes h_\sharp e_\sharp e^* \to p_\otimes h_\sharp.
    \end{equation*}
\end{definition}

\begin{definition}
    Let $\mathscr{D} : \Corr^{\fet}(\mathscr{B})^\mathrm{op} \to \Cat_\infty$ be a norm monoidal $\infty$-category over $\mathscr{B}$ such that, for every smooth morphism $f : X \to S$ in $\mathscr{B}$, the functor $f^* : \mathscr{D}(S) \to \mathscr{D}(X)$ admits a left adjoint $f_\sharp$.
    We say \emph{the distributive law holds} in $\mathscr{D}$ if for every distribution diagram (\ref{eqn:distr-diag}), the associated distribution transformation
    \begin{equation*}
        \mathrm{Dis} : f_\sharp q_\otimes e^* \to p_\otimes h_\sharp : \mathscr{D}(X) \to \mathscr{D}(S)
    \end{equation*}
    is an equivalence.
\end{definition}

\begin{remark}
    Let $\mathscr{D} : \Corr^{\fet}(\mathscr{B})^\mathrm{op} \to \Cat_\infty$ be a norm monoidal $\infty$-category over $\mathscr{B}$ such that, for every smooth morphism $f : X \to S$ in $\mathscr{B}$, the functor $f^* : \mathscr{D}(S) \to \mathscr{D}(X)$ admits a left adjoint $f_\sharp$.
    Let $S \in \mathscr{B}$ and let $\nabla : S^{\sqcup n} \to S$ be the fold map.
    Under the equivalence $\mathscr{D}(S^{\sqcup n}) \to \mathscr{D}(S)^{\times n}$, we can identify $\nabla_\sharp : \mathscr{D}(S^{\sqcup n}) \to \mathscr{D}(S)$ with the coproduct functor.
\end{remark}

\begin{remark} \label{remark:distr-diag-binary-sum}
    Let $p : S \to T$ be a surjective finite \'etale morphism of algebraic spaces and let $h : S \sqcup S \to S$ be the fold map.
    Consider the distribution diagram
    \begin{equation} \label{eqn:distr-diag-binary-sum}
        \begin{tikzcd}
            S \sqcup S \ar[dr,"h"'] & \mathrm{R}_p (S \sqcup S)_S \ar[r,"q"] \ar[d,"g"] \ar[l,"e"'] & \mathrm{R}_p (S \sqcup S) \ar[d,"f"] \\
            & S \ar[r,"p"] & T
        \end{tikzcd}.
    \end{equation}
    The adjoints of the inclusions $p^* T \cong S \to S \sqcup S$ yield maps $T \to \mathrm{R}_p(S \sqcup S)$, which in turn yield a coproduct decomposition $\mathrm{R}_p (S \sqcup S) = T \sqcup C \sqcup T$ for some algebraic space $C$ over $T$.
    The restriction of $e : \mathrm{R}_p(S \sqcup S)_S \to S \sqcup S$ to $C_S \subseteq \mathrm{R}_p(S \sqcup S)_S$ can be decomposed now as
    \begin{equation*}
        e_L \sqcup e_R : L \sqcup R \to S \sqcup S.
    \end{equation*}
    Write $q_l : L \to C$ and $q_r : R \to C$ for the restrictions of $q$ to $L$ and $R$, and write $c : C \to Y$ for the restriction of $f$ to $C$.
    Note that $q_l$ and $q_r$ are again finite \'etale.

    Now let $\mathscr{D} : \Corr^{\fet}(\mathscr{B})^\mathrm{op} \to \Cat_\infty$ be a norm monoidal $\infty$-category over $\mathscr{B}$ such that, for every smooth morphism $f : X \to Y$ in $\mathscr{B}$, the functor $f^* : \mathscr{D}(Y) \to \mathscr{D}(X)$ admits a left adjoint $f_\sharp$.
    Under the equivalence $\mathscr{D}(S \sqcup S) \simeq \mathscr{D}(S) \times \mathscr{D}(S)$, the distribution transformation in $\mathscr{D}$ associated to (\ref{eqn:distr-diag-binary-sum}) can be identified as a transformation
    \begin{equation*}
        p_\otimes(A) \sqcup c_\sharp(q_{l,\otimes} e_l^*(A) \otimes q_{r,\otimes} e_r^*(B)) \sqcup p_\otimes(B) \to p_\otimes(A \sqcup B)
    \end{equation*}
    natural in $(A,B) \in \mathscr{D}(X) \times \mathscr{D}(X)$.
\end{remark}

\begin{definition}
    A \emph{presentably norm monoidal $\infty$-category} over $\mathscr{B}$ is a norm monoidal $\infty$-category $\mathscr{D} : \Corr^{\fet}(\mathscr{B})^\mathrm{op} \to \Cat_\infty$ such that following conditions hold.
    \begin{enumerate}
        \item For every $X \in \mathscr{B}$, the $\infty$-category $\mathscr{D}(X)$ is presentable,
        \item For every morphism $f : S \to T$ in $\mathscr{B}$, the functor $f^* : \mathscr{D}(T) \to \mathscr{D}(S)$ admits a right adjoint $f_*$.
        \item For every smooth morphism $p : X \to S$ in $\mathscr{B}$, the functor $p^* : \mathscr{D}(S) \to \mathscr{D}(X)$ admits a left adjoint $p_\sharp$.
        \item For every Cartesian diagram
        \begin{equation*}
            \begin{tikzcd}
                X \ar[r,"g"] \ar[d,"p"] & Y \ar[d,"q"] \\
                S \ar[r,"f"] & T
            \end{tikzcd}
        \end{equation*}
        in $\mathscr{B}$ with $q$ smooth, the exchange transformation
        \begin{equation*}
            \mathrm{Ex} : p_\sharp g^* \to f^* q_\sharp : \mathscr{D}(Y) \to \mathscr{D}(S)
        \end{equation*}
        is an equivalence.
        \item For every finite \'etale morphism $p : S \to T$ in $\mathscr{B}$, the functor $p_\otimes : \mathscr{D}(S) \to \mathscr{D}(T)$ preserves sifted colimits.
        \item The distributive law holds in $\mathscr{D}$.
    \end{enumerate}
\end{definition}

\begin{remark} \label{remark:deviation-presentably-norm-monoidal}
    The above definition of presentably norm monoidal $\infty$-category is more restrictive than the one in \cite[\S 6]{BachmannHoyois_Norms}.
    There, they only require $p_\sharp$ to exist for $p$ finite \'etale.
    The definition we use here agrees with the one in \cite[\S 3]{Bachmann_MotivicSpectralMackey}.
\end{remark}

\begin{proposition} \label{prop:presentably-norm-adjunction-criterion}
    Let $\mathscr{D}_1$ and $\mathscr{D}_2$ be presentably norm monoidal $\infty$-categories over $\mathscr{B}$, and let $F : \mathscr{D}_1 \to \mathscr{D}_2$ be a lax norm monoidal functor.
    \begin{enumerate}
        \item The norm monoidal functor $F$ is the left adjoint in a norm monoidal adjunction if and only if the following three conditions hold:
        \begin{enumerate}
            \item $F$ is strong norm monoidal
            \item for every $S \in \mathscr{B}$, the functor $F_S : \mathscr{D}_1(S) \to \mathscr{D}_2(S)$ preserves small colimits, and
            \item for every smooth morphism $p : X \to S$ in $\mathscr{B}$, the exchange transformation
                \begin{equation*}
                    p_\sharp F_X \to F_S p_\sharp : \mathscr{D}_1(X) \to \mathscr{D}_2(S)
                \end{equation*}
                is an equivalence.
        \end{enumerate}
        \item The norm monoidal functor $F$ is the right adjoint in a norm monoidal adjunction if and only if the following two conditions hold:
        \begin{enumerate}
            \item for every $S \in \mathscr{B}$, the functor $F_S : \mathscr{D}_1(S) \to \mathscr{D}_2(S)$ is accessible and preserves small limits, and
            \item for every morphism $f : S \to T$ in $\mathscr{B}$, the exchange transformation
                \begin{equation*}
                    f_* F_S \to F_T f_* : \mathscr{D}_1(S) \to \mathscr{D}_2(T)
                \end{equation*}
                is an equivalence.
        \end{enumerate}
    \end{enumerate}
\end{proposition}

\begin{proof}
    Under the presentability hypotheses, these conditions are equivalent to the conditions in Proposition~\ref{prop:norm-adjunction-criterion}.
\end{proof}

\begin{proposition} \label{prop:norm-is-polynomial}
    Let $\mathscr{D} : \Corr^{\fet}(\mathscr{B})^\mathrm{op} \to \Cat_\infty$ be a presentably norm monoidal $\infty$-category over $\mathscr{B}$.
    For every finite \'etale morphsim $p : S \to T$ in $\mathscr{B}$ of degree $\leq n$, the functor $p_\otimes : \mathscr{D}(S) \to \mathscr{D}(T)$ is polynomial of degree $\leq n$.
\end{proposition}

\begin{proof}
    Using Remark (\ref{remark:distr-diag-binary-sum}), the argument of \cite[Proposition~5.25]{BachmannHoyois_Norms} goes through unchanged.
\end{proof}

\begin{proposition} \label{prop:norm-t-structure}
    Let $\mathscr{D}$ be a presentably norm monoidal $\infty$-category over $\mathscr{B}$ such that $\mathscr{D}_1(S)$ is stable for every $S \in \mathscr{B}$.
    For each $S \in \mathscr{B}$, let $\mathscr{E}_0(S) \subseteq \mathscr{D}(S)$ be a collection of objects, and let $\mathscr{E}(S) \subseteq \mathscr{D}(S)$ be the full subcategory generated under colimits and extensions by objects of $\mathscr{E}_0(S)$.
    Suppose that the following conditions hold.
    \begin{enumerate}
        \item For every morphism $f : S \to T$ in $\mathscr{B}$, we have $f^*(\mathscr{E}_0(T)) \subseteq \mathscr{E}(S)$.
        \item For every smooth morphism $h : X \to S$ in $\mathscr{B}$, we have $h_\sharp(\mathscr{E}_0(X)) \subseteq \mathscr{E}(S)$.
        \item For every surjective finite \'etale morphism $p : S \to T$ in $\mathscr{B}$, we have $p_\otimes (\mathscr{E}_0(S)) \subseteq \mathscr{E}(T)$.
        \item For every $S \in \mathscr{B}$ and $M, N \in \mathscr{E}_0(S)$, we have $M \otimes N \in \mathscr{E}(S)$.
        \item For every $S \in \mathscr{B}$, we have $\mathbf{1}_{\mathscr{D}(S)} \in \mathscr{E}(S)$.
    \end{enumerate}
    Then the $\infty$-categories $\mathscr{E}(S)$ can be assembled into a presentably norm monoidal $\infty$-category $\mathscr{E}$ over $\mathscr{B}$.
    Moreover, the inclusion functors $\mathscr{E}(S) \to \mathscr{D}(S)$ can be assembled into the left adjoint of a norm monoidal adjunction $\mathscr{E} \rightleftarrows \mathscr{D}$.
\end{proposition}

\begin{proof}
    The first claim is a very minor modification of the argument for \cite[Proposition 6.7]{BachmannHoyois_Norms}.
    The main difference is that the definition of ``presentably norm monoidal'' used there only requires $h_\sharp$ to exist when $h$ is finite \'etale, but this does not affect the argument in any essential way.
    
    The second claim is then an application of Proposition~\ref{prop:presentably-norm-adjunction-criterion}.
\end{proof}

\begin{proposition} \label{prop:norm-localization}
    Let $\mathscr{D}$ be a presentably norm monoidal $\infty$-category over $\mathscr{B}$.
    For each $S \in \mathscr{B}$, let $w(S)$ be a set of morphisms in $\mathscr{D}(S)$, and let $\mathrm{L}_w \mathscr{D}(S) \subseteq \mathscr{D}(S)$ be the full subcategory of $w(S)$-local objects.
    Let $\mathrm{L}_w : \mathscr{D}(S) \to \mathrm{L}_w \mathscr{D}(S)$ denote the left adjoint of the inclusion, and let $\overline{w}(S)$ denote the class of morphisms whose image under $\mathrm{L}_w$ is invertible.
    Suppose that the following conditions hold.
    \begin{enumerate}
        \item For every morphism $f : S \to T$ in $\mathscr{B}$, we have $f^*(w(T)) \subseteq \overline{w}(S)$.
        \item For every smooth morphism $h : X \to S$ in $\mathscr{B}$, we have $h_\sharp(w(X)) \subseteq \overline{w}(S)$.
        \item For every surjective finite \'etale morphism $p : S \to T$ in $\mathscr{B}$, we have $p_\otimes (w(S)) \subseteq \overline{w}(T)$.
        \item For every $S \in \mathscr{B}$ and $M \in \mathscr{D}(S)$, we have $\mathrm{id}_M \otimes w(S) \subseteq \overline{w}(S)$.
    \end{enumerate}
    Then the $\infty$-categories $\mathrm{L}_w \mathscr{D}(S)$ can be assembled into a presentably norm monoidal $\infty$-category $\mathrm{L}_w \mathscr{D}$ over $\mathscr{B}$.
    Moreover, the localization functors $\mathrm{L}_w : \mathrm{L}_w \mathscr{D}(S) \to \mathscr{D}(S)$ can be assembled into the left adjoint in a norm monoidal adjunction $\mathrm{L}_w : \mathrm{L}_w \mathscr{D} \rightleftarrows \mathscr{D}$.
\end{proposition}

\begin{proof}
    This is a combination of \cite[Proposition 6.15]{BachmannHoyois_Norms} and \cite[Proposition 6.16]{BachmannHoyois_Norms}, with the same caveats as in the proof of Proposition~\ref{prop:norm-t-structure}.
\end{proof}

\section{Motivic Patterns}

In order to streamline proofs and organize the various flavors of transfer for motivic spectra that appear later on in this article, we introduce here the conecpt of a \emph{motivic pattern}.
In essence, a motivic pattern $\mathscr{C}$ is exactly the type of input one needs to construct a nice family of monadic adjunctions
\begin{equation*}
    \SH(S) \rightleftarrows \SH^\mathscr{C}(S)
\end{equation*}
indexed by schemes $S$.
Monadicity tells us that the objects of $\SH^\mathscr{C}(S)$ could be thought of as motivic spectra equipped with extra structure.
The assignment
\begin{equation*}
    S \mapsto \SH^\mathscr{C}(S)
\end{equation*}
is rich with functoriality, yielding a \emph{$(*,\sharp,\otimes)$-formalism} in the sense of \cite{Khan_VoevodskyCriterion}, or a \emph{presentable coefficient system} in the sense of \cite{DrewGallauer_Universal6ff}.

Motivic patterns developed here are also closely related to the \emph{motivic categories of correspondences} of \cite{Bachmann_Cancellation} and the \emph{correspondence categories} of \cite{elmanto2020modules}.

We lay down the foundations of motivic patterns in Subsections~4.1-3.
We discuss in Subsection~4.4 the functoriality of the theory as we vary the motivic pattern.
In Subsections~4.5 and 4.6 we study the compatibility of motivic patterns with norm functors and group completions, respectively.

\subsection{Basic Theory}

\begin{remark}
    Let $\mathscr{C}$ be a small $\infty$-category with finite coproducts.
    Recall from \cite[Lemma 2.1]{BachmannHoyois_Norms} that the Yoneda embedding
    \begin{equation*}
        \mathscr{C}_+ \subseteq \mathscr{C}_\bullet \subseteq \PSh_\Sigma(\mathscr{C})_\bullet
    \end{equation*}
    induces an equivalence $\PSh_\Sigma(\mathscr{C}_+) \cong \PSh_\Sigma(\mathscr{C})_\bullet$.
\end{remark}

\begin{notation}
    Let $S$ be an algebraic space, and let $\theta : \Sm_{S+} \to \mathscr{C}$ be a functor that preserves finite coproducts.
    For $X \in \Sm_S$, we define
    \begin{equation*}
        h_S^\mathscr{C}(X) = \mathrm{Map}_{\mathscr{C}}(\theta({-}), \theta(X_+)) \in \PSh_\Sigma(\Sm_S)_\bullet
    \end{equation*}
    This gives a functor $h_S^\mathscr{C} : \Sm_S \to \PSh_\Sigma(\Sm_S)_\bullet$.
    We'll also write
    \begin{equation*}
        h_S^\mathscr{C} : \PSh_\Sigma(\Sm_S) \to \PSh_\Sigma(\Sm_S)_\bullet
    \end{equation*}
    for the left Kan extension of $h_S^\mathscr{C}$ along the Yoneda embedding $\Sm_S \to \PSh_\Sigma(\Sm_S)$.
\end{notation}

\begin{definition}
    A \emph{motivic pattern} $\mathscr{C} = (\mathscr{C}_\star,\theta)$ consists of
    \begin{itemize}
        \item a functor $\mathscr{C}_\star : \AlgSpc^\op \to \CAlg(\Cat_\infty)$ that preserves finite products, and
        \item a natural transformation
        \begin{equation*}
            \theta_\mathscr{C} : \Sm_{\star+} \to \mathscr{C}_\star
        \end{equation*}
        between functors $\AlgSpc^\op \to \CAlg(\Cat_\infty)$.
    \end{itemize}
    We require that these data satisfy the following axioms.
    \begin{enumerate}
        \item For every algebraic space $S$, the component $\theta: \Sm_{S+} \to \mathscr{C}_S$ is essentially surjective.
        \item For every algebraic space $S$, every algebraic space $X$ smooth over $S$, and every Nisnevich covering sieve $R \hookrightarrow X$ generated by a single map, the induced map $h_{S}^\mathscr{C}(R) \to h_{S}^\mathscr{C}(X)$ is a Nisnevich equivalence.
        \item For every smooth morphism of algebraic spaces $f : S \to T$, the functor $f^* : \mathscr{C}_T \to \mathscr{C}_S$ admits a left adjoint $f_\sharp $, and the exchange transformation
        \begin{equation*}
            f_\sharp \theta \to \theta f_\sharp : \Sm_{S+} \to \mathscr{C}_T
        \end{equation*}
        is an equivalence.
    \end{enumerate}
    We write
    \begin{equation*}
        \mathscr{M}\mathrm{ot}\mathscr{P}\mathrm{att} \subseteq \mathrm{Fun}(\AlgSpc^\mathrm{op} , \CAlg(\Cat_\infty))_{\Sm_{\star+}/}
    \end{equation*}
    for the full subcategory spanned by motivic patterns.
\end{definition}

\begin{example}
    The presheaf $\Sm_{\star+} : \AlgSpc^\op \to \CAlg(\Cat_\infty)$ is the initial motivic pattern.
\end{example}

\begin{example}
    The assignment sending $S \in \AlgSpc$ to $\Sm_{S+} \to \Corr^{\fold}(\Sm_S)$ can be promoted to a motivic pattern.
\end{example}

\begin{example}
    We will encounter a slew of motivic patterns in the following section.
    See Construction~\ref{construct:examples-are-patterns}
\end{example}

\begin{remark}
    Let $(\mathscr{C}_\star, \theta)$ be a motivic pattern, let $S$ be an algebraic space, and let $\nabla : S^{\sqcup n} \to S$ be the fold map.
    Under the equivalence $\mathscr{C}_{S^{\sqcup n}} \cong \mathscr{C}_S^{\times n}$, we can identify $\nabla_\sharp : \mathscr{C}_{S^{\sqcup n}} \to \mathscr{C}_S$ with the $n$-fold coproduct functor.
    In particular, Axiom (3) in the definition of motivic patterns implies that $\mathscr{C}_S$ admits finite coproducts, and $\theta : \Sm_{S+} \to \mathscr{C}_S$ preserves them.
\end{remark}

\begin{proposition}
    Let $(\mathscr{C}_\star, \theta)$ be a motivic pattern and let $S$ be an algebraic space.
    The tensor product in $\mathscr{C}_S$ preserves finite coproducts in each variable.
\end{proposition}

\begin{proof}
    Since $\theta : \Sm_{S+} \to \mathscr{C}_S$ is essentially surjective, symmetric monoidal, and preserves finite coproducts, this follows from the fact that finite products in $\Sm_S$ distribute over finite coproducts.
\end{proof}

\begin{proposition} \label{prop:smooth-bc-projection}
    Let $\mathscr{C} = (\mathscr{C}_\star,\theta)$ be a motivic pattern and let $p : X \to S$ be a smooth morphism of algebraic spaces.
    \begin{enumerate}
        \item For every Cartesian square of algebraic spaces
        \begin{equation*}
            \begin{tikzcd}
                Y \ar[r,"g"] \ar[d,"q"] & X \ar[d,"p"] \\
                T \ar[r,"f"] & S
            \end{tikzcd}
        \end{equation*}
        the exchange transformation
        \begin{equation*}
            \mathrm{Ex} : q_\sharp g^* \to f^* p_\sharp : \mathscr{C}_X \to \mathscr{C}_T
        \end{equation*}
        is an equivalence.

        \item The projection transformation
        \begin{equation*}
            \mathrm{Proj} : p_\sharp ({-} \otimes p^*({-})) \to p_\sharp ({-}) \otimes ({-}) : \mathscr{C}_X \times \mathscr{C}_S \to \mathscr{C}_S
        \end{equation*}
        is an equivalence.
    \end{enumerate}
\end{proposition}

\begin{proof}
    First we prove (1).
    Since $\theta : \Sm_{X+} \to \mathscr{C}_X$ is essentially surjective, it suffices to show that
    \begin{equation*}
        \mathrm{Ex} \theta : q_\sharp g^* \theta \to f^* p_\sharp \theta : \Sm_{X+} \to \mathscr{C}_T
    \end{equation*}
    is an equivalence.
    However, this is equivalent to $\theta$ applied to the exchange transformation
    \begin{equation*}
        \mathrm{Ex} : q_\sharp g^* \to f^* p_\sharp : \Sm_{X+} \to \mathscr{Sm}_{T+}
    \end{equation*}
    which is an equivalence.

    Now we prove (2).
    Since  $\theta : \Sm_{X+} \to \mathscr{C}_X$ and $\theta : \Sm_{S+} \to \mathscr{C}_S$ are essentially surjective, it suffices to show that
    \begin{equation*}
        \mathrm{Proj}(\theta,\theta) : p_\sharp (\theta({-}) \otimes p^*\theta({-})) \to p_\sharp \theta({-}) \otimes \theta({-}) : \Sm_{X+} \times \Sm_{S+} \to \mathscr{C}_S
    \end{equation*}
    is an equivalence.
    However, this is equivalent to $\theta$ applied to the projection transformation
    \begin{equation*}
        \mathrm{Proj} : p_\sharp ({-} \otimes p^*({-})) \to p_\sharp ({-}) \otimes ({-}) : \Sm_{X+} \times \Sm_{S+} \to \Sm_{S+}
    \end{equation*}
    which is an equivalence.
\end{proof}

\begin{proposition} \label{prop:p-lower-sharp-alpha-lower-shriek}
    Let $\alpha : (\mathscr{C}_\star,\theta) \to (\mathscr{D}_\star,\rho)$ be a morphism of motivic patterns, and let $p : X \to S$ be a smooth morphism of algebraic spaces.
    The exchange transformation
    \begin{equation*}
        \mathrm{Ex} : p_\sharp \alpha \to \alpha p_\sharp : \mathscr{C}_X \to \mathscr{D}_S
    \end{equation*}
    is an equivalence.
\end{proposition}

\begin{proof}
    Since $\theta : \Sm_{X+} \to \mathscr{C}_X$ is essentially surjective, it suffices to show that
    \begin{equation*}
        \mathrm{Ex} \theta : p_\sharp \alpha \theta \to \alpha p_\sharp \theta : \Sm_{X+} \to \mathscr{D}_S
    \end{equation*}
    is an equivalence.
    However, this is equivalent to the exchange transformation
    \begin{equation*}
        \mathrm{Ex} : p_\sharp \rho \to \rho p_\sharp : \Sm_{X+} \to \mathscr{D}_S
    \end{equation*}
    which is an equivalence by definition.
\end{proof}

\begin{construction}
    Fix a motivic pattern $\mathscr{C} = (\mathscr{C}_\star, \theta)$.

    Let $S$ be an algebraic space.
    Since the tensor product in $\mathscr{C}_S$ preserves finite coproducts in each variable, we have that $\PSh_\Sigma(\mathscr{C}_S)$ is presentably symmetric monoidal.
    The Yoneda embedding restricts to a fully faithful functor $\mathscr{C}_S \to \PSh_\Sigma(\mathscr{C}_S$$)$ that is symmetric monoidal and preserves finite coproducts.

    The functor $\theta : \Sm_{S+} \to \mathscr{C}_S$ extends uniquely to symmetric monoidal functor $\theta_! : \PSh_\Sigma(\Sm_S)_\bullet \to \PSh_\Sigma(\mathscr{C}_S)$ that preserves sifted colimits.
    Since $\theta_! : \PSh_\Sigma(\Sm_S)_\bullet \to \PSh_\Sigma(\mathscr{C}_S)$ preserves finite coproducts as well, it preserves all small colimits, and thus admits a right adjoint $\theta^*$.
    This right adjoint is given simply by precomposition with $\theta : \Sm_S \to \mathscr{C}_S$.

    For a map $f : S \to T$ of algebraic spaces, the functor $f^* : \mathscr{C}_T \to \mathscr{C}_S$ extends uniquely to a symmetric monoidal functor $f^* : \PSh_\Sigma(\mathscr{C}_T) \to \PSh_\Sigma(\mathscr{C}_S)$ that preserves sifted colimits.
    Since $f^* : \PSh_\Sigma(\mathscr{C}_T) \to \PSh_\Sigma(\mathscr{C}_S)$ preserves finite coproducts as well, it preserves all small colimits, and thus admits a right adjoint $f_*$.
    This right adjoint is given simply by precomposition with $f^* : \mathscr{C}_T \to \mathscr{C}_S$.
    Note that we can form a commutative diagram
    \begin{equation*}
        \begin{tikzcd}
            \PSh_\Sigma(\Sm_T)_\bullet \ar[r,"f^*"] \ar[d,"\theta_!"] & \PSh_\Sigma(\Sm_S)_\bullet \ar[d,"\theta_!"] \\
            \PSh_\Sigma(\mathscr{C}_T) \ar[r,"f^*"] & \PSh_\Sigma(\mathscr{C}_S)
        \end{tikzcd}
    \end{equation*}
    of symmetric monoidal $\infty$-categories and symmetric monoidal functors.

    If $f : S \to T$ is a smooth map of algebraic spaces, the functor $f_\sharp : \mathscr{C}_S \to \mathscr{C}_T$ extends uniquely to a functor $f_\sharp : \PSh_\Sigma(\mathscr{C}_S) \to \PSh_\Sigma(\mathscr{C}_T)$.
    In fact, $f_\sharp : \PSh_\Sigma(\mathscr{C}_S) \to \PSh_\Sigma(\mathscr{C}_T)$ is left adjoint to $f^*:\PSh_\Sigma(\mathscr{C}_T) \to \PSh_\Sigma(\mathscr{C}_S)$.
    Note that we can form a commutative diagram
    \begin{equation*}
        \begin{tikzcd}
            \PSh_\Sigma(\Sm_S)_\bullet \ar[r,"f_\sharp"] \ar[d,"\theta_!"] & \PSh_\Sigma(\Sm_T)_\bullet \ar[d,"\theta_!"] \\
            \PSh_\Sigma(\mathscr{C}_S) \ar[r,"f_\sharp"] & \PSh_\Sigma(\mathscr{C}_T)
        \end{tikzcd}
    \end{equation*}
    of $\infty$-categories.
\end{construction}

\begin{proposition} \label{prop:smooth-bc-projection-p-sigma}
    Fix a motivic pattern $\mathscr{C} = (\mathscr{C}_\star, \theta)$ and smooth morphism $p : X \to S$ of algebraic spaces.
    \begin{enumerate}
        \item For every Cartesian square
        \begin{equation*}
            \begin{tikzcd}
                Y \ar[r,"g"] \ar[d,"q"] & X \ar[d,"p"] \\
                T \ar[r,"f"] & S
            \end{tikzcd}
        \end{equation*}
        the exchange transformation $q_\sharp g^* \to f^* p_\sharp : \PSh_\Sigma(\mathscr{C}_X) \to \PSh_\Sigma(\mathscr{C}_T)$ is an equivalence
        \item The projection transformation
        \begin{equation*}
            p_\sharp (p^*(-) \otimes (-)) \to (-) \otimes p_\sharp (-) : \PSh_\Sigma(\mathscr{C}_S) \times \PSh_\Sigma(\mathscr{C}_X) \to \PSh_\Sigma(\mathscr{C}_S)
        \end{equation*}
        is an equivalence.
    \end{enumerate}
\end{proposition}

\begin{proof}
    Since all the functors involved commute with small colimits in each variable, the claim reduces to Proposition~\ref{prop:smooth-bc-projection}.
\end{proof}

\begin{construction}
    Let $\alpha : (\mathscr{C}, \theta) \to (\mathscr{D}, \rho)$ be a morphism of motivic patterns, and let $S$ be an algebraic space.
    The functor $\alpha : \mathscr{C}_S \to \mathscr{D}_S$ extends uniquely to a symmetric monoidal functor $\alpha_! : \PSh_\Sigma(\mathscr{C}_S) \to \PSh_\Sigma(\mathscr{D}_S)$ that preserves sifted colimits.
    Since $\alpha_! : \PSh_\Sigma(\mathscr{C}_S) \to \PSh_\Sigma(\mathscr{D}_S)$ preserves finite coproducts as well, it preserves all small colimits, and thus admits a right adjoint $\alpha_*$, which is given by precomposition by $\alpha : \mathscr{C}_S \to \mathscr{D}_S$.
\end{construction}

\begin{proposition} \label{prop:p-lower-sharp-alpha-lower-shriek-p-sigma}
    Let $\alpha : (\mathscr{C}_\star,\theta) \to (\mathscr{D}_\star,\rho)$ be a morphism of motivic patterns, and let $p : X \to S$ be a smooth morphism of algebraic spaces.
    The exchange transformation
    \begin{equation*}
        p_\sharp \alpha_! \to \alpha_! p_\sharp : \PSh_\Sigma(\mathscr{C}_X) \to \PSh_\Sigma(\mathscr{D}_S)
    \end{equation*}
    is an equivalence.
\end{proposition}

\begin{proof}
    Since all the functors involved commute with small colimits, the claim reduces to Proposition~\ref{prop:p-lower-sharp-alpha-lower-shriek}
\end{proof}

\subsection{Unstable Homotopy Theory for Motivic Patterns}

\begin{definition}
    Let $\mathscr{C}=(\mathscr{C}_\star,\theta)$ be a motivic pattern, and let $S$ be an algebraic space.
    We say $\mathscr{F} \in \PSh_\Sigma(\mathscr{C}_S)$ is \emph{Nisnevich-local} (\emph{$\mathbb{A}^1$-local}, \emph{motivic-local}, respecitvely) if its restriction along $\theta : \Sm_S \to \mathscr{C}_S$ is so.
    We write $\mathscr{H}^\mathscr{C}(S) \subseteq \PSh_\Sigma(\mathscr{C}_S)$ for the full subcategory spanned by motivic-local objects, and refer to those objects as \emph{motivic spaces (over $S$) with $\mathscr{C}$-structure}.
\end{definition}

\begin{proposition} \label{prop:structured-motivic-spaces-accessible}
    Let $\mathscr{C}=(\mathscr{C}_\star,\theta)$ be a motivic pattern, and let $S$ be an algebraic space.
    The full subcategory $\HH^\mathscr{C}(S) \subseteq \PSh_\Sigma(\mathscr{C}_S)$ is closed under small limits and filtered colimits.
    In particular, it is an accessible localization.
\end{proposition}

\begin{proof}
    This follows from the fact that the property of being motivic-local can be expressed by a small set conditions involving finite limits.
\end{proof}

\begin{definition}
    Let $\mathscr{C}=(\mathscr{C}_\star,\theta)$ be a motivic pattern, and let $S$ be an algebraic space.
    We write $\mathrm{L}_\mathrm{mot} : \PSh_\Sigma(\mathscr{C}_S) \to \HH^\mathscr{C}(S)$ for the left adjoint to the inclusion.
    A morphism in $\PSh_\Sigma(\mathscr{C}_S)$ is called a \emph{motivic equivalence} if its image under $\mathrm{L}_\mathrm{mot}$ is an equivalence.
\end{definition}

\begin{proposition}
    Let $\mathscr{C}=(\mathscr{C}_\star,\theta)$ be a motivic pattern, and let $S$ be an algebraic space.
    \begin{enumerate}
        \item The $\infty$-category $\HH^\mathscr{C}(S)$ is generated under sifted colimits by objects of the form $\mathrm{L}_\mathrm{mot}\theta(X_+)$, where $X \in \Sm_S$ is affine.
        \item For $X \in \Sm_S$, the object $\mathrm{L}_\mathrm{mot}\theta(X_+) \in \HH^\mathscr{C}(S)$ is compact.
        \item The $\infty$-category $\HH^\mathscr{C}(S)$ admits a unique symmetric monoidal structure such that $\mathrm{L}_\mathrm{mot} : \PSh_\Sigma(\mathscr{C}_S) \to \HH^\mathscr{C}(S)$ is symmetric monoidal.
    \end{enumerate}
\end{proposition}

\begin{proof}
    Statement (1) follows by the same argument as in \cite[Proposition~3.16(i)]{Hoyois_6ff}.
    Statement (2) follows immediately from Proposition~\ref{prop:structured-motivic-spaces-accessible}. Statement (3) follows by the same argument as in \cite[Proposition~3.2.10(iv)]{EHKSY_MotivicInfiniteLoops}.
\end{proof}

\begin{proposition} \label{prop:gamma-detects-mot-equiv}
    Let $\mathscr{C} = (\mathscr{C}_\star,\theta)$ be a motivic pattern and let $S$ be an algebraic space.
    \begin{enumerate}
        \item The functor $\theta^* : \PSh_\Sigma(\mathscr{C}_S) \to \PSh_\Sigma(\Sm_S)_\point$ preserves and detects motivic equivalences.
        \item The functor $\theta^* : \HH^\mathscr{C}(S) \to \HH(S)_\bullet$ is conservative and preserves sifted colimits.
    \end{enumerate}
\end{proposition}

\begin{proof}
    The statements follow by the same arguments as in \cite[Proposition~3.2.14]{EHKSY_MotivicInfiniteLoops} and \cite[Proposition~3.2.15]{EHKSY_MotivicInfiniteLoops}.
\end{proof}

\begin{proposition} \label{prop:basic-functoriality-p-sigma}
    Let $\mathscr{C} = (\mathscr{C}_\star,\theta)$ be a motivic pattern, and let $f : S \to T$ be a morphism of algebraic spaces.
    The functor $f_* : \PSh_\Sigma(\mathscr{C}_S) \to \PSh_\Sigma(\mathscr{C}_T)$ restricts to a functor $f_* : \HH^\mathscr{C}(S) \to \HH^\mathscr{C}(T)$.
    If $f : S \to T$ is smooth, then the functor $f^* : \PSh_\Sigma(\mathscr{C}_T) \to \PSh_\Sigma(\mathscr{C}_S)$ restricts to a functor $f^* : \HH^\mathscr{C}(T) \to \HH^\mathscr{C}(S)$
\end{proposition}

\begin{proof}
    Suppose $X \in \PSh_\Sigma(\mathscr{C}_S)$ is a motivic-local object.
    It suffices to show that the object $\theta^* f_* X \in \PSh_\Sigma(\Sm_T)_\bullet$ is motivic-local.
    This follows from the equivalence $\theta^* f_* \simeq f_* \theta^*$ of functors $\PSh_\Sigma(\mathscr{C}_S) \to \PSh_\Sigma(\Sm_T)$ and the fact that $f_* : \PSh_\Sigma(\Sm_S)_\bullet \to \PSh_\Sigma(\Sm_T)_\bullet$ preserves motivic-local objects.

    Suppose now that $f$ is smooth and $Y \in \PSh_\Sigma(\mathscr{C}_T)$ is a motivic-local object.
    It suffices to show that the object $\theta^* f^* Y \in \PSh_\Sigma(\Sm_S)_\bullet$ is motivic-local.
    Since $f$ is smooth, this follows from the equivalence $f^* \theta^* \simeq \theta^* f^*$ of functors $\PSh_\Sigma(\mathscr{C}_T) \to \PSh_\Sigma(\Sm_S)$ and the fact that $f^* : \PSh_\Sigma(\Sm_T)_\bullet \to \PSh_\Sigma(\Sm_S)_\bullet$ preserves motivic-local objects.
\end{proof}

\begin{construction}
    Fix a motivic pattern $\mathscr{C} = (\mathscr{C}_\star,\theta)$ and a morphism $f : S \to T$ of algebraic spaces.

    As a consequence of Proposition~\ref{prop:basic-functoriality-p-sigma}, the functor $f^* : \PSh_\Sigma(\mathscr{C}_T) \to \PSh_\Sigma(\mathscr{C}_S)$ preserves motivic equivalences, and so the composite $\mathrm{L}_\mathrm{mot} f^* : \PSh_\Sigma(\mathscr{C}_T) \to \HH^\mathscr{C}(S)$ extends uniquely to a symmetric monoidal functor $f^* : \HH^\mathscr{C}(T) \to \HH^\mathscr{C}(S)$.
    The functor $f^* : \HH^\mathscr{C}(T) \to \HH^\mathscr{C}(S)$ is left adjoint to $f_* : \HH^\mathscr{C}(S) \to \HH^\mathscr{C}(T)$.

    If $f : S \to T$ is smooth, then again as a consequence of Proposition~\ref{prop:basic-functoriality-p-sigma}, the functor $f_\sharp : \PSh_\Sigma(\mathscr{C}_S) \to \PSh_\Sigma(\mathscr{C}_T)$ preserves motivic equivalences, and so the composite $\mathrm{L}_\mathrm{mot} f_\sharp : \PSh_\Sigma(\mathscr{C}_S) \to \HH^\mathscr{C}(T)$ extends uniquely to a functor $f_\sharp : \HH^\mathscr{C}(S) \to \HH^\mathscr{C}(T)$.
    The functor $f_\sharp : \HH^\mathscr{C}(S) \to \HH^\mathscr{C}(T)$ is left adjoint to $f^* : \HH^\mathscr{C}(T) \to \HH^\mathscr{C}(S)$.
\end{construction}

\begin{proposition} \label{prop:unstable-smooth-bc-projection}
    Fix a motivic pattern $\mathscr{C} = (\mathscr{C}_\star, \theta)$ and smooth morphism $p : X \to S$ of algebraic spaces.
    \begin{enumerate}
        \item For every Cartesian square
        \begin{equation*}
            \begin{tikzcd}
                Y \ar[r,"g"] \ar[d,"q"] & X \ar[d,"p"] \\
                T \ar[r,"f"] & S
            \end{tikzcd}
        \end{equation*}
        the exchange transformation $q_\sharp g^* \to f^* p_\sharp : \HH^\mathscr{C}(X) \to \HH^\mathscr{C}(T)$ is an equivalence
        \item The projection transformation
        \begin{equation*}
            p_\sharp (p^*(-) \otimes (-)) \to (-) \otimes p_\sharp (-) : \HH^\mathscr{C}(S) \times \HH^\mathscr{C}(X) \to \HH^\mathscr{C}(S)
        \end{equation*}
        is an equivalence.
    \end{enumerate}
\end{proposition}

\begin{proof}
    Since all the functors involved preserve motivic equivalences, the claim reduces to Proposition~\ref{prop:smooth-bc-projection-p-sigma}.
\end{proof}

\begin{proposition} \label{prop:unstable-nis-sep}
    Fix a motivic pattern $\mathscr{C} = (\mathscr{C}_\star,\theta)$.
    For every Nisnevich cover $\{j_\alpha : U_\alpha \to S \}$, the family $\{j^*_\alpha : \HH^\mathscr{C}(S) \to \HH^\mathscr{C}(U_\alpha) \}$ is conservative.
\end{proposition}

\begin{proof}
    This follows from conservativity of $\theta^* : \HH^\mathscr{C} \to \HH_\bullet$ and the analogous result for $\HH$ \cite[Proposition~4.5]{Hoyois_6ff}.
\end{proof}

\begin{proposition} \label{prop:alpha-upper-star-preserves-motivic-spaces}
    Let $\alpha : (\mathscr{C}_\star,\theta) \to (\mathscr{D}_\star,\rho)$ be a morphism of motivic pattern, and let $S$ and algebraic space. The functor $\alpha^* : \PSh_\Sigma(\mathscr{D}_S) \to \PSh_\Sigma(\mathscr{C}_S)$ restricts to a functor $\alpha^* : \HH^\mathscr{D}(S) \to \HH^\mathscr{C}(S)$.
\end{proposition}

\begin{proof}
    This follows from the equivalence $\theta^* \alpha^* \simeq \alpha^* \theta^*$.
\end{proof}

\begin{construction}
    Let $\alpha : (\mathscr{C}_\star,\theta) \to (\mathscr{D}_\star,\rho)$ be a morphism of motivic pattern, and let $S$ be an algebraic space.
    By Proposition~\ref{prop:alpha-upper-star-preserves-motivic-spaces}, the functor $\alpha_! : \PSh_\Sigma(\mathscr{C}_S) \to \PSh_\Sigma(\mathscr{D}_S)$ preserves motivic equivalences, and so the composite $\mathrm{L}_\mathrm{mot} \alpha_! : \PSh_\Sigma(\mathscr{C}_S) \to \HH^\mathscr{D}(S)$ extends to a functor $\alpha_! : \HH^\mathscr{C}(S) \to \HH^\mathscr{D}(S)$ left adjoint to $\alpha^*$.
\end{construction}

\begin{proposition} \label{prop:unstable-p-lower-sharp-alpha-lower-shriek}
    Let $\alpha : (\mathscr{C}_\star,\theta) \to (\mathscr{D}_\star,\rho)$ be a morphism of motivic pattern, and let $p : X \to S$ be a smooth morphism of algebraic spaces.
    The exchange transformation $p_\sharp \alpha_! \to \alpha_! p_\sharp : \HH^\mathscr{C}(X) \to \HH^\mathscr{D}(S)$ is an equivalence.
\end{proposition}

\begin{proof}
    Since all the functors involved preserve motivic equivalences, the claim reduces to Proposition~\ref{prop:p-lower-sharp-alpha-lower-shriek-p-sigma}.
\end{proof}

\subsection{Stable Homotopy Theory for Motivic Patterns}

\begin{definition}
    Let $\mathscr{C} = (\mathscr{C}_\star,\theta)$ be a motivic pattern, and let $S$ be an algebraic space.
    We write $\mathbf{T}^\mathscr{C}_S = \theta_!(\mathbb{A}^1_S/\mathbb{A}^1_S \smallsetminus 0)$ and $\Sigma_{\mathbf{T},\mathscr{C}} = \mathbf{T}^\mathscr{C}_S \otimes ({-}) : \HH^\mathscr{C}(S) \to \HH^\mathscr{C}(S)$.
\end{definition}

\begin{construction}
    Let $\mathscr{C} = (\mathscr{C}_\star,\theta)$ be a motivic pattern and let $S$ be an algebraic space.
    We write $\mathscr{SH}^\mathscr{C}(S)$ for the presentably symmetric monoidal $\infty$-category obtained from $\HH^\mathscr{C}(S)$ by formally inverting $\mathbf{T}^\mathscr{C}_S$.
    Objects of $\SH^\mathscr{C}(S)$ are called \emph{motivic spectra (over $S$) with $\mathscr{C}$-structure}.
    This comes with a symmetric monoidal functor $\Sigma^\infty_{\mathbf{T},\mathscr{C}} : \mathscr{H}^\mathscr{C}(S) \to \mathscr{SH}^\mathscr{C}(S)$ that admits a right adjoint $\Omega^\infty_{\mathbf{T},\mathscr{C}}$.
    Since the functor $\theta_! : \HH(S)_\bullet \to \HH^\mathscr{C}(S)$ sends $\mathbf{T}_S$ to $\mathbf{T}^\mathscr{C}_S$, it extends uniquely to a symmetric monoidal functor $\theta_! : \SH(S) \to \SH^\mathscr{C}(S)$ that preserves all small colimits.
    We write $\theta^*$ for the right adjoint of $\theta_! : \SH(S) \to \SH^\mathscr{C}(S)$.
\end{construction}

\begin{proposition}
    Let $\mathscr{C} = (\mathscr{C}_\star,\theta)$ be a motivic pattern, and let $S$ be an algebraic space.
    The cyclic permutation on $(\mathbf{T}^\mathscr{C}_S)^{\otimes 3} \in \HH^\mathscr{C}(S)$ is homotopic to the identity.
    In particular, there is an equivalence
    \begin{equation*}
        \SH^\mathscr{C}(S) \cong \operatorname{colim} \Big( \HH^\mathscr{C}(S) \xrightarrow{\Sigma_{\mathbf{T},\mathscr{C}}} \HH^\mathscr{C}(S) \xrightarrow{\Sigma_{\mathbf{T},\mathscr{C}}} \cdots \Big)
    \end{equation*}
    where the colimit is taken in $\CAlg(\mathrm{Pr}^\mathrm{L})$.
\end{proposition}

\begin{proof}
    Since $\theta_! : \HH(S)_\bullet \to \HH^\mathscr{C}(S)$ is symmetric monoidal, the first claim follows from the analgous statement for $\mathbf{T}_S \in \HH(S)_\bullet$, which is \cite[Lemma~6.3]{Hoyois_6ff}.
    The second claim is then \cite[Corollary 2.22]{Robalo_KTheoryBridge}
\end{proof}

\begin{proposition} \label{prop:SH-C-is-basechanged-from-SH}
    Let $\mathscr{C} = (\mathscr{C}_\star,\theta)$ be a motivic pattern and let $S$ be an algebraic space.
    The functor
    \begin{equation*}
        \SH(S) \otimes_{\HH(S)_\bullet} \HH^\mathscr{C}(S) \to \SH^\mathscr{C}(S)
    \end{equation*}
    obtained by base change of $\theta_! : \HH(S)_\bullet \to \HH^\mathscr{C}(S)$ along $\Sigma^\infty_\mathbf{T} : \HH(S)_\bullet \to \SH(S)$ is an equivalence.
\end{proposition}

\begin{proof}
    This follows directly from comparing universal properties.
\end{proof}

\begin{proposition}
    Let $\mathscr{C}$ be a motivic pattern, and let $S$ be an algebraic space.
    \begin{enumerate}
        \item The $\infty$-category $\SH^\mathscr{C}(S)$ is stable.
        \item The $\infty$-category $\SH^\mathscr{C}(S)$ is generated under sifted colimits by objects of the form $(\mathbf{T}^\mathscr{C}_S)^{\otimes n} \otimes \Sigma^\infty_{\mathbf{T},\mathscr{C}} \theta(X_+)$, where $X \in \Sm_S$ is affine and $n \leq 0$.
        \item For $X \in \Sm_S$, the object $\Sigma^\infty_{\mathbf{T},\mathscr{C}} \theta(X_+)$ is compact.
    \end{enumerate}
\end{proposition}

\begin{proof}
    This follows from the same argument as in \cite[Proposition~3.3.5]{EHKSY_MotivicInfiniteLoops}.
\end{proof}

\begin{proposition}
    Let $\mathscr{C}=(\mathscr{C}_\star,\theta)$ be a motivic pattern, and let $S$ be an algebraic space.
    The functor $\theta^* : \SH^\mathscr{C}(S) \to \SH(S)$ is conservative and preserves all small colimits.
\end{proposition}

\begin{proof}
    This follows from the same argument as in \cite[Proposition~3.5.2]{EHKSY_MotivicInfiniteLoops}.
\end{proof}

\begin{construction}
    Fix a motivic pattern $\mathscr{C} = (\mathscr{C}_\star,\theta)$ and a morphism $f : S \to T$ of algebraic spaces.

    Since the functor $f^* : \HH^\mathscr{C}(T) \to \HH^\mathscr{C}(S)$ is symmetric monoidal, it is a functor of $\HH(T)_\bullet$-modules.
    Base change along $\Sigma^\infty_{\mathbf{T}} : \HH(T)_\bullet \to \SH(T)$ then yields a symmetric monoidal functor $f^* : \SH^\mathscr{C}(T) \to \SH^\mathscr{C}(S)$ of $\SH(T)$-modules.
    The functor $f^* : \SH^\mathscr{C}(T) \to \SH^\mathscr{C}(S)$ admits a right adjoint $f_*$.

    If $f : S \to T$ is smooth, then Proposition~\ref{prop:unstable-smooth-bc-projection}(2) tells us that $f_\sharp : \HH^\mathscr{C}(S) \to \HH^\mathscr{C}(T)$ is also a functor of $\HH(T)_\bullet$-modules.
    Base change along $\Sigma^\infty_{\mathbf{T}} : \HH(T)_\bullet \to \SH(T)$ then yields a functor $f_\sharp : \SH^\mathscr{C}(T) \to \SH^\mathscr{C}(S)$ of $\SH(T)$-modules that is left adjoint to $f^*$.
\end{construction}

\begin{proposition} \label{prop:stable-smooth-bc-projection}
    Fix a motivic pattern $\mathscr{C} = (\mathscr{C}_\star, \theta)$ and smooth morphism $p : X \to S$ of algebraic spaces.
    \begin{enumerate}
        \item For every Cartesian square
        \begin{equation*}
            \begin{tikzcd}
                Y \ar[r,"g"] \ar[d,"q"] & X \ar[d,"p"] \\
                T \ar[r,"f"] & S
            \end{tikzcd}
        \end{equation*}
        the exchange transformation $q_\sharp g^* \to f^* p_\sharp : \SH^\mathscr{C}(X) \to \SH^\mathscr{C}(T)$ is an equivalence
        \item The projection transformation
        \begin{equation*}
            p_\sharp (p^*(-) \otimes (-)) \to (-) \otimes p_\sharp (-) : \SH^\mathscr{C}(S) \times \SH^\mathscr{C}(X) \to \SH^\mathscr{C}(S)
        \end{equation*}
        is an equivalence.
    \end{enumerate}
\end{proposition}

\begin{proof}
    These follow from Proposition~\ref{prop:SH-C-is-basechanged-from-SH}, the functoriality of base change along $\HH_\bullet \to \SH$, and Proposition~\ref{prop:unstable-smooth-bc-projection}.
    See the paragraph before \cite[Proposition~6.5]{Hoyois_6ff}.
\end{proof}

\begin{proposition} \label{prop:stable-nis-sep}
    Fix a motivic pattern $\mathscr{C} = (\mathscr{C}_\star,\theta)$.
    For every Nisnevich cover $\{j_\alpha : U_\alpha \to S \}$, the family $\{j^*_\alpha : \SH^\mathscr{C}(S) \to \SH^\mathscr{C}(U_\alpha) \}$ is conservative.
\end{proposition}

\begin{proof}
    This follows from the conservativity of $\theta^* : \SH^\mathscr{C} \to \SH$ and the analogous result for $\SH$.
    See the paragraph before \cite[Proposition~6.5]{Hoyois_6ff}.
\end{proof}

\begin{proposition} \label{prop:stable-theta-lower-shriek-i-lower-star}
    Let $\mathscr{C} = (\mathscr{C}_\star,\theta)$ be a motivic pattern, and let $i : Z \to S$ be a closed immersion with a smooth retraction $p : S \to Z$.
    The exchange transformation $\theta_! i_* \to i_* \theta_! : \SH(Z) \to \SH^\mathscr{C}(S)$ is an equivalence.
\end{proposition}

\begin{proof}
    We adapt the proof from \cite[Lemma~6.3.13]{CisinskiDeglise_TriangulatedMixedMotives}.

    Let $j : U \to S$ denote the open complement of $i$.
    Fix an object $X \in \SH(Z)$.
    We can form the commutative diagram
    \begin{equation*}
        \begin{tikzcd}
            j_\sharp j^* p^* \theta_! X \ar[r] \ar[d] & p^* \theta_! X \ar[r] \ar[d] & i_* \theta_! X \ar[d] \\
            \theta_! j_\sharp j^* p^* X \ar[r] & \theta_! p^* X \ar[r] & \theta_! i_* X
        \end{tikzcd}
    \end{equation*}
    in $\SH^\mathscr{C}(S)$, where the vertical maps are exchange transformations, the top row is the localization sequence of $\SH^\mathscr{C}$, and the bottom row is $\theta_!$ applied to the localization sequence of $\SH$.
    In particular, the left and middle vertical maps are equivalences, and the bottom row is a cofiber sequence.
    Applying $\theta^*$ to the top row yields the localization sequence
    \begin{equation*}
        j_\sharp j^* p^* \theta^* \theta_! X \to p^* \theta^* \theta_! X \to i_* \theta^* \theta_! X
    \end{equation*}
    for $\SH$ applied to $p^* \theta^* \theta_! X$.
    By exactness and conservativity of $\theta^*$, we deduce that the top row is also a cofiber sequence.
    This yields the desired claim.
\end{proof}

\begin{proposition} \label{prop:stable-weak-loc}
    Fix a motivic pattern $\mathscr{C} = (\mathscr{C}_\star,\theta)$.
    Let $i : Z \to S$ be a closed immersion of algebraic spaces with open complement $j : U \to S$.
    If $i$ admits a smooth retraction $p : S \to Z$, then
    \begin{enumerate}
        \item the functor $i_* : \SH^\mathscr{C}(Z) \to \SH^\mathscr{C}(S)$ is conservative, and
        \item the sequence
        \begin{equation*}
            j_\sharp j^* \to \mathrm{id} \to i_* i^*
        \end{equation*}
        of endofunctors of $\SH^\mathscr{C}(S)$ is a cofiber sequence.
    \end{enumerate}
\end{proposition}

\begin{proof}
    Statement (1) follows from the conservativity of $\theta^*$.

    We now prove statement (2).
    Since all functors commute with sifted colimits and tensoring with $\mathbf{T}^\mathscr{C}$, it suffices to prove the sequence
    \begin{equation*}
        j_\sharp j^* \theta_! \Sigma^\infty_\mathbf{T} X_+ \to \theta_! \Sigma^\infty_\mathbf{T} X_+ \to i_* i^* \theta_! \Sigma^\infty_\mathbf{T} X_+
    \end{equation*}
    is a cofiber sequence.
    This now follows from Proposition~\ref{prop:stable-theta-lower-shriek-i-lower-star} and the localization theorem for $\SH$.
\end{proof}

\begin{construction}
    Let $\alpha : (\mathscr{C}_\star,\theta) \to (\mathscr{D}_\star,\rho)$ be a morphism of motivic patterns, and let $S$ be an algebraic space.
    Since $\alpha_! : \HH^\mathscr{C}(S) \to \HH^\mathscr{D}(S)$ is symmetric monoidal, it is a functor of $\HH(S)_\bullet$-modules.
    Base change along $\Sigma^\infty_\mathbf{T} : \HH(S)_\bullet \to \SH(S)$ then yields a symmetric monoidal functor $\alpha_! : \SH^\mathscr{C}(S) \to \SH^\mathscr{D}(S)$ of $\SH(S)$-modules.
    The functor $\alpha_! : \SH^\mathscr{C}(S) \to \SH^\mathscr{D}(S)$ admits a right adjoint $\alpha^*$.
\end{construction}

\begin{proposition} \label{prop:stable-p-lower-sharp-alpha-lower-shriek}
    Let $\alpha : (\mathscr{C}_\star,\theta) \to (\mathscr{D},\rho)$ be a morphism of motivic patterns, and let $p : X \to S$ be a smooth morphism of algebraic spaces.
    The exchange transformation
    \begin{equation*}
        p_\sharp \alpha_! \to \alpha_! p_\sharp : \SH^\mathscr{C}(X) \to \SH^\mathscr{D}(S)
    \end{equation*}
    is an equivalence.
\end{proposition}

\begin{proof}
    Since all functors involved commute with sifted colimits, tensoring with $\mathbf{T}$, and $\Sigma^\infty_\mathbf{T}$, we are reduced to Proposition~\ref{prop:unstable-p-lower-sharp-alpha-lower-shriek}
\end{proof}

\subsection{Functoriality}

\begin{construction} \label{construct:motivic-pattern-to-formalism}
    Recall that the forgetful functor $\CAlg(\Cat^\mathrm{sift}_\infty) \to \CAlg(\Cat_\infty)$ admits a left adjoint, which we denote by $\PSh_\Sigma$.
    We'll also write $\PSh_\Sigma$ for the composite
    \begin{equation*}
        \MotPatt \times \AlgSpc^\op \to \CAlg(\Cat_\infty) \xrightarrow{\PSh_\Sigma} \CAlg(\Cat^\mathrm{sift}_\infty)
    \end{equation*}
    where the first arrow is $(\mathscr{C}, S) \mapsto \mathscr{C}_S$.

    Since the functor $f^* : \PSh_\Sigma(\mathscr{C}_T) \to \PSh_\Sigma(\mathscr{C}_S)$ preserves motivic equivalences for every motivic pattern $\mathscr{C}$ and every morphism $f : S \to T$ of algebraic spaces, and motivic equivalences are stable under tensor product, we can construct a lift
    \begin{equation*}
        \begin{tikzcd}
            & \CAlg(\mathscr{M}\Cat^\mathrm{sift}_\infty) \ar[d] \\
            \MotPatt \times \AlgSpc^\op \ar[ur,dashed] \ar[r,"\PSh_\Sigma"] & \CAlg(\Cat^\mathrm{sift}_\infty)
        \end{tikzcd}
    \end{equation*}
    which sends $(\mathscr{C},S)$ to $(\PSh_\Sigma(\mathscr{C}_S), \text{motivic equivalences})$.
    The essential image of the lift is in the domain of the partial left adjoint of the functor
    \begin{equation*}
        \CAlg(\Cat^\mathrm{sift}_\infty) \to \CAlg(\mathscr{M}\Cat^\mathrm{sift}_\infty), \qquad C \mapsto (C, \text{equivalences}),
    \end{equation*}
    so we obtain a functor
    \begin{align*}
        \HH^{({-})} : \MotPatt \times \AlgSpc^\op &\to \CAlg(\Cat^\mathrm{sift}_\infty) \\
        (\mathscr{C}, S) &\mapsto \HH^\mathscr{C}(S).
    \end{align*}
    together with a natural transformation $\PSh_\Sigma \to \HH^{({-})}$.

    Since the functor $f^* : \HH^\mathscr{C}(T) \to \HH^\mathscr{C}(S)$ sends $\mathbf{T}^\mathscr{C}_T$ to $\mathbf{T}^\mathscr{C}_S$ for every motivic pattern $\mathscr{C}$ and every morphism $f : S \to T$ of algebraic spaces, we can construct a lift
    \begin{equation*}
        \begin{tikzcd}
            & \CAlg(\mathscr{O}\Cat^\mathrm{sift}_\infty) \ar[d] \\
            \MotPatt \times \AlgSpc^\op \ar[ur,dashed] \ar[r,"\HH^{({-})}"] & \CAlg(\Cat^\mathrm{sift}_\infty)
        \end{tikzcd}
    \end{equation*}
    which sends $(\mathscr{C},S)$ to $(\HH^\mathscr{C}(S), \{\mathbf{T}^\mathscr{C}_S\})$.
    The essential image of the lift is in the domain of the partial left adjoint of the functor
    \begin{equation*}
        \CAlg(\Cat^\mathrm{sift}_\infty) \to \CAlg(\mathscr{O}\Cat^\mathrm{sift}_\infty), \qquad C \mapsto (C, \text{invertible objects}),
    \end{equation*}
    so we obtain a functor
    \begin{align*}
        \SH^{({-})} : \MotPatt \times \AlgSpc^\op &\to \CAlg(\Cat^\mathrm{sift}_\infty)\\
        (\mathscr{C}, S) &\mapsto \SH^\mathscr{C}(S).
    \end{align*}
    together with a natural transformation $\HH^{({-})} \to \SH^{({-})}$.
\end{construction}

\begin{proposition} \label{prop:additivity-from-motivic-patterns}
    For every motivic pattern $\mathscr{C}$, each of the functors
    \begin{equation*}
        \PSh_\Sigma(\mathscr{C}_\star), \HH^\mathscr{C}, \SH^\mathscr{C} : \AlgSpc^\mathrm{op} \to \CAlg(\Cat_\infty^\mathrm{sift})
    \end{equation*}
    preserves finite products.
\end{proposition}

\begin{proof}
    For $\PSh_\Sigma(\mathscr{C}_\star)$, this follows from the fact that
    \begin{equation*}
        \PSh_\Sigma : \CAlg(\Cat_\infty) \to \CAlg(\Cat_\infty^\mathrm{sift})
    \end{equation*}
    preserves finite products.
    For $\HH^\mathscr{C}$ and $\SH^\mathscr{C}$, this is a consequence of Nisnevich separation (Proposition~\ref{prop:unstable-nis-sep}, Proposition~\ref{prop:stable-nis-sep}) and the smooth base change formula (Proposition~\ref{prop:unstable-smooth-bc-projection}(1), Proposition~\ref{prop:stable-smooth-bc-projection}(1)).
\end{proof}

\begin{proposition} \label{prop:all-spheres-inverted}
    Let $\mathscr{C} = (\mathscr{C}_\star,\theta)$ be a normed motivic pattern, and let $S$ be an algebraic space.
    For every vector bundle $E$ over $S$, the object
    \begin{equation*}
        \Sigma^\infty_{\mathbf{T},\mathscr{C}}\theta_! \left(\frac{E}{E\smallsetminus 0}\right) \in \SH^\mathscr{C}(S)
    \end{equation*}
    is invertible.
\end{proposition}

\begin{proof}
    By \cite[Proposition~2.4.11]{CisinskiDeglise_TriangulatedMixedMotives}, this follows from Proposition~\ref{prop:stable-nis-sep} and Proposition~\ref{prop:stable-weak-loc}.
\end{proof}

\subsection{Norms for Motivic Patterns}

\begin{remark}
    Recall from Example~\ref{example:sm-s-plus-normed} that the assignment $S \mapsto \Sm_{S+}$ can be promoted to a norm monoidal $\infty$-category $\Sm_{\star +} : \Corr^\fet(\AlgSpc)^\op \to \CAlg(\Cat_\infty)$.
\end{remark}

\begin{definition}
    A \emph{normed motivic pattern} is a norm monoidal $\infty$-category
    \begin{equation*}
        \mathscr{C}_\star : \Corr^{\fet}(\AlgSpc)^\op \to \CAlg(\Cat_\infty)
    \end{equation*}
    together with a norm monoidal functor
    \begin{equation*}
        \theta : \Sm_{\star+} \to \mathscr{C}_\star : \Corr^{\fet}(\AlgSpc)^\op \to \CAlg(\Cat_\infty)
    \end{equation*}
    whose restriction along $\AlgSpc^\op \to \Corr^\fet(\AlgSpc)^\op$ is a motivic pattern.
    We write
    \begin{equation*}
        \mathscr{N}\mathrm{orm}\MotPatt \subseteq \NMon(\Cat_\infty | \AlgSpc)_{\Sm_{\star+}/}
    \end{equation*}
    for the full subcategory spanned by normed motivic patterns.
\end{definition}

\begin{notation}
    Let $\mathscr{C} = (\mathscr{C}_\star,\theta)$ be a normed motivic pattern, and let $p : X \to S$  be a finite \'etale morphism of algebraic spaces.
    We write $p_\otimes : \mathscr{C}_X \to \mathscr{C}_S$ for image of the span
    \begin{equation*}
        \begin{tikzcd}
            & X \ar[dl,"p"'] \ar[dr,"\mathrm{id}"] \\
            S & & X
        \end{tikzcd}
    \end{equation*}
    under $\theta : \Corr^{\fet}(\AlgSpc)^\op \to \CAlg(\Cat_\infty)$.
\end{notation}

\begin{example}
    The norm monoidal $\infty$-category $\Sm_{\star+}$ is the inital normed motivic pattern.
\end{example}

\begin{proposition}
    The distributive law holds in every normed motivic pattern.
\end{proposition}

\begin{proof}
    Since each functor involved commutes with $\theta$, this follows from the analogous claim for $\Sm_{\star+} : \Corr^{\fet}(\AlgSpc)^\mathrm{op} \to \Cat_\infty$.
\end{proof}

\begin{construction}
    Let $\mathscr{C} = (\mathscr{C}_\star,\theta)$ be a normed motivic pattern, and let $p : X \to S$  be a finite \'etale morphism of algebraic spaces.
    The functor $p_\otimes : \mathscr{C}_X \to \mathscr{C}_T$ extends uniquely to a symmetric monoidal functor $p_\otimes : \PSh_\Sigma(\mathscr{C}_X) \to \PSh_\Sigma(\mathscr{C}_S)$ that preserves sifted colimits.
\end{construction}

\begin{proposition} \label{prop:p-sigma-p-otimes-motivic-equivalences}
    Let $\mathscr{C} = (\mathscr{C}_\star,\theta)$ be a normed motivic pattern, and let $p : X \to S$  be a finite \'etale morphism of algebraic spaces.
    The functor $p_\otimes : \PSh_\Sigma(\mathscr{C}_X) \to \PSh_\Sigma(\mathscr{C}_S)$ preserves motivic equivalences.
\end{proposition}

\begin{proof}
    Fix a motivic equivalence $f$ in $\PSh_\Sigma(\mathscr{C}_X)$.
    By monadicity of the adjunction
    \begin{equation*}
        \theta_! : \PSh_\Sigma(\Sm_X) \rightleftarrows \PSh_\Sigma(\mathscr{C}_X) : \theta^*,
    \end{equation*}
    the morphism $f$ can be written as a simplicial colimit of morphisms of the form $\theta_! \theta^* \cdots \theta_! \theta^* f$.
    In particular, since $\theta_!$ and $\theta^*$ both preserve motivic equivalences, we can write $f$ as a simplicial colimit of morphisms of the form $\theta_! g$ with $g$ a motivic equivalence in $\PSh_\Sigma(\Sm_X)$.
    Since $p_\otimes : \PSh_\Sigma(\mathscr{C}_X) \to \PSh_\Sigma(\mathscr{C}_S)$ preserves sifted colimits and commutes with $\theta_!$, we are reduced to the analogous claim for $\PSh_\Sigma(\Sm_\star)$, which is \cite[Theorem~3.3(4)]{BachmannHoyois_Norms}.
\end{proof}

\begin{construction}
    Let $\mathscr{C} = (\mathscr{C}_\star,\theta)$ be a normed motivic pattern, and let $p : X \to S$  be a finite \'etale morphism of algebraic spaces.
    By Proposition~\ref{prop:p-sigma-p-otimes-motivic-equivalences}, the functor $\mathrm{L}_\mathrm{mot} p_\otimes : \PSh_\Sigma(\mathscr{C}_X) \to \HH^\mathscr{C}(S)$ extends uniquely to a symmetric monoidal functor $p_\otimes : \HH^\mathscr{C}(X) \to \HH^\mathscr{C}(S)$.
    The functor $p_\otimes : \HH^\mathscr{C}(X) \to \HH^\mathscr{C}(S)$ preserves sifted colimits.
\end{construction}

\begin{proposition} \label{prop:compute-p-otimes-sphere}
    Let $\mathscr{C} = (\mathscr{C}_\star,\theta)$ be a normed motivic pattern, let $p : X \to S$  be a finite \'etale morphism of algebraic spaces, and let $\mathscr{E}$ be a finite flat $\mathcal{O}_X$-module.
    \begin{enumerate}
        \item We have identifications
        \begin{equation*}
            p_\otimes \theta(\mathbb{V}(\mathscr{E})_+) = \theta(\mathbb{V}(p_* \mathscr{E})_+) \qquad\text{and}\qquad p_\otimes \theta ((\mathbb{V}(\mathscr{E}) \smallsetminus 0)_+) = \theta((\mathbb{V}(p_* \mathscr{E})\smallsetminus 0)_+)
        \end{equation*}
        in $\HH^\mathscr{C}(S)$.
        \item The canonical map
        \begin{equation*}
            p_\otimes \theta_! \left(\frac{\mathbb{V}(\mathscr{E})}{\mathbb{V}(\mathscr{E}) \smallsetminus 0} \right) \to \theta_! \left(\frac{\mathbb{V}(p_*\mathscr{E})}{\mathbb{V}(p_*\mathscr{E}) \smallsetminus 0} \right)
        \end{equation*}
        in $\HH^\mathscr{C}(S)$ is an equivalence.
    \end{enumerate}
\end{proposition}

\begin{proof}
    Since $p_\otimes$ commutes with $\theta_!$, the statement reduces to the analogous one for $\HH^\mathscr{C}(S)_\bullet$ \cite[Proposition~3.13]{BachmannHoyois_Norms}.
\end{proof}

\begin{construction}
    Let $\mathscr{C} = (\mathscr{C}_\star,\theta)$ be a normed motivic pattern, and let $p : X \to S$  be a finite \'etale morphism of algebraic spaces.
    By Proposition~\ref{prop:compute-p-otimes-sphere} and Proposition~\ref{prop:all-spheres-inverted}, the composite
    \begin{equation*}
        \begin{tikzcd}
            \HH^\mathscr{C}(X) \ar[r,"p_\otimes"] & \HH^\mathscr{C}(S) \ar[r,"\Sigma^{\infty}_{\mathbf{T},\mathscr{C}}"] & \SH^\mathscr{C}(S)
        \end{tikzcd}
    \end{equation*}
    sends $\mathbf{T}^\mathscr{C}_X$ to an invertible object.
    We can thus apply \cite[Proposition~4.1]{BachmannHoyois_Norms} to extend $\Sigma^\infty_{\mathbf{T},\mathscr{C}} p_\otimes : \HH^\mathscr{C}(X) \to \SH^\mathscr{C}(S)$ uniquely to a symmetric monoidal functor $p_\otimes : \SH^\mathscr{C}(X) \to \SH^\mathscr{C}(S)$ that preserves sifted colimits.
\end{construction}

\begin{construction} \label{construct:normed-patt-to-normed-cat}
    Using the same strategy as in Construction~\ref{construct:motivic-pattern-to-formalism}
    we can construct functors
    \begin{equation*}
        \PSh_\Sigma, \HH^{({-})}, \SH^{({-})} : \mathscr{N}\mathrm{orm}\MotPatt \to \NMon(\Cat_\infty^\mathrm{sift} | \AlgSpc)
    \end{equation*}
    and natural transformations
    \begin{equation*}
        \PSh_\Sigma \to \HH^{({-})} \to \SH^{({-})}
    \end{equation*}
    of functors $\mathscr{N}\mathrm{orm}\MotPatt \to \NMon (\Cat_\infty^\mathrm{sift} | \AlgSpc)$.
\end{construction}

\begin{proposition} \label{prop:psh-h-sh-presentably-normed}
    Let $\mathscr{C}$ be a normed motivic pattern.
    Each of the norm monoidal $\infty$-categories $\PSh_\Sigma(\mathscr{C}_\star)$, $\HH^\mathscr{C}$, and $\SH^\mathscr{C}$ over $\AlgSpc$ is presentably norm monoidal.
    Each of the norm monoidal functors
    \begin{equation*}
        \PSh_\Sigma(\mathscr{C}_\star) \to \HH^\mathscr{C} \to \SH^\mathscr{C}
    \end{equation*}
    is the left adjoint in a norm monoidal adjunction.
\end{proposition}

\begin{proof}
    Since all functors involved in the distribution transformation for $\PSh_\Sigma(\mathscr{C}_\star)$ preserve sifted colimits and motivic equivalences, the distributive laws for $\PSh_\Sigma(\mathscr{C}_\star)$ and $\HH^\mathscr{C}$ reduce to the distributive law for $\mathscr{C}_\star$.
    Since all functors involved in the distribution transformation for $\HH^\mathscr{C}$ commute with tensoring with $\mathbf{T}^\mathscr{C}$, the distributive law for $\SH^\mathscr{C}$ reduces to that of $\HH^\mathscr{C}$.

    The first claim now follows from additivity (Proposition~\ref{prop:additivity-from-motivic-patterns}), and smooth base change (Propositions~\ref{prop:smooth-bc-projection-p-sigma}, \ref{prop:unstable-smooth-bc-projection}, and \ref{prop:stable-smooth-bc-projection}).
    The second claim is an application of Proposition~\ref{prop:presentably-norm-adjunction-criterion}.
\end{proof}

\subsection{Motivic Patterns and Group Completion}

\begin{definition}
    A \emph{semiadditive motivic pattern} is a motivic pattern $(\mathscr{C}_\star,\theta)$ such that $\mathscr{C}_S$ is semiadditive for every algebraic space $S$.
    We write
    \begin{equation*}
        \MotPatt^\mathrm{semiadd} \subseteq \MotPatt
    \end{equation*}
    for the full subcategory spanned by semiadditive motivic patterns.
\end{definition}

\begin{definition}
    Let $\mathscr{C}$ be a semiadditive motivic pattern, and let $S$ be an algebraic space.
    A morphism in $\PSh_\Sigma(\mathscr{C}_S)$ is called a \emph{$\langle \mathrm{mot},\mathrm{gp} \rangle$-equivalence} if is sent to an equivalence by the localization functor $\PSh_\Sigma(\mathscr{C}_S) \to \HH^\mathscr{C}(S)^\mathrm{gp}$.
\end{definition}

\begin{proposition} \label{prop:f-upper-star-preserves-gp-equivalences}
    Let $\mathscr{C}$ be a semiadditive motivic pattern and let $f : S \to T$ be a morphism of algebraic spaces.
    The functors $f^* : \PSh_\Sigma(\mathscr{C}_T) \to \PSh_\Sigma(\mathscr{C}_S)$ and $f^* : \HH^\mathscr{C}(T) \to \HH^\mathscr{C}(S)$ both preserve $({-})^\mathrm{gp}$-equivalences.
    If $f$ is smooth, then the functors $f_\sharp : \PSh_\Sigma(\mathscr{C}_S) \to \PSh_\Sigma(\mathscr{C}_T)$ and $f_\sharp : \HH^\mathscr{C}(S) \to \HH^\mathscr{C}(T)$ also preserve $({-})^\mathrm{gp}$-equivalences.
\end{proposition}

\begin{proof}
    This is immediate since the functors in question are left adjoint between semiadditive presentable $\infty$-categories.
\end{proof}

\begin{construction}
    By Proposition~\ref{prop:f-upper-star-preserves-gp-equivalences}, we can use the same strategy as in Construction~\ref{construct:motivic-pattern-to-formalism} to construct a functors
    \begin{equation*}
        \PSh_\Sigma^\mathrm{gp}, \HH^{({-}),\mathrm{gp}} : \MotPatt^\mathrm{semiadd} \times \AlgSpc^\op \to \CAlg(\Cat_\infty^\mathrm{sift})
    \end{equation*}
    and a commutative diagram
    \begin{equation*}
        \begin{tikzcd}
            \PSh_\Sigma \ar[r] \ar[d] & \HH^{({-})} \ar[d] \ar[r] & \SH^{({-})}\\
            \PSh^\mathrm{gp}_\Sigma \ar[r] & \HH^{({-}),\mathrm{gp}} \ar[ur]
        \end{tikzcd}
    \end{equation*}
    of functors $\MotPatt^\mathrm{semiadd} \times \AlgSpc^\op \to \CAlg(\Cat_\infty^\mathrm{sift}) $
\end{construction}

\begin{definition}
    A \emph{normed semiadditive motivic pattern} is a normed motivic pattern $(\mathscr{C}_\star,\theta)$ such that $\mathscr{C}_S$ is semiadditive for every algebraic space $S$.
    We write
    \begin{equation*}
        \mathscr{N}\mathrm{orm}\MotPatt^\mathrm{semiadd} \subseteq \mathscr{N}\mathrm{orm}\MotPatt
    \end{equation*}
    for the full subcategory spanned by normed semiadditive motivic patterns.
\end{definition}

\begin{proposition} \label{prop:p-lower-otimes-preserves-gp-equivalence}
    Let $\mathscr{C}$ be a normed semiadditve motivic pattern.
    For every finite \'etale morphism $p : X \to S$, the functor $p_\otimes : \PSh_\Sigma(C_X) \to \PSh_\Sigma(C_S)$ preserves $({-})^\mathrm{gp}$-equivalences.
\end{proposition}

\begin{proof}
    This is an immediate consequence of Propositions~\ref{prop:psh-h-sh-presentably-normed}, \ref{prop:norm-is-polynomial}, and \ref{prop:polynomial-preserves-gp-equiv}.
\end{proof}

\begin{proposition}
    Let $\mathscr{C}$ be a normed semiadditve motivic pattern.
    For every finite \'etale morphism $p : X \to S$, the functor $p_\otimes : \HH^\mathscr{C}(X) \to \HH^\mathscr{C}(S)$ preserves $({-})^\mathrm{gp}$-equivalences.
\end{proposition}

\begin{proof}
    It suffices to show that the functor $p_\otimes : \PSh_\Sigma(\mathscr{C}_X) \to \PSh_\Sigma(\mathscr{C}_S)$ preserves $\langle \mathrm{mot},\mathrm{gp} \rangle$-equivalences.
    By \cite[Lemma~2.10]{BachmannHoyois_Norms}, it suffices to show $p_\otimes$ sends $f \oplus \mathrm{id}_C$ to a $\langle \mathrm{mot},\mathrm{gp} \rangle$-equivalence, where the map $f : A \to B$ in $\PSh_\Sigma(\mathscr{C}_X)$ is a motivic equivalence or a $({-})^\mathrm{gp}$-equivalence, and the object $C$ is in $\mathscr{C}_X$.
    By Remark~\ref{remark:distr-diag-binary-sum}, this follows now from Proposition~\ref{prop:p-sigma-p-otimes-motivic-equivalences} and Proposition~\ref{prop:p-lower-otimes-preserves-gp-equivalence}.
\end{proof}

\begin{construction}
    Using the same strategy as in Construction~\ref{construct:motivic-pattern-to-formalism}
    we can construct functors
    \begin{equation*}
        \PSh_\Sigma^\mathrm{gp}, \HH^{({-}),\mathrm{gp}} : \mathscr{N}\mathrm{orm}\MotPatt^\mathrm{semiadd} \to \NMon(\Cat_\infty^\mathrm{sift}|\AlgSpc)
    \end{equation*}
    and a commutative diagram
    \begin{equation*}
        \begin{tikzcd}
            \PSh_\Sigma \ar[r] \ar[d] & \HH^{({-})} \ar[d] \ar[r] & \SH^{({-})}\\
            \PSh^\mathrm{gp}_\Sigma \ar[r] & \HH^{({-}),\mathrm{gp}} \ar[ur]
        \end{tikzcd}
    \end{equation*}
    of functors $\mathscr{N}\mathrm{orm}\MotPatt^\mathrm{semiadd} \to \NMon(\Cat_\infty^\mathrm{sift} | \AlgSpc)$.
\end{construction}

\begin{proposition}
    Let $\mathscr{C}$ be a normed semiadditive motivic pattern.
    The normed $\infty$-categories $\PSh_\Sigma(\mathscr{C}_\star)^\mathrm{gp}$ and $\HH^{\mathscr{C},\mathrm{gp}}$ over $\AlgSpc$ are both presentably normed.
    Each of the norm monoidal functors
    \begin{equation*}
        \begin{tikzcd}
            \PSh_\Sigma(\mathscr{C}_\star) \ar[r] \ar[d] & \HH^{\mathscr{C}} \ar[d] \ar[r] & \SH^{\mathscr{C}} \\
            \PSh_\Sigma(\mathscr{C}_\star)^\mathrm{gp} \ar[r] & \HH^{\mathscr{C},\mathrm{gp}} \ar[ur]
        \end{tikzcd}
    \end{equation*}
    is the left adjoint in a norm monoidal adjunction.
\end{proposition}

\begin{proof}
    This is proved using the same argument as Proposition~\ref{prop:psh-h-sh-presentably-normed}.
\end{proof}

\section{Examples of Normed Motivic Patterns}

The goal of this section is to prove the following

\begin{theorem} \label{theorem:THE-DIAGRAM}
    The assignment sending an algebraic space $S$ to the diagram
    \begin{equation} \label{eqn:THE-DIAGRAM}
        \begin{tikzcd}
            \Sm_{S+} \ar[d] \\
            \Corr^{\fet}(\Sm_S) \ar[d] \\
            \Corr^{\fr}(\Sm_S) \ar[r] & \Corr^{\orfsyn}(\Sm_S) \ar[r] \ar[d] & \Corr^{\orfgor}(\Sm_S) \ar[d] \\
            & \Corr^{\fsyn}(\Sm_S) \ar[r] & \Corr^{\fflat}(\Sm_S).
        \end{tikzcd}
    \end{equation}
    can be promoted to a diagram of normed motivic patterns.
\end{theorem}

\begin{theorem} \label{theorem:THE-CONSEQUENCES}
    Let $\mathscr{F} \in \{ \fflat, \fsyn, \orfgor, \orfsyn, \fr, \fet\}$.
    The assignment
    \begin{equation*}
        S \mapsto \HH^\mathscr{F}(S)
    \end{equation*}
    can be promoted to a norm monoidal $\infty$ categoy.
    The functors $\HH(S)_\bullet \to \HH^\mathscr{F}(S)$ can be assembled a norm monoidal functor.
\end{theorem}

\begin{proof}
    This comes from applying Construction~\ref{construct:normed-patt-to-normed-cat} to (\ref{eqn:THE-DIAGRAM}).
\end{proof}

The proof of Theorem~\ref{theorem:THE-DIAGRAM} will take the entire section.
Let us explain the notation a bit more carefully.

\begin{notation}
    Let $S$ be an algebraic space and let
    \begin{equation*}
        \mathscr{F} \in \{ \fflat, \fsyn, \orfgor, \orfsyn, \fr, \fet\}.
    \end{equation*}
    For $X, Y \in \AlgSpc_S$, we let $\Corr^\mathscr{F}_S(X, Y)$ denote the space of spans
    \begin{equation*}
        \begin{tikzcd}
            & \widetilde{X} \ar[dl,"p"'] \ar[dr] \\
            X & & Y
        \end{tikzcd}
    \end{equation*}
    in $\AlgSpc_S$ where the backwards morphism $p$ is finite flat, possibly with some extra properties or structures as follows.
    \begin{enumerate}
        \item For $\mathscr{F} = \fflat$, we require nothing more.
        \item For $\mathscr{F} = \fsyn$, we require that $p$ is syntomic.
        \item For $\mathscr{F} = \orfgor$, we require that $p$ is equipped with an isomorphism $\mathcal{O}_{\widetilde{X}} \cong \omega_p$.
        \item For $\mathscr{F} = \orfsyn$, we require that $p$ is syntomic and equipped with an isomorphism $\mathcal{O}_{\widetilde{X}} \cong \omega_p$.
        \item For $\mathscr{F} = \fr$, we require that $p$ is syntomic and equipped with an equivalence $\mathscr{L}_p \cong 0$ in the $\mathrm{K}$-theory space $\mathrm{K}(\widetilde{X})$.
        \item For $\mathscr{F} = \fet$, we require that $p$ is finite \'etale.
    \end{enumerate}
    We write $\mathscr{L}_p$ and $\omega_p$ for the cotangent complex and the dualizing sheaf of $p$, respectively.
    Note that, for $p$ finite syntomic, we have $\omega_p \cong \mathrm{det}(\mathscr{L}_p)$ by \cite[Lemma 7.1]{HJNY_Hermitian}.
\end{notation}

\begin{notation}
    Let $S$ be an algebraic space.
    Each of the $\infty$-categories in (\ref{eqn:THE-DIAGRAM}) can be described as follows.
    \begin{itemize}
        \item Objects are the same as those of $\Sm_S$.
        \item For $X, Y \in \Sm_S$, the space of morphisms in $\Corr^\mathscr{F}(\Sm_S)$ from $X$ to $Y$ is $\Corr^\mathscr{F}_S(X, Y)$
    \end{itemize}
    Each of $\Corr^\mathscr{F}(\Sm_S)$ is semiadditive and comes with a symmetric monoidal structure where the tensor product is the product of algebraic spaces over $S$.
    Each of the functors in (\ref{eqn:THE-DIAGRAM}) is symmetric monoidal and preserves finite coproducts.
\end{notation}

\begin{remark}
    For more details on these $\infty$-categories, we refer the reader to
    \begin{itemize}
        \item \cite[\S 3.2]{EHKSY_MotivicInfiniteLoops} and \cite{Hoyois_2021} for $\mathscr{F} = \fr$,
        \item \cite[\S 4]{EHKSY_ModulesOverMGL} for $\mathscr{F} = \fsyn$ and $\orfsyn$,
        \item \cite[\S 5]{HJNTY_HilbertSchemes} for $\mathscr{F} = \fflat$, and
        \item \cite[\S 5]{HJNY_Hermitian} for $\mathscr{F} = \orfgor$.
    \end{itemize} 
\end{remark}

\begin{construction} \label{construct:examples-are-patterns}
    The assignment sending $S$ to the diagram (\ref{eqn:THE-DIAGRAM}) can be promoted to a diagram of motivic patterns.

    Indeed, the functoriality of the diagram in $S \in \AlgSpc^\op$ is straightforward, except for potentially $S \mapsto \Corr^\fr(\Sm_S)$.
    To construct this object and the arrows connected to it, one can use the technology of labeled triples as developed in \cite[\S 4]{EHKSY_MotivicInfiniteLoops}.
    We will perform a slightly more elaborate version of this construction in the following subsections.

    It remains to verify the motivic pattern axioms, but this is again straightforward.
    Axiom (1) holds by construction.
    For Axiom (2), one uses the same argument from \cite[Lemma 3.1]{Bachmann_Cancellation} or \cite[Proposition 2.3.7]{EHKSY_MotivicInfiniteLoops}.
    To verify Axiom (3), one makes use of the following observation:
    for $\mathscr{F} \in \{ \fflat, \fsyn, \orfgor, \orfsyn, \fr, \fet\}$ and
    \begin{equation*}
        U \to X \xrightarrow{p} S
    \end{equation*}
    smooth morphisms of algebraic spaces, the functor
    \begin{equation*}
        \Corr^\mathscr{F}_X(U, p^*({-})) : \Corr^\mathscr{F}(\Sm_S) \to \Spc
    \end{equation*}
    is corepresented by $p_\sharp U$.
\end{construction}

\begin{notation}
    For $\mathscr{F} \in \{ \fflat, \fsyn, \orfgor, \orfsyn, \fr, \fet\}$, we use $\mathscr{F}$ to denote the motivic pattern $\Sm_{\star+} \to \Corr^\mathscr{F}(\Sm_\star)$.
    We write $\PSh^\mathscr{F}_\Sigma(\Sm_\star) = \PSh_\Sigma(\Corr^\mathscr{F}(\Sm_\star))$.
\end{notation}

\subsection{Norm Monoidal \texorpdfstring{$\infty$}{oo}-Categories via Descent}

\begin{notation}
    Let $B$ be an algebraic space, let $\mathscr{B} \subseteq_\fet \AlgSpc_B$, and let $\mathscr{D}$ be an $\infty$-category with finite products.
    We write
    $$\NMon_\fet(\mathscr{D}|\mathscr{B}) \subseteq \Fun(\Corr^\fet(\mathscr{B})^\op, \mathscr{D})$$
    for the full subcategory spanned by presheaves whose restriction to $\mathscr{B}$ is a sheaf for the finite \'etale topology.
\end{notation}

\begin{proposition}
    Let $B$ be an algebraic space, let $\mathscr{B} \subseteq_\fet \AlgSpc_B$, and let $\mathscr{D}$ be an $\infty$-category with finite products.
    There is an equivalence
    \begin{equation*}
        \NMon_\fet(\mathscr{D}|\mathscr{B}) \cong \Sh_{\fet}(\mathscr{B}; \CMon(\mathscr{D}))
    \end{equation*}
\end{proposition}

\begin{proof}
    This is a direct consequence of \cite[Proposition C.13]{BachmannHoyois_Norms}.
\end{proof}

\begin{construction}
    Let $S$ be an algebraic space.
    We can form a diagram
    \begin{equation} \label{eqn:THE-DIAGRAM-except-fr-pointwise}
        \begin{tikzcd}
            \Sm_{S+} \ar[d] \\
            \Corr^{\fet}(\Sm_S) \ar[dr] \\
            & \Corr^{\orfsyn}(\Sm_S) \ar[r] \ar[d] & \Corr^{\orfgor}(\Sm_S) \ar[d] \\
            & \Corr^{\fsyn}(\Sm_S) \ar[r] & \Corr^{\fflat}(\Sm_S).
        \end{tikzcd}
    \end{equation}
    of symmetric monoidal $\infty$-categories.
    We can promote this to a diagram of symmetric monoidal flagged $\infty$-categories by equipping each $\infty$-category with the evident essentially surjective functor from $(\Sm_{S+})^\simeq$.

    The assignment sending $S$ to (\ref{eqn:THE-DIAGRAM-except-fr-pointwise}) can now be promoted to a diagram
    \begin{equation} \label{eqn:THE-DIAGRAM-except-fr}
        \begin{tikzcd}
            \Sm_{\star+} \ar[d] \\
            \Corr^{\fet}(\Sm_\star) \ar[dr] \\
            & \Corr^{\orfsyn}(\Sm_\star) \ar[r] \ar[d] & \Corr^{\orfgor}(\Sm_\star) \ar[d] \\
            & \Corr^{\fsyn}(\Sm_\star) \ar[r] & \Corr^{\fflat}(\Sm_\star).
        \end{tikzcd}
    \end{equation}
    of presheaves $\AlgSpc^\op \to \CAlg(\fCat)$ that preserves finite products.
\end{construction}

\begin{proposition} \label{prop:descent-for-everyone-except-fr}
    Each of the presheaves of symmetric monoidal flagged $\infty$-categories appearing in (\ref{eqn:THE-DIAGRAM-except-fr}) is a sheaf for the \'etale topology.
    In particular, the diagram can be promoted to one of normed monoids of flagged $\infty$-categories over $\AlgSpc$ in a unique way.
\end{proposition}

\begin{proof}
    Let $\mathscr{F} \in \{\fflat, \fsyn,\orfgor, \orfsyn, \fet\}$.
    We use Proposition~\ref{prop:sheaf-criteria-flagged-cats} to show that $((\Sm_{\star+})^\simeq \to \Corr^\mathscr{F}(\Sm_\star))$ is a sheaf of flagged $\infty$-categories for the \'etale topology.
    Note $(\Sm_{\star+})^\simeq$ is an \'etale sheaf since smooth algebraic spaces satisfy \'etale descent.

    Let $S \in \AlgSpc$ and let $X, Y \in \Sm_S$.
    We need to show that the presheaf
    \begin{equation*}
        \AlgSpc_S^\op \to \Spc, \quad (p : T \to S) \mapsto \Corr^\mathscr{F}_T(p^*X, p^* Y)
    \end{equation*}
    satisfies \'etale descent.
    Note that $\Corr^\mathscr{F}_T(p^*X, p^* Y) \cong \Corr^\mathscr{F}_S(T \times_S X, Y)$.
    Thus, the claim follows from the fact that $\Corr^\mathscr{F}_S({-}, Y)$ is an \'etale sheaf on $\AlgSpc_S$.
\end{proof}

\subsection{Recollections on Labeled Correspondences}

\begin{definition}
    A \emph{triple} is a triple $(C,L,R)$ where $C$ is an $\infty$-category, $L$ and $R$ are classes of morphisms in $C$, and the following conditions hold.
    \begin{enumerate}
        \item $L$ and $R$ are closed under composition,
        \item for every diagram
            \begin{equation*}
                \begin{tikzcd}
                    & \bullet \ar[d,"p"] \\
                    \bullet \ar[r,"f"] & \bullet
                \end{tikzcd}
            \end{equation*}
            if $f$ is in $R$ and $p$ is in $L$, then a pullback exists, and
        \item for every Cartesian diagram
            \begin{equation*}
                \begin{tikzcd}
                    \bullet \ar[r,"g"] \ar[d] \arrow[dr,phantom,"\lrcorner", very near start] & \bullet \ar[d] \\
                    \bullet \ar[r,"f"] & \bullet
                \end{tikzcd}
            \end{equation*}
            if $f$ is in $R$ (in $L$), then $g$ is in $R$ (in $L$, respectively).
    \end{enumerate}
    If $R = \all$, we simply write $(C, L) = (C, L, \all)$.

    A commutative square
    \begin{equation*}
        \begin{tikzcd}
            \bullet \ar[r,"g"] \ar[d,"q"] & \bullet \ar[d,"p"] \\
            \bullet \ar[r,"f"] & \bullet
        \end{tikzcd}
    \end{equation*}
    in $C$ is called \emph{ambigressive} if it is Cartesian, $f$ is in $R$, and $p$ is in $L$.

    A \emph{morphism} of triples between $(C,L,R)$ and $(C',L',R')$ is a functor $f : C \to C'$ that sends $L$ to $L'$, sends $R$ to $R'$, and sends ambigressive squares to ambigressive squares.
    We write $\mathrm{Trip}$ for the $\infty$-category of triples.
\end{definition}

\begin{example}
    All of the examples we will consider will have $C$ some $\infty$-category of algebraic spaces closed under fiber products, and $R$ the class of all morphisms.
    We then only need $L$ to be closed under composition and stable by base change, which holds for most classes of morphisms considered in algebraic geometry.
    Of particular importance for us is the example $(\AlgSpc_S,\fsyn,\mathrm{all})$, where $\fsyn$ denotes the class of finite syntomic maps.
\end{example}

\begin{example}
    Let $\Sigma_n$ denote the twisted arrow category of $\Delta^n \in \mathbf{\Delta}$.
    Explicitly, the objects of $\Sigma_n$ are inequalities $i \leq j$ where $i, j \in \Delta^n$, and there is a unique morphism $(i \leq j) \to (i' \leq j')$ if and only if $i \leq i'$ and $j \geq j'$.
    Let
    \begin{enumerate}
        \item $\Sigma^L_n$ denote the class of morphisms $(i \leq j) \to (i' \leq j')$ where $i = i$, and
        \item $\Sigma^R_n$ denote the class of morphisms $(i \leq j) \to (i' \leq j')$ where $j = j'$.
    \end{enumerate}
    Then $(\Sigma_n, \Sigma^L_n, \Sigma^R_n)$ is a triple.

    For example, we can view $(\Sigma_3, \Sigma^L_3, \Sigma^R_3)$ as a diagram
    
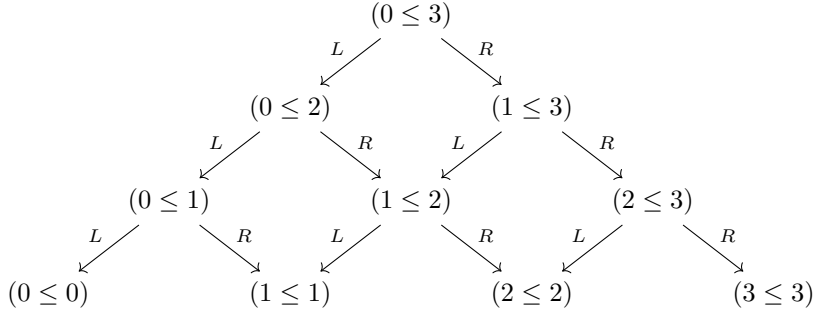
\begin{equation*}
        \begin{tikzcd}[column sep=tiny]
            & & & (0 \leq 3) \ar[dl,"L"'] \ar[dr,"R"] \\
            & & (0 \leq 2) \ar[dl,"L"'] \ar[dr,"R"] & & (1 \leq 3) \ar[dl,"L"'] \ar[dr,"R"] \\
            & (0 \leq 1) \ar[dl,"L"'] \ar[dr,"R"] & & (1 \leq 2) \ar[dl,"L"'] \ar[dr,"R"] & & (2 \leq 3) \ar[dl,"L"'] \ar[dr,"R"] \\
            (0 \leq 0) & & (1 \leq 1) & & (2 \leq 2) & & (3 \leq 3)
        \end{tikzcd}
    \end{equation*}
    where the morphisms labeled with $L$ and $R$ are in $\Sigma^L_3$ and $\Sigma^R_3$, respectively.

    By functoriality of the twisted arrow category, the assignment $\Delta^n \mapsto (\Sigma_n, \Sigma_n^L, \Sigma_n^R)$ can be promoted to functor $\mathbf{\Delta} \to \Trip$.
\end{example}

\begin{definition}
    Let $(C,L,R)$ be a triple.
    For $n \geq 0$, we'll write $\mathbf{Corr}_n(C,L,R)$ for the space
    \begin{equation*}
        \mathrm{Map}_{\Trip}((\Sigma_n,\Sigma_n^L,\Sigma_n^R),(C,L,R))
    \end{equation*}
    of morphisms of triples.

    The assignment $[n] \mapsto \mathbf{Corr}_n(C,L,R)$ defines a complete Segal space.
    See \cite[Proposition 5.6]{Barwick_SpectralMackey}.
    Furthermore, the assignment $(C,L,R) \mapsto \mathbf{Corr}_\bullet(C,L,R)$ defines a functor $\mathbf{Corr}_\bullet : \mathrm{Trip} \to \CompleteSegalSpaces$.
\end{definition}

\begin{example}
    Let $(C,L,R)$ be a triple.
    The objects of the space $\mathbf{Corr}_3(C,L,R)$ are commutative diagrams in $C$ of the form
    \begin{equation*}
        \begin{tikzcd}[column sep=tiny]
            & & & X_{0,3} \ar[dl,"L"'] \ar[dr,"R"] \arrow[dd,phantom,"\lrcorner" rotate=-45, very near start] \\
            & & X_{0,2} \ar[dl,"L"'] \ar[dr,"R"] \arrow[dd,phantom,"\lrcorner" rotate=-45, very near start] & & X_{1,3} \ar[dl,"L"'] \ar[dr,"R"] \arrow[dd,phantom,"\lrcorner" rotate=-45, very near start] \\
            & X_{0,1} \ar[dl,"L"'] \ar[dr,"R"] & & X_{1,2} \ar[dl,"L"'] \ar[dr,"R"] & & X_{2,3} \ar[dl,"L"'] \ar[dr,"R"] \\
            X_{0,0} & & X_{1,1} & & X_{2,2} & & X_{3,3}
        \end{tikzcd}.
    \end{equation*}
    Morphisms of $\mathbf{Corr}_3(C,L,R)$ are equivalences of such diagrams.
\end{example}

\begin{definition}
    Let $(C,L,R)$ be a triple.
    For every $n \geq 0$, we'll write $\Phi_n(C,L,R)$ for the subcategory of $\mathrm{Fun}(\Delta^n, C)$ whose objects are functors $\sigma : \Delta^n \to C$ that send every edge in $\Delta^n$ to $L$, and whose morphisms are Cartesian transformations with components in $R$.

    The assignement $[n] \mapsto \Phi_n(C,L,R)$ defines a functor $\Phi_\bullet(C,L,R) : \mathbf{\Delta}^\op \to \Cat_\infty$.
    Furthermore, the assignment $(C,L,R) \mapsto \Phi_\bullet(C,L,R)$ defines a functor $\Phi_\bullet : \Trip \to \mathrm{Fun}(\mathbf{\Delta}^\op, \Cat_\infty)$
\end{definition}

\begin{example}
    Let $(C,L,R)$ be a triple.
    An object $\sigma$ in $\Phi_3(C,L,R)$ are diagrams
    \begin{equation*}
        \begin{tikzcd}
            X_0 & X_1 \ar[l,"L"'] & X_2 \ar[l,"L"'] & X_3 \ar[l,"L"']
        \end{tikzcd}
    \end{equation*}
    in $C$ where each morphism is in $L$.
    A morphism $\sigma \to \tau$ in $\Phi_3(C,L,R)$ is a diagram
    \begin{equation*}
        \begin{tikzcd}
            \sigma \ar[d] & X_0 \ar[d,"R"] & X_1 \ar[l,"L"'] \ar[d,"R"] \arrow[dl,phantom,"\llcorner", very near start] & X_2 \ar[l,"L"'] \ar[d,"R"] \arrow[dl,phantom,"\llcorner", very near start] & X_3 \ar[l,"L"'] \ar[d,"R"] \arrow[dl,phantom,"\llcorner", very near start] \\
            \tau & Y_0 & Y_1 \ar[l,"L"'] & Y_2 \ar[l,"L"'] & Y_3 \ar[l,"L"']
        \end{tikzcd}
    \end{equation*}
    where the horizontal morphisms are in $L$, the vertical morphisms are in $R$, and each square is Cartesian.
\end{example}

\begin{definition}
    Let $(C,L,R)$ be a triple.
    A \emph{labeling functor} on $(C,L,R)$ is a functor
    \begin{equation*}
        F : \int_{\bullet \in \mathbf{\Delta}^\op} \Phi_\bullet(C,L,R)^\op \to \Spc
    \end{equation*}
    such that, for every $[n] \in \mathbf{\Delta}^\op$ and every $\sigma \in \Phi_n(C,L,R)$, the canonical map
    \begin{equation*}
        F(\sigma) \to F(\sigma_{0,1}) \times_{F(\sigma_1)} \cdots \times_{F(\sigma_{n-1})} F(\sigma_{n-1,n})
    \end{equation*}
    is an equivalence.
    Here, for $\phi : [k] \to [n]$ and $\sigma \in \Phi_n(C,L,R)$, we are writing $\phi^*\sigma = \sigma_{\phi(0), \ldots, \phi(k)}$.

    We'll write $\Lab(C,L,R) \subseteq \mathrm{Fun}(\int_{\bullet \in \mathbf{\Delta}^\op} \Phi_\bullet(C,L,R)^\op, \Spc)$ for the full subcategory spanned by labeling functors.
    The assignment $(C,L,R) \mapsto \Lab(C,L,R)$ defines a functor $\Trip^\op \to \Cat_\infty$, and we'll write $\Lab\Trip \to \Trip$ for the corresponding Cartesian fibration.
\end{definition}

\begin{remark}
    We are using the definition of labeling functor from \cite[Appendix~B]{EHKSY_ModulesOverMGL}, which is strictly more general than the notion of labeling functor defined in \cite[Section~4]{EHKSY_MotivicInfiniteLoops}.
\end{remark}

\begin{example}
    Let $(C,L,R)$ be a triple, and let
    \begin{equation*}
        F : \int_{\bullet \in \mathbf{\Delta}^\op} \Phi_\bullet(C,L,R)^\op \to \Spc
    \end{equation*}
    be a labeling functor on $(C,L,R)$.
    Consider an object $\sigma \in \Phi_2(C,L,R)$ given by
    \begin{equation*}
        \begin{tikzcd}
            X_0 & X_1 \ar[l,"f_{0,1}"swap] & X_2 \ar[l,"f_{1,2}"swap]
        \end{tikzcd}
    \end{equation*}
    Then we have an equivalence
    \begin{equation*}
        F(\Delta^2,\sigma) \simeq F(\Delta^1, X_0 \leftarrow X_1) \times_{F(\Delta^0,X_1)} F(\Delta^1, X_1 \leftarrow X_2)
    \end{equation*}
    Part of the functoriality of $F$ yields a morphism
    \begin{equation*}
        F(\Delta^1, X_0 \leftarrow X_1) \times_{F(\Delta^0,X_1)} F(\Delta^1, X_1 \leftarrow X_2) \to F(\Delta^1, X_1 \leftarrow X_2).
    \end{equation*}
    Heuristically, the rest of the functoriality of $F$ asserts that this operation is suitably associative.
\end{example}

\begin{definition}
    Let $(C,L,R)$ be a triple, and let $F \in \Lab(C,L,R)$ be a labeling functor.
    For $n \geq 0$, we'll write $\mathbf{Corr}^F_n(C,L,R) \to \mathbf{Corr}_n(C,L,R)$ for the coCartesian fibration in spaces classified by the composite
    \begin{equation*}
        \begin{tikzcd}
            \mathbf{Corr}_n(C,L,R) \ar[d,equal] \\
            \mathrm{Map}_{\Trip}((\Sigma_n,\Sigma_n^L,\Sigma_n^R),(C,L,R)) \ar[d] \\
            \mathrm{Fun}_{/\mathbf{\Delta}^\op}(\int_\bullet \Phi_\bullet(\Sigma_n, \Sigma_n^L, \Sigma_n^R)^\op, \int_\bullet \Phi_\bullet(C,L,R)^\op) \ar[d,"F"] \\
            \mathrm{Fun}(\int_\bullet \Phi_\bullet(\Sigma_n, \Sigma_n^L, \Sigma_n^R)^\op, \Spc) \ar[r,"\lim"] & \Spc
        \end{tikzcd}
    \end{equation*}

    The assignment $[n] \mapsto \mathbf{Corr}^F_n(C,L,R)$ defines a Segal space $\mathbf{Corr}^F_\bullet(C,L,R) : \mathbf{\Delta}^\op \to \Spc$. See \cite[Theorem~4.1.23]{EHKSY_MotivicInfiniteLoops}.
    Furthermore, the assignment $(C,L,R;F) \mapsto \mathbf{Corr}^F_\bullet(C,L,R)$ can be promoted to a functor $\mathbf{Corr}^{({-})}_\bullet : \Lab\Trip \to \SegalSpaces$ that preserves finite products.
\end{definition}

\begin{example}
    Let $(C,L,R) \in \Trip$ be a triple, and let $F \in \Lab(C,L,R)$ be a labeling functor.
    An object of $\mathbf{Corr}^F_2(C,L,R)$ consists of an object
    \begin{equation*}
        \begin{tikzcd}[column sep=tiny]
            & & X_{0,2} \ar[dl,"L"'] \ar[dr,"R"] \arrow[dd,phantom,"\lrcorner" rotate=-45, very near start] \\
            & X_{0,1} \ar[dl,"L"'] \ar[dr,"R"] & & X_{1,2} \ar[dl,"L"'] \ar[dr,"R"] \\
            X_{0,0} & & X_{1,1} & & X_{2,2}
        \end{tikzcd}
    \end{equation*}
    of $\mathbf{Corr}_2(C,L,R)$ together with
    \begin{enumerate}
        \item a label
        \begin{align*}
            \lambda_0 \in &F(\Delta^2, X_{1,1} \leftarrow X_{1,2} \leftarrow X_{1,3}) \\
            &\simeq F(\Delta^0,X_{0,0} \leftarrow X_{0,1}) \times_{F(\Delta^1,X_{0,1})} F(\Delta^1,X_{0,1} \leftarrow X_{0,2}),
        \end{align*}
        \item a label $\lambda_1 \in F(\Delta^1, X_{1,1} \leftarrow X_{1,2})$, and
        \item a label $\lambda_2 \in F(\Delta^0, X_{2,2})$
    \end{enumerate}
    all of which are suitably compatible.
\end{example}

\subsection{Norm Monoidal \texorpdfstring{$\infty$}{oo}-Categories via Labeling Functors}

\begin{definition}
    Let $B$ be an algebraic space, and let $\mathscr{B} \subseteq_\fet \AlgSpc_B$.
    A \emph{normed triple} over $\mathscr{B}$ is a functor $\Corr^{\fet}(\mathscr{B})^\op \to \Trip$ that preserves finite products.
\end{definition}

\begin{definition}
    Let $B$ be an algebraic space, let $\mathscr{B} \subseteq_\fet \AlgSpc_B$, and and let $(C_\star, L_\star,R_\star) : \Corr^{\fet}(\mathscr{B})^\op \to \Trip$ be a normed triple over $\mathscr{B}$.
    A \emph{normed labeling functor} for $(C_\star, L_\star,R_\star)$ is a lift
    \begin{equation*}
        \begin{tikzcd}
            & \Lab\Trip \ar[d] \\
            \Corr^{\fet}(\mathscr{B})^\op \ar[ur,dashed,"{(C_\star, L_\star,R_\star;F_\star)}"] \ar[r,"{(C_\star, L_\star,R_\star)}"'] & \Trip
        \end{tikzcd}
    \end{equation*}
    that preserves finite products.
\end{definition}

\begin{lemma} \label{lemma:lift-to-labtrip}
    Let $F : A \to \Trip$.
    There is an equivalence between
    \begin{enumerate}
        \item the $\infty$-category of lifts of $F$ along $\Lab\Trip \to \Trip$, and
        \item the full subcategory of $\mathrm{Fun}(\int_{(a,\bullet) \in A \times \mathbf{\Delta}^\op} \Phi_\bullet(F(a))^\op, \Spc)$ spanned by functors
        \begin{equation*}
            \Lambda : \int_{(a,\bullet) \in A \times \mathbf{\Delta}^\op} \Phi_\bullet(F(a))^\op \to \Spc
        \end{equation*}
        such that, for every $a_0 \in A$, the restriction of $\Lambda$ to
        \begin{equation*}
            \int_{\bullet \in \mathbf{\Delta}^\op} \Phi_\bullet(F(a_0))^\op \subseteq \int_{(a,\bullet) \in A \times \mathbf{\Delta}^\op} \Phi_\bullet(F(a))^\op
        \end{equation*}
        is a labeling functor on $F(a_0)$.
    \end{enumerate}
\end{lemma}

\begin{proof}
    This is \cite[Lemma~4.3.4]{EHKSY_MotivicInfiniteLoops}
\end{proof}

\begin{construction}
    Let $B$ be an algebraic space, let $\mathscr{B} \subseteq_{\fet} \AlgSpc_B$, and let $(C_\star, L_\star,R_\star;F_\star)$ be a normed triple over $\mathscr{B}$ with a normed labeling functor.
    We'll write $\Corr^F(C_\star,L_\star,R_\star) : \Corr^{\fet}(\mathscr{B})^\op \to \Cat_\infty$
    for the composite
    \begin{equation*}
        \begin{tikzcd}
            \Corr^{\fet}(\mathscr{B})^\op \ar[r] & \Lab\Trip \ar[r] & \SegalSpaces \ar[r] & \Cat_\infty
        \end{tikzcd}
    \end{equation*}
    where the final functor takes a Segal space to its underlying $\infty$-category.
\end{construction}

\begin{lemma} \label{lemma:nomred-labeling-to-normed-cat}
    Let $B$ be an algebraic space, let $\mathscr{B} \subseteq_{\fet} \AlgSpc_B$, and let
    \begin{equation*}
        (C_\star, L_\star,R_\star;F_\star)
    \end{equation*}
    be a normed triple over over $\mathscr{B}$ with a normed labeling functor.
    The functor
    \begin{equation*}
        \Corr^F(C_\star,L_\star,R_\star) : \Corr^{\fet}(\mathscr{B})^\op \to \Cat_\infty
    \end{equation*}
    is a norm monoidal $\infty$-category over $\mathscr{B}$.
\end{lemma}

\begin{proof}
    The proof is the same as that of \cite[Proposition~4.3.8]{EHKSY_MotivicInfiniteLoops}.
\end{proof}

\subsection{The Normed Labeling Functor for Framed Correspondences}

\begin{proposition}
    The assignment $S \mapsto (\AlgSpc_S, \fsyn)$ can be promoted to a normed triple over $\AlgSpc$.
\end{proposition}

\begin{proof}
    Since Weil restriction along a finite \'etale map preserves finite syntomic morphisms, we can lift the functor $\AlgSpc_\star : \Corr^\fet(\AlgSpc)^\op \to \Cat_\infty$ along $\Trip \to \Cat_\infty$.
\end{proof}

\begin{notation}
    For $n \geq 0$, we'll write $\Ar^n = \mathrm{Fun}(\Delta^1, \Delta^n)$.
    We'll identify objects of $\Ar^n$ with inequalities $(i \leq j)$ where $i, j \in \Delta^n$.
    We'll write $\iota_n : \Delta^n \to \Ar^n$ for the functor $i \mapsto (0 \leq i)$.
\end{notation}

\begin{definition}
    Let $p : \mathscr{E} \to \mathscr{B}$ be a coCartesian fibration.
    A morphism $f$ in $\mathscr{E}$ is \emph{vertical} if $p(f)$ is an equivalence.
\end{definition}

\begin{construction} \label{def:gap}
    Let $p : \mathscr{E} \to \mathscr{B}$ be a coCartesian fibration such that
    \begin{enumerate}
        \item for each $s \in \mathscr{B}$, the fiber $\mathscr{E}_s$ admits finite colimits, and
        \item for each morphism $f : s \to t$ in $\mathscr{B}$, the functor $f_* : \mathscr{E}_s \to \mathscr{E}_t$ preserves finite colimits.
    \end{enumerate}
    For $n \geq 0$, we'll write $\mathrm{Gap}_\mathscr{B}(n,\mathscr{E})$ for the full subcategory of $\mathrm{Fun}(\Ar^n, \mathscr{E})$ spanned by functors $\sigma : \Ar^n \to \mathscr{E}$ such that
    \begin{enumerate}
        \item for $0 \leq i \leq j \leq k \leq n$, the morphism $\sigma(i \leq k) \to \sigma(j \leq k)$ is vertical,
        \item for $0 \leq i \leq n$, the object $\sigma(i\leq i)$ is $p$-relative initial, and
        \item for $0 \leq i \leq j \leq k \leq n$, the diagram
        \begin{equation*}
            \begin{tikzcd}
                \sigma(i \leq j) \ar[r] \ar[d] & \sigma(i \leq k) \ar[d] \\
                \sigma(j \leq j) \ar[r] & \sigma(j \leq k)
            \end{tikzcd}
        \end{equation*}
        is a $p$-relative colimit.
    \end{enumerate}
\end{construction}

\begin{example} \label{example:gap-qcoh}
    Let $p : \mathbf{QCoh} \to \AlgSpc^\op$ be the coCartesian fibration classified by the functor $\AlgSpc^\op \to \Cat_\infty^\mathrm{st}$ sending an algebraic space $S$ to the derived $\infty$-category $\mathbf{QCoh}(S)$ of quasi-coherent $\mathcal{O}_S$-modules.
    We can view an object of $\mathrm{Gap}_{\AlgSpc^\op}(\Delta^3, \mathbf{QCoh})$ as a diagram
    \begin{equation*}
        \begin{tikzcd}
            0 \ar[r,dashed] & \mathscr{F}_{0,1} \ar[r,dashed] \ar[d] \arrow[dr,phantom,"\ulcorner", very near end] & \mathscr{F}_{0,2} \ar[r,dashed] \ar[d] \arrow[dr,phantom,"\ulcorner", very near end] & \mathscr{F}_{0,3} \ar[d] \\
            & 0 \ar[r,dashed] & \mathscr{F}_{1,2} \ar[r,dashed] \ar[d] \arrow[dr,phantom,"\ulcorner", very near end] & \mathscr{F}_{1,3} \ar[d] \\
            & & 0 \ar[r,dashed] & \mathscr{F}_{2,3} \ar[d] \\
            & & & 0 \\
            X_0 & X_1 \ar[l,"f_{0,1}"'] & X_2 \ar[l,"f_{1,2}"'] & X_3 \ar[l,"f_{2,3}"']
        \end{tikzcd}
    \end{equation*}
    where the vertical morphisms over an algebraic space $X_i$ indicate morphism in $\mathbf{QCoh}(X_i)$, and a dashed horizontal arrows $\mathscr{F}_{i,j-1} \dashrightarrow \mathscr{F}_{i,j}$ indicate morphisms $(f_{j-1,j})^* \mathscr{F}_{i,j-1} \to \mathscr{F}_{i,j}$ in $\mathbf{QCoh}(X_j)$.
    The pushout marks in each square indicates that the induced diagram
    \begin{equation*}
        \begin{tikzcd}
            (f_{j-1,j})^* \mathscr{F}_{i-1,j-1} \ar[r] \ar[d] & \mathscr{F}_{i-1,j} \ar[d] \\
            (f_{j-1,j})^* \mathscr{F}_{i,j-1} \ar[r] & \mathscr{F}_{i,j}
        \end{tikzcd}
    \end{equation*}
    is coCartesian.
\end{example}

\begin{construction}
    Let $p : \mathscr{E} \to \mathscr{B}$ be as in Construction~\ref{def:gap}.
    For $n \geq 0$, we'll write $\mathrm{Filt}_\mathscr{B}(n, \mathscr{E}) = \mathrm{Fun}(\Delta^n, \mathscr{E})$.
    Note that restriction along $\iota_n : \Delta^n \to \Ar^n$ yields a fully faithful functor $\iota_n^*: \mathrm{Gap}_\mathscr{B}(n, \mathscr{E}) \to \mathrm{Filt}_\mathscr{B}(n, \mathscr{E})$.
    The essential image consists of those functors $\sigma : \Delta^n \to \mathscr{E}$ such that $\sigma(0)$ is $p$-relatively initial.
    Moreover, $\iota_n^*: \mathrm{Gap}_\mathscr{B}(n, X) \to \mathrm{Filt}_\mathscr{B}(n, \mathscr{E})$ admits a left adjoint.
    See \cite[after Remark~B.0.4]{EHKSY_ModulesOverMGL}.

    The assignment $[n] \mapsto (\mathrm{Filt}_\mathscr{B}(n, \mathscr{E}) \to \mathrm{Gap}_\mathscr{B}(n, \mathscr{E}))$ can be promoted to a natural transformation among functors $\mathbf{\Delta}^\op \to \Cat_\infty$.
\end{construction}

\begin{example} \label{example:gap-qcoh-cont}
    We use notation as in Example~\ref{example:gap-qcoh}.
    An object of
    \begin{equation*}
        \mathrm{Filt}_{\AlgSpc^\op}(\Delta^3,\mathbf{QCoh})
    \end{equation*}
    can be viewed as a diagram
    \begin{equation*}
        \begin{tikzcd}
            \mathscr{F}_0 \ar[r,dashed] & \mathscr{F}_1 \ar[r,dashed] & \mathscr{F}_2 \ar[r,dashed] & \mathscr{F}_3 \\
            X_0 & X_1 \ar[l,"f_{0,1}"'] & X_2 \ar[l,"f_{1,2}"'] & X_3. \ar[l,"f_{2,3}"']
        \end{tikzcd}
    \end{equation*}
    The left adjoint $\mathrm{Filt}_{\AlgSpc^\op}(\Delta^3,\mathbf{QCoh}) \to \mathrm{Gap}_{\AlgSpc^\op}(\Delta^3,\mathbf{QCoh})$
    sends this to
    \begin{equation*}
        \begin{tikzcd}
            0 \ar[r,dashed] & \mathscr{F}_1 / (f_{0,1})^*\mathscr{F}_0 \ar[r,dashed] \ar[d] \arrow[dr,phantom,"\ulcorner", very near end] & \mathscr{F}_2 / (f_{0,2})^*\mathscr{F}_0 \ar[r,dashed] \ar[d] \arrow[dr,phantom,"\ulcorner", very near end] & \mathscr{F}_3 / (f_{0,3})^*\mathscr{F}_0 \ar[d] \\
            & 0 \ar[r,dashed] & \mathscr{F}_2 / (f_{1,2})^*\mathscr{F}_1 \ar[r,dashed] \ar[d] \arrow[dr,phantom,"\ulcorner", very near end] & \mathscr{F}_3 / (f_{1,3})^*\mathscr{F}_1 \ar[d] \\
            & & 0 \ar[r,dashed] & \mathscr{F}_3 / (f_{2,3})^*\mathscr{F}_2 \ar[d] \\
            & & & 0 \\
            X_0 & X_1 \ar[l,"f_{0,1}"'] & X_2 \ar[l,"f_{1,2}"'] & X_3 \ar[l,"f_{2,3}"']
        \end{tikzcd}
    \end{equation*}
\end{example}

\begin{definition}
    We write $\mathrm{K} : \Cat_\infty^\mathrm{st} \to \Spc$ for the algebraic $\mathrm{K}$-theory functor as defined in \cite{BGT_UniversalKTheory}.
    This comes with a canonical natural transformation $\eta : ({-})^\simeq \to \mathrm{K}$.
\end{definition}

\begin{definition}
    Let $p : \mathscr{E} \to \mathscr{B}$ be a coCartesian fibration, which is classified by a functor $\chi : \mathscr{B} \to \Cat_\infty$
    We'll write $p^\mathrm{vop} : \mathscr{E}^\mathrm{vop} \to \mathscr{B}$ for the coCartesian fibration classified by $\chi({-})^\op : \mathscr{B} \to \Cat_\infty$.
    Similarly, we'll write $\mathrm{Tw}_\mathscr{B}(\mathscr{E})^\mathrm{vop} \to \mathscr{B}$ for the coCartesian fibration classified by $\mathrm{Tw}(\chi({-}))^\op : \mathscr{B} \to \Cat_\infty$.
    The natural transformation $\mathrm{Tw}(\chi({-}))^\op \to \chi({-})^\op \times \chi({-})$ then classifies a functor $\mathrm{Tw}_\mathscr{B}(\mathscr{E})^\mathrm{vop} \to \mathscr{E}^\mathrm{vop} \times_\mathscr{B} \mathscr{E}$ which is in fact a coCartesian fibration in spaces.
    The functor $\mathrm{Map}^\mathrm{fiber} : \mathscr{E}^\mathrm{vop} \times_\mathscr{B} \mathscr{E} \to \Spc$ which classifies this last functor is called the \emph{fiberwise mapping space} of $p : \mathscr{E} \to \mathscr{B}$.
\end{definition}

\begin{example}
    Let $p : \mathscr{E} \to \mathscr{B}$ be a coCartesian fibration, and let $x : \mathscr{B} \to \mathscr{E}^\mathrm{vop}$ be a section of $p^\mathrm{vop}$.
    Then we can form the functor $\mathrm{Map}^\mathrm{fiber}(x,({-})) : \mathscr{E} \to \Spc$ which sends an object $y \in \mathscr{E}$ living over $s = p(y)$ to the space $\mathrm{Map}_{\mathscr{E}_s}(x(s),y)$.
\end{example}

\begin{construction} \label{construct:label-fr}
    We construct now a normed labeling functor that refines \cite[Proposition 4.3.13]{EHKSY_MotivicInfiniteLoops}.

    For an algebraic space $S$, we have a diagram
    \begin{equation*}
        \begin{tikzcd}
            & \mathbf{QCoh}_S \ar[d,"p"] \\
            \mathrm{Fun}^\mathrm{flat}(\Delta^1, \AlgSpc_S)^\op \ar[r,"s"]  \ar[ur,"\mathscr{L}"] & \AlgSpc_S^\op
        \end{tikzcd}
    \end{equation*}
    of symmetric monoidal $\infty$-categories and symmetric monoidal functors.
    Here, $\mathrm{Fun}^\mathrm{flat}(\Delta^1, \AlgSpc_S)^\op$ is the full subcategory of $\mathrm{Fun}(\Delta^1, \AlgSpc_S)^\op$ spanned by flat morphisms, and $p: \mathbf{QCoh}_S \to \AlgSpc^\op_S$ is the coCartesian fibration classified by the functor $\AlgSpc^\op_S \to \Cat_\infty$ sending an algebraic space $X$ to the derived $\infty$-category $\mathbf{QCoh}(X)$ of quasi-coherent sheaves on $X$.
    The functor $\mathscr{L}$ sends a flat morphism $f : X \to Y$ of algebraic spaces over $S$ to (the object of $\mathbf{QCoh}_S$ corresponding to) the cotangent complex $\mathscr{L}_f \in \mathbf{QCoh}(X)$.

    The assignment
    \begin{equation*}
        S \mapsto \left(
        \begin{tikzcd}
            & \mathbf{QCoh}_S \ar[d,"p"] \\
            \mathrm{Fun}^\mathrm{flat}(\Delta^1, \AlgSpc_S)^\op \ar[r,"s"]  \ar[ur,"\mathscr{L}"] & \AlgSpc_S^\op
        \end{tikzcd}
        \right)
    \end{equation*}
    can be promoted to a diagram
    \begin{equation*}
        \begin{tikzcd}
            & \mathbf{QCoh}_\star \ar[d,"p"] \\
            \mathrm{Fun}^\mathrm{flat}(\Delta^1, \AlgSpc_\star)^\op \ar[r,"s"]  \ar[ur,"\mathscr{L}"] & \AlgSpc_\star^\op
        \end{tikzcd}
    \end{equation*}
    of finite-product-preserving functors $\AlgSpc^\op \to \CAlg(\Cat_\infty)$.
    Each object in this diagram satisfies $\fet$-descent, so we can apply \cite[Lemma~C.13]{BachmannHoyois_Norms} to extend uniquely to a diagram of functors $\Corr^\fet(\AlgSpc)^\op \to \CAlg(\Cat_\infty)$

    Observe now that we have a restriction
    \begin{equation*}
        \begin{tikzcd}[column sep=huge]
            \Phi_{\bullet^\op}(\AlgSpc_\star,\mathrm{fsyn})^\op \ar[d,symbol=\subseteq] \ar[r,dashed] & \mathrm{Fun}(\Delta^\bullet, \mathrm{Fun}^\mathrm{flat}(\Delta^1, \AlgSpc_\star)^\op) \ar[d,symbol=\subseteq] \\
            \mathrm{Fun}(\Delta^n, \AlgSpc_\star^\op) \ar[r,"\mathrm{Fun}(\Delta^1{,}({-}))"]& \mathrm{Fun}(\Ar^n, \mathrm{Fun}(\Delta^1,\AlgSpc_\star)^\op)
        \end{tikzcd}
    \end{equation*}
    where we have made use of the canonical equivalence $(\Delta^1)^\op \simeq \Delta^1$ in the bottom right corner.
    We can thus form the commutative diagram
    \begin{equation*}
        \begin{tikzcd}
            \Phi_{\bullet^\op}(\AlgSpc_\star,\mathrm{fsyn})^\op \ar[r,"\mathscr{L}"] \ar[d] & \mathrm{Fun}(\Delta^\bullet, \mathbf{QCoh}_\star) \ar[d,"p"] \\
            \mathrm{Fun}(\Ar^\bullet, \mathrm{Fun}^\mathrm{flat}(\Delta^1, \AlgSpc_\star)^\op) \ar[r,"s"] \ar[ur,"\mathscr{L}"] & \mathrm{Fun}(\Ar^\bullet,\AlgSpc_\star^\op) \ar[d,"\iota_n^*"] \\
            & \mathrm{Fun}(\Delta^\bullet,\AlgSpc_\star^\op)
        \end{tikzcd}
    \end{equation*}
    In fact, the arrow along the top factors through the subcategory
    \begin{equation*}
        \mathrm{Gap}_{\AlgSpc_\star^\op}(\bullet,\mathbf{Perf}_\star) \subseteq \mathrm{Fun}(\Delta^\bullet, \mathbf{QCoh}_\star).
    \end{equation*}
    Altogether, we obtain a diagram
    \begin{equation*}
        \begin{tikzcd}
            \Phi_{\bullet^\op}(\AlgSpc_\star,\mathrm{fsyn})^\op \ar[dr]  \ar[r,"\mathscr{L}"] & \mathrm{Gap}_{\AlgSpc_\star^\op}(\bullet,\mathbf{Perf}_\star) \ar[d] \\
            & \mathrm{Fun}(\Delta^\bullet,\AlgSpc_\star^\op)
        \end{tikzcd}
    \end{equation*}
    of functors $\Corr^{\fet}(\AlgSpc)^\op \to \mathrm{Fun}(\mathbf{\Delta}^\op, \Cat_\infty)$.

    Passing to coCartesian fibrations over $\Corr^{\fet}(\AlgSpc)^\op \times \mathbf{\Delta}^\op$, we obtain a functor
    \begin{equation*}
        \int_{\star,\bullet} \mathrm{Gap}_{\AlgSpc_\star^\op}(\bullet,\mathbf{Perf}_\star) \to \int_{\star,\bullet} \mathrm{Fun}(\Delta^\bullet,\AlgSpc_\star^\op)
    \end{equation*}
    where the ``integral'' is taken over $(\star, \bullet) \in \Corr^{\fet}(\AlgSpc)^\op \times \mathbf{\Delta}^\op$.
    This functor is itself a coCartesian fibration classifying a functor
    \begin{equation*}
        \int_{\star,\bullet} \mathrm{Fun}(\Delta^\bullet,\AlgSpc_\star^\op) \to \Cat_\infty^\mathrm{st}
    \end{equation*}
    which we whisker with the natural transformation $\eta : ({-})^\simeq \to \mathrm{K} : \Cat_\infty^\mathrm{st} \to \Spc$ and unstraighten to get
    \begin{equation*}
        \begin{tikzcd}[column sep={3cm,between origins}]
            \int_{\star,\bullet} \mathrm{Gap}_{\AlgSpc_\star^\op}(\bullet,\mathbf{Perf}_\star)^\mathrm{cocart} \ar[rr,"\eta"] \ar[dr] & & \int_{\star,\bullet} \mathrm{K}\mathrm{Gap}_{\AlgSpc_\star^\op}(\bullet,\mathbf{Perf}_\star) \ar[dl] \\
            & \int_{\star,\bullet} \mathrm{Fun}(\Delta^\bullet,\AlgSpc_\star^\op).
        \end{tikzcd}
    \end{equation*}
    Finally, we write $\mathrm{fr}$ for the composite
    \begin{equation*}
        \begin{tikzcd}
            \int_{\star,\bullet} \Phi_{\bullet}(\AlgSpc_\star,\mathrm{fsyn})^\op \ar[d] & \int_{\star,\bullet} \mathrm{Gap}_{\AlgSpc_\star^\op}(\bullet,\mathbf{Perf}_\star)^\mathrm{cocart} \ar[d,"\eta"] \\
            \int_{\star,\bullet} \Phi_{\bullet^\op}(\AlgSpc_\star,\mathrm{fsyn})^\op \ar[ur,"\mathscr{L}"] & \int_{\star,\bullet} \mathrm{K}\mathrm{Gap}_{\AlgSpc_\star^\op}(\bullet,\mathbf{Perf}_\star) \ar[r] & \Spc
        \end{tikzcd}
    \end{equation*}
    where the the first functor is induced by the natural equivalence $\mathrm{id} \to ({-})^\op : \mathbf{\Delta} \to \mathbf{\Delta}$, and the last functor is the fiberwise mapping space construction $\mathrm{Map}^\mathrm{fiber}(0,({-}))$ for the coCartesian fibration
    \begin{equation*}
        \int_{\star,\bullet} \mathrm{K}\mathrm{Gap}_{\AlgSpc_\star^\op}(\bullet,\mathbf{Perf}_\star) \to \int_{\star,\bullet} \mathrm{Fun}(\Delta^\bullet,\AlgSpc_\star^\op).
    \end{equation*}
\end{construction}

\begin{example}
    Fix an algebraic space $S$ and let $\sigma \in \Phi_2(\AlgSpc_S,\fsyn)$ be the sequence
    \begin{equation*}
        \begin{tikzcd}
            X_0 & X_1 \ar[l,"f_{0,1}"'] & X_2 \ar[l,"f_{1,2}"']
        \end{tikzcd}
    \end{equation*}
    of finite syntomic maps in $\AlgSpc_S$.
    Note that we can identify the fiber of the coCartesian fibration
    \begin{equation*}
        \int_{\star,\bullet} \mathrm{K}\mathrm{Gap}_{\AlgSpc_\star^\op}(\bullet,\mathbf{Perf}_\star) \to \int_{\star,\bullet} \mathrm{Fun}(\Delta^\bullet,\AlgSpc_\star^\op)
    \end{equation*}
    over $(S,[2],\sigma)$ as
    \begin{equation*}
        \mathrm{K}(X_1) \times \mathrm{K}(X_2).
    \end{equation*}
    Under this equivalence, we have
    \begin{equation*}
        \fr(S,[2],\sigma) = \mathrm{Map}_{\mathrm{K}(X_1)}(0, \mathscr{L}_{f_{0,1}}) \times \mathrm{Map}_{\mathrm{K}(X_2)}(0, \mathscr{L}_{f_{1,2}}).
    \end{equation*}
    See the proof of \cite[Proposition~4.2.31]{EHKSY_MotivicInfiniteLoops}.
\end{example}

\begin{variant} \label{construct:label-fet-to-fr}
    Here we promote \cite[4.3.15]{EHKSY_MotivicInfiniteLoops} to a normed labeling functor.

    Consider the composite $\mathrm{fr}'$
    \begin{equation*}
        \begin{tikzcd}
            \int_{\star,\bullet} \Phi_{\bullet}(\AlgSpc_\star,\mathrm{fsyn})^\op \ar[d]\\
            \int_{\star,\bullet} \Phi_{\bullet^\op}(\AlgSpc_\star,\mathrm{fsyn})^\op \ar[r,"\mathscr{L}"] & \int_{\star,\bullet} \mathrm{Gap}_{\AlgSpc_\star^\op}(\bullet,\mathbf{Perf}_\star)^\mathrm{cocart} \ar[r] & \Spc
        \end{tikzcd}
    \end{equation*}
    where the the first functor is induced by the natural equivalence $\mathrm{id} \to ({-})^\op : \mathbf{\Delta} \to \mathbf{\Delta}$, and the last functor is the fiberwise mapping space construction $\mathrm{Map}^\mathrm{fiber}(0,({-}))$.
    The natural transformation $\eta : ({-})^\simeq \to \mathrm{K} : \Cat^\mathrm{st}_\infty \to \Spc$ induces a natural transformation $\mathrm{fr}' \to \mathrm{fr}$.
\end{variant}

\begin{remark}
    The following variant will not play a role in the present article.
\end{remark}

\begin{variant} \label{construct:label-fr-to-K}
    This is a norm monoidal refinement of \cite[Appendix~B]{EHKSY_ModulesOverMGL}.
    Note that the constructions in Construction~\ref{construct:label-fr} work equally well if one replaces
    \begin{itemize}
        \item the $1$-category $\AlgSpc$ of algebraic spaces with the $\infty$-category $\mathrm{d}\AlgSpc$ of \emph{derived} algebraic spaces, and
        \item the class $\mathrm{fsyn}$ of finite syntomic maps with the class $\mathrm{perf}$ of maps with perfect cotangent complex.
    \end{itemize}
    We work here in that generality for this variant.

    Consider the pullback
    \begin{equation*}
        \begin{tikzcd}
            \mathrm{P}_{\star,\bullet} \ar[r] \ar[d] & \mathrm{K}\mathrm{Filt}_{\mathrm{d}\AlgSpc^\op_\star} (\bullet, \mathbf{Perf}_\star) \ar[d] \\
            \Phi_{\bullet^\op}(\mathrm{d}\AlgSpc_\star, \mathrm{perf})^\op \ar[r,"\mathscr{L}"] & \mathrm{K}\mathrm{Gap}_{\mathrm{d}\AlgSpc^\op_\star} (\bullet, \mathbf{Perf}_\star)
        \end{tikzcd}
    \end{equation*}
    of functors $\Corr^{\fr}(\mathrm{d}\AlgSpc)^\op \times \mathbf{\Delta}^\op \to \Cat_\infty$.
    Unstraightening the left vertical map yields a functor
    \begin{equation} \label{eqn:label-fr-to-K}
        \int_{\star,\bullet} \mathrm{P}_{\star,\bullet} \to \int_{\star,\bullet} \Phi_{\bullet^\op}(\mathrm{d}\AlgSpc_\star, \mathrm{perf})^\op
    \end{equation}
    which is a coCartesian fibration in spaces.
    We'll write
    \begin{equation*}
        \mathscr{K} : \int_{\star,\bullet} \Phi_{\bullet}(\mathrm{d}\AlgSpc_\star, \mathrm{perf})^\op \to \Spc
    \end{equation*}
    for the composite
    \begin{equation*}
        \int_{\star,\bullet} \Phi_{\bullet}(\mathrm{d}\AlgSpc_\star, \mathrm{perf})^\op \to \int_{\star,\bullet} \Phi_{\bullet^\op}(\mathrm{d}\AlgSpc_\star, \mathrm{perf})^\op \to \Spc
    \end{equation*}
    where the first functor is induced by the canonical equivalence $\mathrm{id} \to ({-})^\op : \mathbf{\Delta} \to \mathbf{\Delta}$, and the second is the straightening of (\ref{eqn:label-fr-to-K}).
\end{variant}

\begin{construction} \label{construct:label-fr-to-orfsyn}
    Let $n \geq 0$, $S \in \AlgSpc$, and $\sigma \in \Phi_n(\AlgSpc_S,\mathrm{fsyn})$.
    We can view $\sigma$ as a sequence of maps
    \begin{equation*}
        \begin{tikzcd}
            X_0 & X_1 \ar[l,"f_{0,1}"'] & \cdots \ar[l,"f_{1,2}"'] & X_n \ar[l,"f_{n-1,n}"'].
        \end{tikzcd}
    \end{equation*}
    We can associate to such a triplet $([n],S,\sigma)$ the set
    \begin{equation*}
        \orfsyn(S,[n],\sigma) = \mathrm{Isom}(\mathcal{O}_{X_1}, \omega_{f_{0,1}}) \times \cdots \times \mathrm{Isom}(\mathcal{O}_{X_n}. \omega_{f_{n-1,n}})
    \end{equation*}
    where $\mathrm{Isom}(\mathscr{F},\mathscr{G})$ is the set of isomorphisms between two quasi-coherent sheaves.

    Using the functoriality of dualizing sheaves, the association
    \begin{equation*}
        (S,[n],\sigma) \mapsto \orfsyn(S,[n],\sigma)
    \end{equation*}
    can in fact be promoted to a functor
    \begin{equation*}
        \orfsyn : \int_{\star,\bullet} \Phi_\bullet(\AlgSpc_\star,\mathrm{fsyn})^\op \to \mathscr{S}\mathrm{et}.
    \end{equation*}
    It comes with a natural transformation $\fr \to \orfsyn$ induced by determinant functors.
\end{construction}

\begin{example}
    Fix an algebraic space $S$ and let $\sigma \in \Phi_2(\AlgSpc_S,\fsyn)$ be the sequence
    \begin{equation*}
        \begin{tikzcd}
            X_0 & X_1 \ar[l,"f_{0,1}"'] & X_2 \ar[l,"f_{1,2}"']
        \end{tikzcd}
    \end{equation*}
    of finite syntomic maps in $\AlgSpc_S$.

    Recall that we have a coCartesian fibration
    \begin{equation*}
        \int_{\star,\bullet} \Phi_\bullet(\AlgSpc_\star,\mathrm{fsyn})^\op \to \Corr^{\fet}(\AlgSpc)^\op \times \mathbf{\Delta}^\op
    \end{equation*}
    Let
    \begin{equation*}
        \psi : (S,[2],\sigma) \to (S,[1],\tau)
    \end{equation*}
    denote a coCartesian lift of the map
    \begin{equation*}
        (\mathrm{id},\delta_1) : (S,[1]) \to (S,[2])
    \end{equation*}
    in $\Corr^{\fet}(\AlgSpc) \times \mathbf{\Delta}$.
    The map
    \begin{equation*}
        \orfsyn(\psi) : \orfsyn(S,[2],\sigma) \to \orfsyn(S,[1],\tau)
    \end{equation*}
    is then the map
    \begin{equation*}
        \mathrm{Isom}(\mathcal{O}_{X_1},\omega_{f_{0,1}}) \times \mathrm{Isom}(\mathcal{O}_{X_2},\omega_{f_{1,2}}) \to \mathrm{Isom}(\mathcal{O}_{X_2},\omega_{f_{0,1}f_{1,2}})
    \end{equation*}
    that sends
    \begin{equation*}
        (\alpha : \mathcal{O}_{X_1} \to \omega_{f_{0,1}}, \beta : \mathcal{O}_{X_2} \to \omega_{f_{1,2}})
    \end{equation*}
    to
    \begin{equation*}
        ((f_{1,2})^* \omega_{f_{0,1}} \otimes \beta) \circ (f_{1,2})^*\alpha : \mathcal{O}_{X_2} \to \omega_{f_{0,2}}.
    \end{equation*}
    Note that we are making use of the canonical isomorphism $\omega_{f_{0,2}} \simeq (f_{1,2})^* \omega_{f_{0,1}} \otimes \omega_{f_{1,2}}$.

    The natural map $\fr(S,[2],\sigma) \to \orfsyn(S,[2],\sigma)$ is the determinant map
    \begin{equation*}
        \mathrm{Map}_{\mathrm{K}(X_1)}(0, \mathscr{L}_{f_{0,1}}) \times \mathrm{Map}_{\mathrm{K}(X_2)}(0, \mathscr{L}_{f_{1,2}}) \to \mathrm{Isom}(\mathcal{O}_{X_1},\omega_{f_{0,1}}) \times \mathrm{Isom}(\mathcal{O}_{X_2},\omega_{f_{1,2}}).
    \end{equation*}
\end{example}

\begin{lemma} \label{lemma:example-normed-labeling-functors}
    Each of the functors
    \begin{enumerate}
        \item $\fr$ (Construction~\ref{construct:label-fr}),
        \item $\fr'$ (Construction~\ref{construct:label-fet-to-fr}),
        \item $\orfsyn$ (Construction~\ref{construct:label-fr-to-orfsyn})
    \end{enumerate}
    satisfies the condition in Lemma~\ref{lemma:lift-to-labtrip}(2).
    Each of the corresponding functors
    \begin{equation*}
        \Corr^{\fet}(\AlgSpc)^\op \to \Lab\Trip
    \end{equation*}
    is a normed triple with labeling functor.
\end{lemma}

\begin{proof}
    This follows from the same proofs as in \cite[Proposition~4.2.31]{EHKSY_MotivicInfiniteLoops} and \cite[Proposition~4.3.13]{EHKSY_MotivicInfiniteLoops}.
\end{proof}

\begin{construction} \label{construct:fet-fr-orfsyn}
    From Constructions~\ref{construct:label-fr}, \ref{construct:label-fet-to-fr}, and \ref{construct:label-fr-to-orfsyn} we have a diagram
    \begin{equation*}
        \begin{tikzcd}
            \fr' \ar[d] \ar[dr] \\
            \fr \ar[r] & \orfsyn
        \end{tikzcd}
    \end{equation*}
    of functors $\int_{\star,\bullet} \Phi_\bullet(\AlgSpc_\star,\mathrm{fsyn})^\op \to \Spc$.
    Using Lemma~\ref{lemma:example-normed-labeling-functors}, we can apply Lemma~\ref{lemma:lift-to-labtrip} to get a commutative diagram
    \begin{equation*}
        \begin{tikzcd}
            \fet \ar[d] \ar[dr] \\
            \fr \ar[r] & \orfsyn
        \end{tikzcd}
    \end{equation*}
    of normed triples with labeling functors $\Corr^\fet(\AlgSpc)^\op \to \Lab\Trip$.
    Applying Lemma~\ref{lemma:nomred-labeling-to-normed-cat} yields a commutative diagram
    \begin{equation} \label{eqn:THE-DIAGRAM-corner-piece}
        \begin{tikzcd}
            \Corr^{\fet}(\AlgSpc_\star) \ar[d] \ar[dr] \\
            \Corr^{\fr}(\AlgSpc_\star) \ar[r] & \Corr^{\orfsyn}(\AlgSpc_\star)
        \end{tikzcd}
    \end{equation}
    of norm monoidal $\infty$-categories over $\AlgSpc$ and norm monoidal functors.
\end{construction}

\begin{proof}[Proof of Theorem~\ref{theorem:THE-DIAGRAM}]
    Since smooth algebraic spaces are preserved under base change and Weil restriction along finite \'etale maps, we get from (\ref{eqn:THE-DIAGRAM-corner-piece}) a diagram
    \begin{equation} \label{eqn:THE-DIAGRAM-corner-piece-smoothly}
        \begin{tikzcd}
            \Corr^{\fet}(\Sm_\star) \ar[d] \ar[dr] \\
            \Corr^{\fr}(\Sm_\star) \ar[r] & \Corr^{\orfsyn}(\Sm_\star)
        \end{tikzcd}
    \end{equation}
    of norm monoidal $\infty$-categories over $\AlgSpc$.
    We conclude by noting that the uniqueness statement in Proposition~(\ref{prop:descent-for-everyone-except-fr}) implies that the diagonal arrow in (\ref{eqn:THE-DIAGRAM-corner-piece-smoothly}) must agree with the diagonal arrow in (\ref{eqn:THE-DIAGRAM-except-fr}).
\end{proof}

\section{Motivic Homotopy Theory with Norms and Transfers}

\subsection{Multiplicative Theory of Motivic Infinite Loop Spaces}

\begin{notation}
    We write $\gamma : \Sm_{\star+} \to \Corr^\fr(\Sm_\star)$ for the normed motivic pattern constructed in Theorem~\ref{theorem:THE-DIAGRAM}.
\end{notation}

\begin{remark}
    Recall from \cite[Theorem 18]{Hoyois_2021} that the functor
    \begin{equation*}
        \gamma_! : \SH(S) \to \SH^\fr(S)
    \end{equation*}
    is an equivalence for every algebraic space $S$.
    Beware that the functor is called $\gamma^*$ there.
\end{remark}

\begin{theorem} \label{theorem:multiplicative-reconstruction}
    The equivalences
    \begin{equation*}
        \gamma_! : \SH(S) \to \SH^\fr(S), \qquad S \in \AlgSpc
    \end{equation*}
    can be assembled into an equivalence of norm monoidal $\infty$-categories over $\AlgSpc$.
\end{theorem}

\begin{proof}
    This is a direct consequence of Construction~\ref{construct:fet-fr-orfsyn}.
\end{proof}

\begin{construction}
    Recall that, for $S \in \AlgSpc$, the $\infty$-category of \emph{very effective motivic spectra} over $S$ is the full subcategory $\SH^\veff(S) \subseteq \SH(S)$ generated under colimits and extensions by objects of the form $\Sigma_\mathbf{T}^n \Sigma^\infty_+ X$, with $X \in \Sm_S$ and $n \geq 0$.
    Using Proposition~\ref{prop:norm-t-structure}, we obtain a presentably norm monoidal $\infty$-category $\SH^{\veff}$ over $\AlgSpc$, together with a norm monoidal adjunction $\SH^\veff \rightleftarrows \SH$.
\end{construction}

\begin{proposition} \label{prop:fr-gp-veff-comparison}
    The diagram
    \begin{equation*}
        \begin{tikzcd}
            \HH^{\fr,\gp} \ar[d,"\Sigma^\infty_{\mathbf{T},\fr}"'] \ar[r,dashed] & \SH^\veff \ar[d,"\Sigma^\infty_{\mathbf{T}}"] \\
            \SH^\fr \ar[r,"\gamma_*"] & \SH
        \end{tikzcd}
    \end{equation*}
    of norm monoidal $\infty$-categories over $\AlgSpc$ can be completed to a commutative diagram.
    Moreover, each of the arrows is the left adjoint of a norm monoidal adjunction.
\end{proposition}

\begin{proof}
    For the first claim, it suffices to give a factorization objectwise.
    That is, we need to show that
    \begin{equation}
        \gamma_* \Sigma^\infty_{\mathbf{T},\fr} : \HH^\fr(S)^\gp \to \SH(S)
    \end{equation}
    factors through $\SH^\veff(S)$ for every algebraic space $S$.
    Since this functor preserves small colimits and $\HH^\fr(S)^\gp$ is generated under sifted colimits by objects of the form $h^\fr_S(X)^\gp = \gamma_!(X_+)^\gp$ for $X \in \Sm_S$, it suffices to show that $\gamma_* \Sigma^\infty_{\mathbf{T},\fr} \gamma_! (X_+)^\gp$ is very effective, but this is clear since
    \begin{align*}
        \gamma_* \Sigma^\infty_{\mathbf{T},\fr} \gamma_! (X_+)^\gp &\cong \gamma_* \Sigma^\infty_{\mathbf{T},\fr} \gamma_! (X_+) \\
        &\cong \gamma_* \gamma_! \Sigma^\infty_{\mathbf{T}}(X_+) \\
        &\cong \Sigma^\infty_{\mathbf{T}}(X_+).
    \end{align*}
    The second claim follows now from Proposition~\ref{prop:presentably-norm-adjunction-criterion}.
\end{proof}

\begin{theorem} \label{theorem:multiplicative-recognition}
    Let $k$ be a perfect field.
    There is an equivalence
    \begin{equation*}
        \NAlg (\SH^\veff|k) \cong \NAlg (\HH^{\fr,\gp}|k)
    \end{equation*}
    of $\infty$-categories of normed algebras over $k$
\end{theorem}

\begin{proof}
    Using the functor constructed in Proposition~\ref{prop:fr-gp-veff-comparison}, we can form a commutative diagram
    \begin{equation*}
        \begin{tikzcd}
            \NAlg(\HH^{\fr,\gp}|k) \ar[d] \ar[r] & \NAlg(\SH^\veff | k) \ar[d] \\
            \HH^\fr(k)^\gp \ar[r] & \SH^\veff(k)
        \end{tikzcd}
    \end{equation*}
    where the horizontal arrows are left adjoints, and the vertical arrows are the forgetful functors.
    The bottom horizontal arrow is an equivalence by \cite[Theorem 3.5.14]{EHKSY_MotivicInfiniteLoops}.
    Since the forgetful functors are conservative (combine Proposition~\ref{prop:nalg-forget-conservative} and Remark~\ref{remark:parametrized-calg}) and compatible with the adjunctions, the top horizontal arrow is also an equivalence.
\end{proof}

\begin{construction}
    Let $S$ be an algebraic space.
    We write
    \begin{equation*}
        \Sh^{\fr}_{\fet}(\Sm_S) \subseteq \PSh_\Sigma(\Corr^{\fr}(\Sm_S))
    \end{equation*}
    for the full subcategory spanned by $\Sigma$-presheaves $\mathscr{F} : \Corr^{\fr}(\Sm_S)^\op \to \Spc$ whose restriction to $\Sm_S$ is a sheaf for the finite \'etale topology.
    This is an accessible localization with left adjoint $\mathrm{L}_\fet : \PSh_\Sigma(\Corr^{\fr}(\Sm_S)) \to \Sh^{\fr}_{\fet}(\Sm_S)$.

    Since $\mathrm{L}_\fet$-equivalences are stable under base change and norm functors, the functors
    \begin{equation*}
        \mathrm{L}_\fet : \PSh_\Sigma(\Corr^{\fr}(\Sm_S)) \to \Sh^{\fr}_{\fet}(\Sm_S), \qquad S \in \AlgSpc
    \end{equation*}
    can be assembled into a norm monoidal functor
    \begin{equation*}
        \mathrm{L}_\fet : \PSh_\Sigma(\Corr^{\fr}(\Sm_\star)) \to \Sh^{\fr}_{\fet}(\Sm_\star)
    \end{equation*}
    between norm monoidal $\infty$-categories over $\AlgSpc$.
    Using Proposition~\ref{prop:presentably-norm-adjunction-criterion}, we find that $\mathrm{L}_\fet$ is a left adjoint in a norm monoidal adjunction.
    In particular, we can form the following commutative diagram
    \begin{equation}
        \begin{tikzcd} \label{eqn:nalg-by-etale-descent}
            \NAlg(\Sh_\fet^\fr(\Sm_\star)|\mathscr{B}) \ar[r] \ar[d] & \NAlg(\PSh^{\fr}_\Sigma (\Sm_\star)|\mathscr{B}) \ar[d] \\
            \CAlg(\Sh_\fet^\fr(\Sm_\star)|\mathscr{B}) \ar[r] & \CAlg(\PSh^{\fr}_\Sigma(\Sm_\star)|\mathscr{B})
        \end{tikzcd}
    \end{equation}
    for every algebraic space $S$ and every $\mathscr{B} \subseteq_\fet \AlgSpc_S$.
\end{construction}

\begin{proposition} \label{prop:nalg-by-descent}
    Let $S$ be an algebraic space.
    If $\mathscr{B} \subseteq_\fet \Sm_S$, then the left vertical arrow in (\ref{eqn:nalg-by-etale-descent}) induces an equivalence
    \begin{equation*}
        \NAlg(\Sh_\fet^\fr(\Sm_\star)|\mathscr{B}) \cong \CAlg(\Sh_\fet^\fr(\Sm_S)).
    \end{equation*}
\end{proposition}

\begin{proof}
    Unwinding definitions, we find that the horizontal arrows in (\ref{eqn:nalg-by-etale-descent}) are the inclusions of full subcategories of sections that satisfy finite \'etale descent in the sense of \cite[\S C.2]{BachmannHoyois_Norms}.
    We conclude using \cite[Corollary C.16]{BachmannHoyois_Norms} and Remark~\ref{remark:parametrized-calg}.
\end{proof}

\subsection{Normed Orientations}

\begin{construction} \label{construct:examples-of-nalgs}
    Let $S$ be an algebraic space.
    From Theorem~\ref{theorem:THE-DIAGRAM}, we can construct a commutative diagram
    \begin{equation*}
        \begin{tikzcd}
            \PSh_\Sigma^\fr(\Sm_\star) \ar[r] & \PSh_\Sigma^\orfsyn(\Sm_\star) \ar[r] \ar[d] & \PSh_\Sigma^\fsyn(\Sm_\star) \ar[d] \\
            & \PSh_\Sigma^\orfgor(\Sm_\star) \ar[r] & \PSh_\Sigma^\fflat(\Sm_\star)
        \end{tikzcd}
    \end{equation*}
    of norm monoidal $\infty$-categories over $\AlgSpc_S$.
    All of the arrows are left adjoints in norm monoidal adjunctions.
    Applying the various right adjoints to the various norm monoidal units at each vertex of the square, we obtain a commutative diagram
    \begin{equation} \label{eqn:examples-of-nalgs}
        \begin{tikzcd}
            \FSyn^{\oriented}_S \ar[r] \ar[d] & \FSyn_S \ar[d] \\
            \FGor^{\oriented}_S \ar[r] & \FFlat_S
        \end{tikzcd}
    \end{equation}
    in $\NAlg(\PSh_\Sigma^\fr(\Sm_\star)|\AlgSpc_S)$.
\end{construction}

\begin{proposition} \label{prop:examples-of-steady-nalgs}
    Let $S$ be an algebraic space.
    Each of
    \begin{equation*}
        \mathrm{L}_\mot \FFlat_S^\gp, \quad \mathrm{L}_\mot \FSyn_S^\gp, \quad \mathrm{L}_\mot \FGor^{\oriented,\gp}_S, \quad\text{and}\quad \mathrm{L}_\mot \FSyn^{\oriented,\gp}_S
    \end{equation*}
    is a steady normed algebra in $\HH^\fr$ over $\AlgSpc_S$.
\end{proposition}

\begin{proof}
    Each of these is a consequence of \cite[Lemma A.0.4]{EHKSY_ModulesOverMGL} and \cite[Lemma A.0.5]{EHKSY_ModulesOverMGL}.
    For $\mathrm{L}_\mot \FSyn_S^\gp$ and $\mathrm{L}_\mot \FSyn^{\oriented,\gp}_S$, we use the fact that the algebraic stack of (oriented) finite syntomic algebraic spaces is smooth with quasi-affine diagonal.
    For $\mathrm{L}_\mot \FFlat_S^\gp$ and $\mathrm{L}_\mot \FGor^{\oriented,\gp}_S$, we use \cite[Theorem 3.1]{HJNTY_HilbertSchemes} and \cite[Theorem 3.1]{HJNY_Hermitian}, together with the fact that the algebraic stack of vector bundles (equipped with a non-degenerate symmetric bilinear form) is smooth with quaso-affine diagonal.
\end{proof}

\begin{notation}
    Let $S$ be an algebraic space.
    We say $(*)_S$ holds if $\mathrm{f}_1(\mathrm{H}\mathbb{Z}_S) = 0$ and $\tau_{\geq 2 }\mathrm{HW}_S = 0$.
    Here, we write
    \begin{itemize}
        \item $\mathrm{H}\mathbb{Z}_S$ for Spitzweck's motivic cohomology spectrum,
        \item $\mathrm{H}\mathrm{W}$ for the periodic Witt cohomology spectrum,
        \item $\mathrm{f}_n$ for the $n$-effective cover functor for the slice filtration, and
        \item $\tau_{\geq n}$ for the $n$-connective cover functor for the homotopy $t$-structure.
    \end{itemize}
    See \cite{Bachmann_veffCovers} for further details.
\end{notation}

\begin{remark} \label{remark:assumptions-are-true}
    The condition $(*)_S$ is known to hold when $S$ is a scheme essentially smooth over a Dedekind scheme.
    See \cite[Remark 1.2]{Bachmann_veffCovers}.
    In fact, it will be shown in \cite{BEM_cdh} that $(*)_S$ holds for all schemes $S$.
    From there, the extension to all algebraic spaces is straightforward.
\end{remark}

\begin{theorem} \label{theorem:normed-orientations}
    Let $S$ be an algebraic space with $2 \in \mathcal{O}_S^\times$, and let $\mathscr{B} \subseteq_\fet \AlgSpc_S$.
    Suppose $(*)_X$ holds for all $X \in \mathscr{B}$.
    The commutative diagram
    \begin{equation*}
        \begin{tikzcd}
            \mathrm{MSL}_S \ar[r] \ar[d] & \mathrm{MGL}_S \ar[d] \\
            \mathrm{ko}_S \ar[r] & \mathrm{kgl}_S
        \end{tikzcd}
    \end{equation*}
    can be promoted to one in $\NAlg^\mathrm{stdy}(\SH|\mathscr{B})$.
\end{theorem}

\begin{proof}
    From Construction~\ref{construct:examples-of-nalgs} and Proposition~\ref{prop:examples-of-steady-nalgs}, we can construct a diagram 
    \begin{equation*}
        \begin{tikzcd}
            \Sigma^\infty_{\mathbf{T},\fr} \mathrm{L}_\mot \FSyn^{\oriented, \gp}_S \ar[r] \ar[d] & \Sigma^\infty_{\mathbf{T},\fr} \mathrm{L}_\mot \FSyn^\gp_S \ar[d] \\
            \Sigma^\infty_{\mathbf{T},\fr} \mathrm{L}_\mot \FGor^{\oriented, \gp}_S \ar[r] & \Sigma^\infty_{\mathbf{T},\fr} \mathrm{L}_\mot \FFlat^\gp_S
        \end{tikzcd}
    \end{equation*}
    of steady normed algebras in $\SH^\fr$ over $\mathscr{B}$.
    We now use \cite[Theorem 3.4.1]{EHKSY_ModulesOverMGL}, \cite[Theorem 3.4.3]{EHKSY_ModulesOverMGL}, \cite[Theorem 3.1]{HJNTY_HilbertSchemes} and \cite[Theorem 3.1]{HJNY_Hermitian}, and \cite[Theorem 1.1]{Bachmann_veffCovers} to identify the objects.
\end{proof}

\begin{remark}
    For $S$ and algebraic space, let $Q_S$ denote the commutative diagram (\ref{eqn:examples-of-nalgs}).
    The formation of $Q_S$ is stable by pullbacks in the following sense.
    Given any algebraic space $S$, we have that $Q_S$ is the restriction of $Q_{\Spec \mathbb{Z}}$ along $\Corr^\fet(\AlgSpc_S)^\op \to \Corr^\fet(\AlgSpc)^\op$.

    In particular, once we know that $(*)_S$ holds for every algebraic space $S$, then we'll know that the commutative diagram in Theorem~\ref{theorem:normed-orientations} is stable by pullbacks in the same sense.
\end{remark}

\begin{remark}
    If one is only interested in Theoem~\ref{theorem:normed-orientations} when $\mathscr{B} \subseteq_\fet \Sm_S$, then one can replace Construction~\ref{construct:examples-of-nalgs} with the following.
    One can directly construct (\ref{eqn:examples-of-nalgs}) as a diagraim in $\CAlg(\PSh^\fr_\Sigma(\Sm_S))$.
    Since each object is a sheaf for the finite \'etale topology, we can use Proposition~\ref{prop:nalg-by-descent} to promote the diagram to one of normed algebras over $\Sm_S$.
    
    The uniqueness statement in Proposition~\ref{prop:nalg-by-descent} implies that this produces the same result as taking Construction~\ref{construct:examples-of-nalgs} and restricting along $\Corr^\fet(\Sm_S)^\op \to \Corr^\fet(\AlgSpc)^\op$.
\end{remark}

\subsection{Normed Algebras over Cobordism and \texorpdfstring{$\mathrm{K}$}{K}-Theory Spectra}

\begin{theorem}
    There are equivalences
    \begin{equation*}
        \Mod_{\mathrm{MGL}}(\SH) \cong \SH^\fsyn \quad\text{and}\quad \Mod_{\mathrm{MSL}}(\SH) \cong \SH^\orfsyn
    \end{equation*}
    of norm monoidal $\infty$-categories over $\AlgSpc$.
\end{theorem}

\begin{proof}
    Combining \cite[Theorem 3.4.1]{EHKSY_ModulesOverMGL} and \cite[Theorem 3.4.3]{EHKSY_ModulesOverMGL} with \cite[Lemma 4.1.3]{EHKSY_ModulesOverMGL}, this is immediate from Proposition~\ref{prop:norm-adj-factorization} and Theorem~\ref{theorem:THE-DIAGRAM}.
\end{proof}

\begin{theorem} \label{theorem:normed-algebras-over-cobordism}
    Let $k$ be a perfect field.
    There are equivalences
    \begin{gather*}
        \NAlg_{\mathrm{MGL}}(\SH^\veff|k) \cong \NAlg(\HH^{\fsyn,\gp}|k) \qquad\text{and} \\ \NAlg_{\mathrm{MSL}}(\SH^\veff|k) \cong \NAlg(\HH^{\orfsyn,\gp}|k)
    \end{gather*}
    of $\infty$-categories of normed algebras.
\end{theorem}

\begin{proof}
    This is similar to the proof of Theorem~\ref{theorem:multiplicative-recognition}.

    Using Proposition~\ref{prop:norm-adj-factorization}, we get a norm monoidal adjunction
    \begin{equation*}
        \Mod_\mathscr{A}(\HH^{\fr,\gp}) \rightleftarrows \HH^{\fsyn,\gp},
    \end{equation*}
    where we write $\mathscr{A}$ for the image of the norm monoidal unit under the right adjoint.
    Using Construction~\ref{construct:nalg-adj}, Theorem~\ref{theorem:multiplicative-recognition}, and \cite[Theorem 3.4.1]{EHKSY_ModulesOverMGL}, we get a commutative diagram
    \begin{equation*}
        \begin{tikzcd}
            \NAlg_{\mathrm{MGL}}(\SH^\veff|k) \ar[r] \ar[d] & \NAlg(\HH^{\fsyn,\gp}| k) \ar[d] \\
            \Mod_\mathrm{MGL}(\SH^\veff(k)) \ar[r] & \HH^\fsyn(k)^\gp
        \end{tikzcd}
    \end{equation*}
    where the horizontal arrows are left adjoints, and the vertical arrows are the forgetful functors.
    Since the bottom arrow is an equivalence by \cite[Theorem 4.1.4]{EHKSY_ModulesOverMGL}, we find that the top horizontal arrow is the first desired equivalence.

    Replacing $\fsyn$ with $\orfsyn$, \cite[Theorem 3.4.1]{EHKSY_ModulesOverMGL} with \cite[Theorem 3.4.3]{EHKSY_ModulesOverMGL}, and \cite[Theorem 4.1.4]{EHKSY_ModulesOverMGL} with \cite[Theorem 4.2.2]{EHKSY_ModulesOverMGL} gives the second desired equivalence.
\end{proof}

\begin{notation}
    Let $\mathscr{D}$ be an additive $\infty$-category.
    For a positive integer $n$, we write $\mathscr{D}[1/n] \subseteq \mathscr{D}$ for the full subcategory spanned by objects $X \in \mathscr{D}$ such that $X \xrightarrow{n} X$ is an equivalence.
\end{notation}

\begin{notation}
    Let $S$ be an algebraic space, let $\mathscr{B} \subseteq_\fet \AlgSpc_S$ with $\mathscr{B} \subseteq \Sm_S$, and let $\mathscr{D}$ be a norm monoidal $\infty$-category over $\mathscr{B}$ such that $\mathscr{D}(X)$ is additive for every $X \in \mathscr{B}$.
    For a positive integer $n$, we write
    \begin{equation*}
        \NAlg(\mathscr{D}|S)[1/n] \subseteq \NAlg(\mathscr{D}|S)
    \end{equation*}
    for the full subcategory spanned by $\mathscr{A} \in \NAlg(\mathscr{D}|S)$ such that $\mathscr{A}_S$ is in $\mathscr{D}(S)[1/n]$.
\end{notation}

\begin{theorem} \label{theorem:normed-algebras-over-k-theory}
    Let $k$ be a field of exponential characteristic $e$.
    There is an equivalence
    \begin{equation*}
        \NAlg_{\mathrm{kgl}}(\SH|k)[1/e] \cong \NAlg(\SH^\fflat|k)[1/e].
    \end{equation*}
    If $e \neq 2$, then we also have an equivalence
    \begin{equation*}
        \NAlg_{\mathrm{ko}}(\SH|k)[1/e] \cong \NAlg(\SH^\orfgor|k)[1/e].
    \end{equation*}
\end{theorem}

\begin{proof}
    The proof of Theorem~\ref{theorem:normed-algebras-over-cobordism} goes through essentially unchanged.
    One just needs to replace \cite[Theorem 3.4.1]{EHKSY_ModulesOverMGL} with either \cite[Theorem 5.4]{HJNTY_HilbertSchemes} or \cite[Theorem 7.12]{HJNY_Hermitian}, and \cite[Theorem 4.1.4]{EHKSY_ModulesOverMGL} with either \cite[Corollary 4.3]{Bachmann_Cancellation} or \cite[Theorem 9.2]{HJNY_Hermitian},
\end{proof}

\bibliographystyle{amsalpha}
\bibliography{biblio}

\end{document}